\documentclass[11pt,reqno]{amsart}

\usepackage{cite}
\usepackage[T1]{fontenc}
\usepackage{enumitem}
\usepackage{hyperref}
\hypersetup{
	colorlinks   = true,
	citecolor    = blue,
	linkcolor    = blue
}

\usepackage{amsmath,amsfonts,amsthm,color,tikz, comment, xcolor, xfrac,amssymb}
\usepackage{algpseudocode} 
\usepackage{algorithm}
\usepackage{booktabs}


  
\usepackage[normalem]{ulem}
\usepackage{bbm}
\usepackage{mathtools}
\usepackage{shuffle}
\usepackage{multirow}
\usepackage{pdfsync}
\usepackage[font={scriptsize}]{caption}
\usepackage{environ}

\usepackage[margin=1in]{geometry}
\newcommand{\E}{\mathbb E}
\newcommand{\R}{\mathbb R}
\newcommand{\N}{\mathbb N}
\newcommand{\PP}{\mathbb P}

\newcommand{\Z}{\mathbb Z}
\newcommand{\Y}{\mathbb{Y}}
\newcommand{\X}{\mathbb{X}}
\newcommand{\1}{\mathbbm 1}

\newcommand{\cf}{\mathcal F}

\newcommand{\cs}{\mathcal S}
\newcommand{\ct}{\mathcal T}

\newcommand{\al}{\alpha}

\newcommand{\ep}{\varepsilon}

\newtheorem{theorem}{Theorem}[section]

\newtheorem{corollary}[theorem]{Corollary}

\newtheorem{definition}[theorem]{Definition}

\newtheorem{assumption}[theorem]{Assumption}
\newtheorem{lemma}[theorem]{Lemma}

\newtheorem{proposition}[theorem]{Proposition}

\theoremstyle{remark}
\newtheorem{remark}[theorem]{Remark}

\theoremstyle{remark}

\theoremstyle{definition}
\newtheorem{example}[theorem]{Example}

\newcommand{\bean}{\begin{eqnarray*}}
	\newcommand{\eean}{\end{eqnarray*}}
\newcommand{\ben}{\begin{enumerate}}
	\newcommand{\een}{\end{enumerate}}
\newcommand{\beq}{\begin{equation}}
	\newcommand{\eeq}{\end{equation}}

\title[quickest detection with signatures]{Quickest Detection with Rough Path Signatures}

\author[MINGRUI WANG]{MINGRUI WANG}
\address{(M. Wang) Harold and Inge Marcus Department of Industrial and Manufacturing
	Engineering, The Pennsylvania State University, University Park, PA 16802 United States\\
}
\email{mvw5822@psu.edu}

\author[PRAKASH CHAKRABORTY]{PRAKASH CHAKRABORTY}
\address{(P. Chakraborty) Harold and Inge Marcus Department of Industrial and Manufacturing
	Engineering, The Pennsylvania State University, University Park, PA 16802 United States\\
}
\email{prakashc@psu.edu}

\thanks{P. Chakraborty acknowledge support from the National Science Foundation under grant DMS-2153915, DARPA HR0011-25-3-E012 and Korb Early Career Professorship.}

\begin{document}
		
\begin{abstract}
We develop a framework for quickest detection of distributional  change in signals modeled as rough paths. By representing the pre-change and post-change dynamics as two independent rough paths and modeling the observed signal as their concatenation at an unknown random change-point, we formulate quickest detection as an optimal stopping problem using rough paths. We show that the optimal stopping rule takes the form of the first hitting time to a half-space by a linear functional of the rough path signature, and establish that the same structural form arises independently from the intrinsic geometry of the observed path. Statistical guarantees on detection delay and false alarm probability are derived, and the framework is extended to a distributionally robust formulation in which both the pre-change and post-change models are uncertain. The proposed rules are implemented via zeroth-order stochastic approximation over truncated-signature coefficients and evaluated numerically under Brownian and fractional Brownian dynamics, achieving performance comparable to optimal methods in the Brownian setting while outperforming them and remaining robust to adversarial path perturbations in the fractional Brownian setting.
\end{abstract}
\maketitle
\section{Introduction}

\subsection{Classical quickest detection}
The quickest detection problem concerns the real-time detection of a change in the law of an observed process, with the objective of raising an alarm as quickly as possible while keeping false alarms under control. It is a central topic in sequential analysis, with classical applications in quality control, surveillance, signal processing, finance, and reliability \cite{Shiryaev2010,VeeravalliBanerjee2014,TartakovskyNikiforovBasseville2014,shiryaev2002quickest,Lai1995}. In its basic form, one observes a stochastic process whose distribution changes at an unknown disorder time $\theta$, and seeks a stopping time $\tau$, adapted to the observation filtration, that balances early detection against false alarms. Both Bayesian and non-Bayesian formulations have been studied extensively, and classical procedures include the Shiryaev, CUSUM, and Shiryaev-Roberts rules \cite{Page1954,Pollak1985,Moustakides1986,Shiryaev1996,Shiryaev2010}. A standard continuous time model is the Brownian disorder problem,
$$
dX_t = r \1_{\{t \ge \theta\}}\,dt + \sigma\,dB_t,
$$
in which the observation is pure noise before the disorder and acquires a drift after the change. The task is to infer the disorder time $\theta$ from the sample path of $X$. This model occupies a central place in the continuous time theory of quickest detection and links the subject naturally with stochastic control and optimal stopping \cite{Shiryaev1978,Shiryaev2010,PeskirShiryaev2006}.

Classical quickest detection admits two main formulations. In the Bayesian setting, the disorder time $\theta$ is random, and one typically minimizes a Bayes risk such as
$$
\inf_{\tau}\left\{\mathbb{P}(\tau<\theta)+c\,\mathbb{E}\big[(\tau-\theta)^+\big]\right\}.
$$
One may also consider constrained Bayesian formulations in which the detection delay is minimized subject to an upper bound on the false alarm probability. In the classical Bayesian model, the Shiryaev procedure is optimal \cite{Shiryaev1963,TartakovskyVeeravalli2005}.

In the non-Bayesian setting, the change-point is treated as an unknown deterministic time. Detection delay is measured in a worst case sense, while false alarms are controlled through quantities such as the mean time to false alarm. A classical benchmark is Lorden's criterion,
$$
\inf_{\tau}\ \sup_{t}\ \operatorname*{ess\,sup}\,\mathbb{E}_t\!\left[(\tau-t)^+\mid \mathcal{F}_{t}\right]\qquad\text{subject to}\qquad\mathbb{E}_{\infty}[\tau]\ge \gamma.
$$
Within this minimax framework, the CUSUM procedure is optimal for Lorden's problem \cite{Lorden1971,Moustakides1986,Shiryaev1996}. Closely related formulations, such as Pollak's criterion, lead to the Shiryaev-Roberts procedure and its variants \cite{Pollak1985,VeeravalliBanerjee2014}.

A defining feature of the classical theory is that the original path-space problem can often be reduced to a low dimensional statistic. In Bayesian diffusion models, one works with posterior or odds-ratio processes; in likelihood-ratio based minimax procedures, one recursively updates a statistic built from the likelihood ratio. Consequently, the quickest detection problem becomes an optimal stopping problem for a tractable, Markovian state process. This dimension reduction is one of the primary reasons for the success and mathematical elegance of the classical theory.  

\subsection{Challenges of non-Markovianity and path irregularity}
Many modern signals of practical interest, however, evade description by classical Markovian or semimartingale models. In practice, the observed dynamics may heavily depend on the past, exhibit long memory, or be driven by highly oscillatory and irregular inputs \cite{Mandelbrot1968,Gatheral2018,Bender2011,Bennedsen2021}. In such situations, the current state $X_t$ no longer summarizes all relevant information, and a finite dimensional sufficient statistic may no longer be available. The continuation value naturally becomes a functional of the entire historical path $X|_{[0,t]}$, rather than the current state alone. This introduces the precise analytical difficulties inherent to non-Markovian optimal stopping.

There is a second difficulty of an analytical nature. Much of classical stochastic modeling relies on integration theories suited either to sufficiently regular paths or to semimartingales. However, when the observed path is highly irregular and falls outside the semimartingale class, classical integration breaks down, and even the formulation of a differential equation driven by the path becomes problematic. 

Rough path theory \cite{Lyons1998,LyonsQian2002,Lyons2007,friz-victoir,Friz2020,RoughPathsNote} addresses these difficulties  simultaneously by providing a deterministic, pathwise integration theory for irregular signals and, through the signature, a universal feature map of the path history that serves as a substitute for the 
finite-dimensional sufficient statistics of the 
classical theory.

\subsection{Rough path  signatures for quickest detection}

To execute quickest detection in complex, modern signal environments, one requires a probabilistically flexible and analytically rigorous framework for non-Markovian, nonstationary dynamics driven by irregular paths.
Rough path theory and signature methods provides such a framework. Elsewhere, its applications span rough stochastic optimal control \cite{DiehlFrizGassiat2017,chakraborty2024,ashkarian2026}, asset pricing under rough volatility \cite{BayerFrizGatheral2016}, as universal non-Markovian feature extractors in statistical learning \cite{Chevyrev2025,KidgerEtAl2020}, and as robust covariates in functional linear regression \cite{Fermanian2022, Fermanian2021}.
In addition, recent works on non-Markovian optimal stopping demonstrate that intractable, path-dependent continuation values can be efficiently approximated by threshold rules on linear functionals of the truncated rough path signature \cite{Bayer2023,BayerPelizzariSchoenmakers2025}.
Since quickest detection is fundamentally an optimal stopping problem, we propose replacing classical likelihood-ratio statistics with such stopping rules built from signature features, namely rules of the form
\begin{equation*}%\label{eq:tau-intro}
\tau_{l}=\inf\Big\{t\ge 0:\big|\langle l,\mathbb{X}^{< \infty}_{0,t}\rangle\big|\ge 1\Big\}\wedge T.
\end{equation*}
Unlike classical procedures such as CUSUM and Shiryaev, whose sufficient statistics are derived under explicit parametric and Markov assumptions, the proposed rule encodes the entire observed history through the rough path signature. This flexible alternative to the classical finite-dimensional sufficient statistics motivates our formulation of quickest detection on rough-path space.

\subsection{Adversarial (robust) quickest detection}
In modern machine learning, models are notoriously vulnerable to adversarial attacks, where small, targeted perturbations can induce severe failures \cite{kurakin2016,madry2017towards,Pelekis2025}. This fragility extends naturally to sequential decision-making: because true data-generating distributions are rarely known exactly, adversarial corruption and model misspecification can severely degrade quickest detection performance. 

Robust quickest detection traditionally addresses this issue by optimizing worst-case performance over prescribed uncertainty classes, either through least favorable distributions in the classical minimax framework \cite{Unnikrishnan2011} or through data-driven approaches based on Wasserstein ambiguity sets and score-based divergences \cite{VeeravalliBanerjee2014,Wu2023,Xie2024,Moushegian2024}. In this paper, we adopt a minimax formulation over uncertainty classes defined intrinsically on rough-path space and propose a signature-based method, accommodating adversarial perturbations under irregular and non-Markovian noise.

\subsection{Contribution} 
Motivated by these challenges, we develop a novel framework for quickest detection on rough path space. The main contributions are as follows.
\begin{enumerate}[wide, labelwidth=!, labelindent=0pt, label=(\roman*)]
\item \emph{Optimal stopping on rough path space.} 
We model the observed signal as the concatenation of two geometric $p$-rough paths at an unknown change-point and explore its properties (Lemma~\ref{lem:geo-rp}-\ref{lem:d-p-var}).  %formulate quickest detection as an optimal stopping problem.
We use the signature half-space hitting time as the optimal stopping rule (Proposition~\ref{thm:signature-stopping}),  extending the classical theory to non-Markovian and path-irregular settings.

\item \emph{Geometric detector.} %Independently of the payoff structure, 
We construct a detector from the intrinsic geometry of the observed path, % via the homogeneous $p$-variation distance. 
and show that this detector takes the same structural form as the optimal stopping rule. We derive explicit statistical guarantees on detection delay and false alarm probability 
(Propositions~\ref{prop:l1-X-de-ep}-\ref{prop:stat_1}).

\item \emph{Repeated experiments.} When independent replications share a common change point, we show that aggregation reduces both detection delay and false alarm probability exponentially in the number of replications (Proposition~\ref{prop:stat_2}).

\item \emph{Adversarial or distributionally robust detection.} 
We formulate a minimax extension over uncertainty classes on rough-path space, including Wasserstein ambiguity sets and total-variation budgets for adversarial perturbations (Section~\ref{sec:adv}).

\item \emph{Numerical implementation.} 
We implement the proposed stopping rules by applying zeroth-order optimization to truncated-signature coefficients and evaluate them under Brownian and fractional Brownian dynamics. The signature-based rules perform comparably to CUSUM and Shiryaev in the Brownian setting, outperform them in the fractional Brownian setting, and exhibit robustness to adversarial path perturbations (Section~\ref{sec:numeric}).
\end{enumerate}

\subsection{Structure of the paper}

Section~\ref{sec:preliminaries} reviews the necessary rough path preliminaries. Section~\ref{sec:QD-rough} formulates quickest detection on rough path space, develops its properties and presents Brownian and fractional Brownian examples. Section~\ref{sec:detector} develops the geometric detector and derives its statistical guarantees, while Section~\ref{sec:adv} treats adversarial quickest detection. Section~\ref{sec:numeric} presents the numerical implementation and experiments, and Section~\ref{sec:conclusion} concludes.

\section{Preliminaries}\label{sec:preliminaries}

The signature of a path takes values in an algebraic structure called the tensor algebra. We provide details in Appendix~\ref{sec:tensor} together with the associated Lie algebra, which captures the antisymmetric part of iterated integrals and underlies the group structure of signatures. Appendix~\ref{sec:shuffle} provides notions of shuffle product and group-like elements, which are needed to evaluate signatures against linear functionals and to exploit their multiplicative structure. In this section, we define geometric $p$-rough paths and their signatures, and state a density result for signature-based functionals that is the key technical tool behind our main results.

\subsection{Signatures and geometric $p$-rough paths}\label{sec:geometric-p}
We now define the signature of a path and extend the construction to the rough path setting, where the driving signal is too irregular for classical integration.

Let $N\in\N$. For a continuous path $x:[0,T]\to\R^d$ of bounded variation and $(s,t)\in\Delta_T:=\{(s,t)\in[0,T]^2:s\le t\}$, the \emph{truncated signature} of order $N$ of $x$ over $[s,t]$ is defined by
$$
S_N(x)_{s,t}:=\left(1,\int_s^t dx_u,\int_{s<u_1<u_2<t} dx_{u_1}\otimes dx_{u_2},
\dots,\int_{s<u_1<\cdots<u_N<t} dx_{u_1}\otimes\cdots\otimes dx_{u_N}\right).
$$
Equivalently, for each $k=1,\dots,N$,
$$
\pi_k\bigl(S_N(x)_{s,t}\bigr)=\int_{s<u_1<\cdots<u_k<t} dx_{u_1}\otimes\cdots\otimes dx_{u_k}\in (\R^d)^{\otimes k}.
$$
The \emph{full signature} is defined by
$$
S(x)_{s,t}:=\bigl(1,\pi_1(S(x)_{s,t}),\pi_2(S(x)_{s,t}),\dots\bigr).
$$
For paths of bounded variation, the signature is multiplicative: for all $0\le s\le u\le t\le T$,
$$
S_N(x)_{s,t}=S_N(x)_{s,u}\otimes S_N(x)_{u,t}.
$$
In particular, $S_N(x)_{s,t}\in G^N(\R^d)$ for every $(s,t)\in\Delta_T$ (cf.\ Appendix~\ref{sec:shuffle} for $G^N(\R^d)$).
This multiplicativity, known as Chen's identity, is the fundamental concatenation property of signatures and will be used repeatedly in
the sequel. In particular, Chen's identity shows that the signature accumulates information about the path in a multiplicative, path-consistent
manner, a structure we will exploit directly when concatenating pre- and post-change rough paths in Section~\ref{sec:QD-rough}.

Before introducing geometric rough paths, we recall the metric structure on $G^N(\R^d)$. For $g=(1,g_1,\dots,g_N)\in G^N(\R^d)$ and $\lambda\in\R$, define the dilation map by
$$
	\delta_\lambda(g):=(1,\lambda g_1,\dots,\lambda^N g_N).
$$
A \emph{homogeneous norm} on $G^N(\R^d)$ is a continuous map
$$
	\|\cdot\|:G^N(\R^d)\to [0,\infty)
$$
such that
\begin{enumerate}
		\item $\|g\|=0$ if and only if $g=\mathbf{1}$,
		\item $\|\delta_\lambda g\|=|\lambda|\,\|g\|$ for all $\lambda\in\R$ and $g\in G^N(\R^d)$.
\end{enumerate}
A homogeneous norm is called \emph{symmetric} if $\|g\|=\|g^{-1}\|$ for all $g\in G^N(\R^d)$, and \emph{subadditive} if
$$
	\|g\otimes h\|\le \|g\|+\|h\|,\qquad g,h\in G^N(\R^d).
$$
\begin{definition}
The \emph{Carnot--Carath\'eodory norm}, denoted by $\|\cdot\|_{CC}$, is defined for $g\in G^N(\R^d)$ by
	$$
	\|g\|_{CC}:=\inf\left\{\int_0^1 |d\gamma|:\;\gamma\in C^{1\text{-var}}([0,1],\R^d)\text{ and }S_N(\gamma)_{0,1}=g\right\},
	$$
	where $C^{1\text{-var}}([0,1],\R^d)$ denotes the space of continuous bounded variation paths.
\end{definition}
By \cite[Theorem 7.32]{friz-victoir}, the Carnot--Carath\'eodory norm is well defined. Moreover, by \cite[Proposition 7.40]{friz-victoir}, it is a symmetric, subadditive, homogeneous norm. This induces the \emph{Carnot--Carath\'eodory metric}
$$
d_{CC}(g,h):=\|g^{-1}\otimes h\|_{CC},
\qquad g,h\in G^N(\R^d).
$$
By \cite[Proposition 7.36]{friz-victoir}, the metric $d_{CC}$ is continuous and left-invariant, that is,
$$
d_{CC}(g\otimes h,g\otimes k)=d_{CC}(h,k),
\qquad g,h,k\in G^N(\R^d).
$$

We now extend this framework to rough paths of finite $p$-variation. Recall the simplex $\Delta_T$ defined above. Let $\X:[0,T]\to G^N(\R^d)$ be a continuous path with $\X_0=\mathbf{1}$. For $0\le s\le t\le T$, define its increment by
$$
\X_{s,t}:=\X_s^{-1}\otimes \X_t.
$$
Then $\X_{t,t}=\mathbf{1},\,\X_{s,u}\otimes \X_{u,t}=\X_{s,t}$ for $0\le s\le u\le t\le T$. That is, $\X$ canonically induces a multiplicative map on $\Delta_T$. Following \cite[Definition 8.1]{friz-victoir}, we now introduce the homogeneous $p$-variation and $1/p$-H\"older distances.

\begin{definition}\label{def:d-p-var-CC}
Let $\X,\Y:[0,T]\to G^N(\R^d)$ be continuous paths with $\X_0=\Y_0=\mathbf{1}$. For $0\le s<t\le T$, define
$$
d_{p\text{-var};[s,t]}(\X,\Y):=\left(\sup_{\mathcal{D}\subset[s,t]}\sum_i d_{CC}\bigl(\X_{t_i,t_{i+1}},\Y_{t_i,t_{i+1}}\bigr)^p\right)^{1/p},
$$
where the supremum is taken over all partitions	$\mathcal{D}=\{s=t_0<\cdots<t_m=t\}$ of $[s,t]$. We call $d_{p\text{-var};[s,t]}$ the \emph{homogeneous $p$-variation distance}. We also define the \emph{homogeneous $p$-variation} of $\X$ by
$$
\|\X\|_{p\text{-var};[0,T]}:=\left(\sup_{\mathcal D\subset [0,T]}\sum_i \|\X_{t_i,t_{i+1}}\|_{CC}^p\right)^{1/p}.
$$
Similarly, the \emph{homogeneous $1/p$-H\"older distance} between $\X$ and $\Y$ is defined by
$$
d_{1/p\text{-H\"ol};[0,T]}(\X,\Y):=\sup_{0\le s<t\le T}\frac{d_{CC}\bigl(\X_{s,t},\Y_{s,t}\bigr)}{|t-s|^{1/p}},
$$
and the \emph{homogeneous $1/p$-H\"older norm} of $\X$ is defined by
$$
\|\X\|_{1/p\text{-H\"ol};[0,T]}:=\sup_{0\le s<t\le T}\frac{\|\X_{s,t}\|_{CC}}{|t-s|^{1/p}}.
$$
\end{definition}

We now define geometric $p$-rough paths.

\begin{definition}
	Let $p\ge1$ and set $N:=\lfloor p\rfloor$. A \emph{geometric $p$-rough path} is a continuous path
	$$
	\X:[0,T]\to G^N(\R^d)
	$$
	with $\X_0=\mathbf{1}$ and finite homogeneous $p$-variation such that there exists a sequence of bounded variation paths $(x^m)_{m\ge1}$ in $\R^d$ satisfying
	$$
	d_{p\text{-var};[0,T]}\bigl(S_N(x^m),\X\bigr)\to0\qquad\text{as }m\to\infty.
	$$
\end{definition}
The space of geometric $p$-rough paths is denoted by $\Omega^p_T$ (also written as $C^{p\text{-var}}([0,T],G^N(\R^d))$ in some references). For a geometric $p$-rough path $\X$, we write
$$
\X_{s,t}=\bigl(1,X_{s,t},\pi_2(\X_{s,t}),\dots,\pi_N(\X_{s,t})\bigr),
$$
and refer to $X$ as the path trajectory (first level) of $\X$, where $X_{s,t}=\pi_1(\X_s^{-1}\otimes\X_t)=X_t-X_s$ is the increment of the trajectory.  

By Lyons' extension theorem \cite[Theorem 3.7]{Lyons2007}, every geometric $p$-rough path
$\X\in \Omega_T^p$ has a unique lift $\X^{<\infty}$ with values in
$G\bigl((\R^d)\bigr)$ (cf.\ Appendix~\ref{sec:shuffle} for $G(\R^d)$), such that for each $N\ge 1$,
$$
\|\pi_{\le N}(\X^{<\infty})\|_{p\text{-var};[0,T]}<\infty,
\qquad\text{and}\qquad
\pi_{\le \lfloor p\rfloor}(\X^{<\infty})=\X.
$$
This lifted path is called the \emph{signature} of $\X$. We adopt the shorthand
$$
\X^{<\infty}_{0,t}:=\X^{<\infty}\big|_{[0,t]},
\qquad
\X^{\le N}:=\pi_{\le N}(\X^{<\infty}),
$$
and refer to $\X^{\le N}$ as the \emph{truncated signature} of order $N$ of $\X$.
Note that for a bounded-variation path $x$ the signature $S(x)_{0,t}$ coincides with
$\X^{<\infty}_{0,t}$ when $\X = S_N(x)$ is viewed as the canonical rough path lift. We define $\widehat{\Omega}_T^p$ to be the closure, in the homogeneous $p$-variation distance, of the rough path lifts $\widehat{\X}^{\le \lfloor p\rfloor}$ corresponding to bounded variation paths
$$
\widehat{X}_t=(t,X_t)\in\R^{1+d}.
$$
Time augmentation ensures that the signature encodes not only the geometric trajectory but also the timing of events, which is essential for the optimal stopping problems studied in Sections~\ref{sec:QD-rough}--\ref{sec:detector}. An element of $\widehat{\Omega}_T^p$ is called a \emph{time-augmented geometric $p$-rough path} and is denoted $\widehat{\X}$. By Lyons' extension theorem, $\widehat{\X}$ admits a unique full signature $\widehat{\X}^{<\infty}_{0,t}\in G((\R^{1+d}))$. 

\subsection{Approximation by signature functionals}\label{sec:approx}
The key technical tool underlying our main results is a density property for signature-based functionals: any continuous stopping policy can be approximated, uniformly on a compact set of probability one, by a linear functional of the signature. This approximation theorem allows us to reduce the optimal stopping problem over all adapted stopping times to the much more tractable class of signature hitting times.

\begin{definition}
	We set $\Lambda_T:=\bigcup_{t \in [0,T]} \hat{{\Omega}}_t^p$ and $\mathcal{T} \coloneqq C(\Lambda_T,\R)$. We further define the space $\mathcal{T}_{\mathrm{sig}} \subset \mathcal{T}$ as
	\begin{align*}
		\mathcal{T}_{\mathrm{sig}} = \left\{ \phi \in \mathcal{T}\, :\, \exists l \in T((\R^{1+d})^*) \text{ such that } \phi(\widehat{\X}|_{[0,t]}) = \langle l,\widehat{\X}^{< \infty}_{0,t} \rangle \ \text{ for all } \widehat{\X}|_{[0,t]} \in\Lambda_T \right\}.
	\end{align*}
\end{definition}

\begin{remark}
Note that every $l \in T((\R^{1+d})^*)$ defines a $\phi_l \in \mathcal{T}$ by setting $\phi_l(\widehat{\X}|_{[0,t]}) \coloneqq \langle l,\widehat{\X}^{< \infty}_{0,t} \rangle$.
\end{remark}

As a consequence of the Stone--Weierstrass theorem, $\mathcal{T}_{\mathrm{sig}}$ is dense in $\mathcal{T}$. The following strengthening, which provides uniform approximation on a set of high probability, can be found in \cite[Lemma~5.2]{Bayer2023} or \cite[Lemma~B.3]{doi:10.1137/19M1259778}.
\begin{lemma}\label{lem:sig_dense}
	Let $\PP$ be a probability measure on $(\hat{\Omega}_T^p, \mathcal{B}(\hat{\Omega}_T^p))$. Then, for every $\varepsilon > 0$, there is a compact set $\mathcal{K} \subset \hat{\Omega}_T^p$ such that
	\begin{enumerate}
		\item $\PP(\mathcal{K}) > 1 - \varepsilon$,
		\item $\mathcal{T}_{\mathrm{sig}}$, restricted to $\mathcal{K}$, is dense in $\mathcal{T}$. More precisely, for every $\phi \in \mathcal{T}$ there is a sequence $\phi_n \in \mathcal{T}_{\mathrm{sig}}$ such that
		\begin{align*}
			\sup_{\widehat{\X} \in \mathcal{K};\ t \in [0,T]} |\phi_n(\widehat{\X}|_{[0,t]}) - \phi(\widehat{\X}|_{[0,t]}) | \to 0
		\end{align*}
		as $n \to \infty$.
	\end{enumerate}
\end{lemma}
Lemma~\ref{lem:sig_dense} will be applied in the sequel to approximate the value functionals of the optimal stopping problems, which depend on the entire observed path history rather than the current state alone, by linear functionals of the signature, 
thereby reducing these intractable functional optimization problems to finite-dimensional ones.

\section{Quickest detection with rough path signatures}\label{sec:QD-rough}
We now formalize the quickest detection problem on rough path space. The key modeling choice is to represent the pre-change and post-change dynamics as two independent geometric $p$-rough paths, and to construct the observed signal as their concatenation at the random change-point.
\subsection{Problem setup}
Recall the space $\Omega_T^p$ of geometric $p$-rough paths introduced in Section~\ref{sec:geometric-p}.
Then consider two rough stochastic processes $\X^{(1)}, \X^{(2)} \in \Omega_T^p$ with distinct laws given by $\mu_1$ and $\mu_2$ respectively. These two processes model respectively the pre and post change dynamics of our observation process. We model the transition between the two regimes by a non-negative random variable $\theta$ representing the change-point. Let $\theta$ follow the distribution $\nu$ with finite mean.
   	
We formalize our setup by constructing an appropriate product probability space. Let the sample space be $\bar{\Omega} = \Omega_T^p \times \Omega_T^p \times [0,\infty)$, with a generic element denoted by $\omega = (\mathbf{x}^{(1)}, \mathbf{x}^{(2)}, u)$. We define the probability measure $\PP$ on this space as
$$
   	\PP = \mu_1 \otimes \mu_2 \otimes \nu.
$$
By construction, under $\PP$, the processes $\X^{(1)}$ and $\X^{(2)}$ and the random variable $\theta$ are mutually independent. For completeness, we denote the background filtered probability space $(\bar{\Omega}, \mathcal{F}_t, \mathbb{P})$, where $\mathcal{F}_t := \sigma(\theta, X_{s}^i : i=1,2 \text{ and } 0 \leq s \leq t)$ to be the full filtration generated by both the latent change-point and the individual paths. 

To ensure detection and for their study using rough path formalism, we assume that the pre- and post-change processes $\X^{(1)}$ and $\X^{(2)}$ satisfy the following assumption.
\begin{assumption}\label{asm:distinct-paths}
     The stochastic processes $\X^{(1)},\X^{(2)} \in \Omega_T^p$. Furthermore, there exist $q>1,M>0$ and an increasing function $g:\R \rightarrow \R$ with $g(0)=0$ such that for any $0\leq s < t \leq T$
     \begin{align}
        &\max \left\{\E\left[\|\X^{(2)}\|_{1/p-\text{H\"{o}l};[0,T]}^q\right] , \E\left[\|\X^{(1)}\|_{1/p-\text{H\"{o}l};[0,T]}^q\right]\right\}\leq M,\quad \text{and}\label{eq:X12-q-moment}\\
        &\E\left(d_{p-\mathrm{var};[s,t]}\left(\X^{(2)},\X^{(1)}\right)\right)\geq g(|t-s|),\label{eq:X12>g}
 \end{align}
where the notations are as in Definition~\ref{def:d-p-var-CC}.
\end{assumption}	
\begin{remark}
\cite[Thm A.10]{friz-victoir} provides a class of stochastic rough paths that satisfy the moment bound assumption in \eqref{eq:X12-q-moment}.
\end{remark}
\begin{remark}
Equation \eqref{eq:X12>g} captures the dissimilarity of the two paths which is necessary to reliably distinguish between the paths. In our setting, the divergence of the two paths is captured by the homogeneous $p$-variation distance which is assumed to accumulate over time owing to \eqref{eq:X12>g} and the growth of the function $g$. This provides additional information to detect the change-point with growing time post changepoint, and is equivalent to the so called 'signal-to-noise ratio' in classical quickest detection problems.
\end{remark}

\begin{figure}[h]   
\centering

\begin{tikzpicture}[scale=0.8]
	
	% Define colors
	\definecolor{myred}{RGB}{228,26,28}
	\definecolor{myblue}{RGB}{55,126,184}
	\definecolor{mypurple}{RGB}{152,78,163} 
	% ==========================================
	% LEFT PANEL: Both full paths defined on [0, T]
	% ==========================================
	\begin{scope}[xshift=0cm]
		% Grid
		\draw[step=1cm,gray!30,very thin] (-0.5,-0.5) grid (4.5,3.5);
		
		% Axes
		\draw[->, >=stealth] (-0.5,0) -- (4.5,0) node[right] {$x_1$};
		\draw[->, >=stealth] (0,-0.5) -- (0,3.5) node[above] {$x_2$};
		
		% Ticks
		\foreach \x in {1,2,3,4} \draw (\x,-0.1) -- (\x,0.1) node[below=0.15cm] {\small $\x$};
		\foreach \y in {1,2,3} \draw (-0.1,\y) -- (0.1,\y) node[left=0.15cm] {\small $\y$};
		
		% Path 1: X^(1) Full [0, T]
		\draw[myred, thick] (0,0) .. controls (0.5, 1.5) .. (2,1) node[midway, below right] {$\X^{(1)}_{0,\theta}$};
		\draw[myred, thick] (2,1) .. controls (2.5, 0.5) .. (3, 0.5) node[right] {$\X^{(1)}_{\theta,T}$};
		
		% Junction point \theta for X^(1)
		\filldraw[myred] (2,1) circle (1.5pt) node[above right] {$\X^{(1)}_\theta$};
		
		% Path 2: X^(2) Full [0, T]
		\draw[myblue, thick] (0,1.5) .. controls (0.5, 1.5) .. (1,2.5) node[below left=0.3cm,left] {$\X^{(2)}_{0,\theta}$};
		\draw[myblue, thick] (1,2.5) .. controls (1.5, 3.5) .. (2.5, 3.5) node[right] {$\X^{(2)}_{\theta,T}$};
		
		% Junction point \theta for X^(2)
		\filldraw[myblue] (1,2.5) circle (1.5pt) node[right] {$\X^{(2)}_\theta$};
		
		% Panel Label
		\node at (2.3, -1) {The hidden paths $\X^{(1)}$ and $\X^{(2)}$};
	\end{scope}
	
	% ==========================================
	% RIGHT PANEL: Concatenated Path (\X)
	% ==========================================
	\begin{scope}[xshift=7cm]
		% Grid
		\draw[step=1cm,gray!30,very thin] (-0.5,-0.5) grid (4.5,3.5);
		
		% Axes
		\draw[->, >=stealth] (-0.5,0) -- (4.5,0) node[right] {$x_1$};
		\draw[->, >=stealth] (0,-0.5) -- (0,3.5) node[above] {$x_2$};
		
		% Ticks
		\foreach \x in {1,2,3,4} \draw (\x,-0.1) -- (\x,0.1) node[below=0.15cm] {\small $\x$};
		\foreach \y in {1,2,3} \draw (-0.1,\y) -- (0.1,\y) node[left=0.15cm] {\small $\y$};
		
		% 1. X [0, \theta] (KEPT from X^(1))
		\draw[myred, thick] (0,0) .. controls (0.5, 1.5) .. (2,1) node[midway, below right, text=black] {$\X_{0,\theta}$};
		
		% 2. X^(1) [\theta, T] (DISCARDED - Drawn faded/dashed for context)
		\draw[myred, thick, dashed, opacity=0.3] (2,1) .. controls (2.5, 0.5) .. (3, 0.5);
		
		% 3. X^(2) [0, \theta] (SHIFTED & DISCARDED - Drawn faded/dashed)
		% Original points: (0,2) to (1,3). Shift vector is (+1, -2)
		\draw[myblue, thick, dashed, opacity=0.3] (1,0) .. controls (1.5, 0) .. (2,1);
		
		% 4. X [\theta, T] (SHIFTED & KEPT from X^(2))
		% Original points: (1,3) to (2.5,4). Applied shift (+1, -2).
		\draw[myblue, thick] (2,1) .. controls (2.5, 2) .. (3.5, 2) node[right, text=black] {$\X_{\theta,T}$};
		
		% Junction point \theta for the new concatenated path \X
		\filldraw[black] (2,1) circle (1.5pt) node[right] {$\X_\theta$};
		
		% Panel Label
		\node at (2, -1) {The observed path $\X$};
		
	\end{scope}
	
\end{tikzpicture}
\caption{An example of the hidden and observed paths in $\R^2$}
	\end{figure}
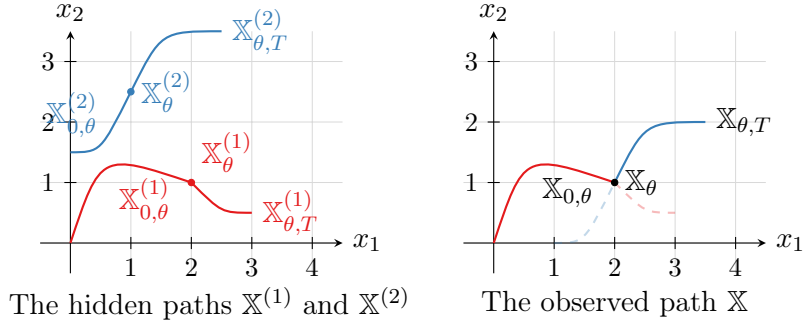

\subsection{The observed process}
Having specified the probabilistic setup, we now describe the observation model. The observer has access neither to $X^{(1)}$
and $X^{(2)}$ individually nor to the change-point $\theta$, but
only to the path obtained by switching from one regime to the other at time $\theta$. In our setting of the quickest detection problem, the individual processes $\X^{(1)}$ and $\X^{(2)}$ remain hidden from the observer. Instead, one observes a process $\X$, defined to be the concatenation of $\X^{(1)}$ and $\X^{(2)}$ at the random change-point $\theta$. More precisely, for any time $t\in[0,T]$, the observed rough path is given by
\begin{equation}\label{eq:X}
	\X_{0,t}=\1_{\{t<\theta\}} \X^{(1)}_{0,t}+\1_{\{t\geq \theta\}}\left(\X^{(1)}_{0,\theta}\otimes\X^{(2)}_{\theta,t}\right).
\end{equation}
For $t \geq \theta$, the resulting signature on $[0,t]$ follows directly from Chen's identity:
$$
\pi_n(\X_{0, t})=\sum_{k=0}^n \pi_k(\X_{0, \theta}^{(1)}) \otimes \pi_{n-k}(\X_{\theta, t}^{(2)}), \quad \text{for all } n \in \{1, \dots, N \}.
$$
Specifically, for the level $1$ signature, which corresponds to the increment of the underlying path trajectory, we have:
$$
\pi_1(\X_{0,t})=\pi_1(\X_{0, \theta}^{(1)}) \cdot 1+1 \cdot \pi_1(\X_{\theta, t}^{(2)})=\pi_1(\X_{0, \theta}^{(1)})+\pi_1(\X_{\theta, t}^{(2)})=X_\theta^{(1)}+X_{t}^{(2)}-X_{\theta}^{(2)}.
$$
This demonstrates that the underlying trajectory is continuous in $t$, naturally justifying the term "concatenation". In the sequel, we will establish that the concatenated process $\X$ is indeed a geometric $p$-rough path.

\subsection{Quickest detection objective}\label{sec:obj}
With the observation model in place, we now state the two optimal stopping objectives that formalize the trade-off between detection delay and false alarm probability. Recall the notation $\hat{X}$ for the time-augmented version of the rough path $X$. Denote $\cf^X_t:=\sigma(\widehat{X}_{s}:0\leq s\leq t )$ to be the natural filtration generated by the time-augmented observed process $\widehat{\X}$. Here the process $\X$ is the observation process as defined in \eqref{eq:X} with a latent unobservable change-point $\theta$. The goal of quickest detection is a stopping time $\tau$ to signal a change-point or that regime shift has already occurred. For such a stopping time $\tau$, the expected delay is given by:
 \begin{equation*}
    	\E\left[(\tau-\theta)^{+}\right],
 \end{equation*} 
 while the false alarm probability is:
 \begin{equation*}
    	\mathbb{P}(\tau<\theta).
 \end{equation*}
 Let $\cs$ be the set of all $(\cf^X_t)$-stopping times. We now consider the following objectives:
 \begin{enumerate}[ wide,labelwidth=!, labelindent=0pt]
 	\item[{\bf Objective I}:] Seek an $(\cf^X_t)$-stopping time $\tau$ such that the following cost function is minimized
 	\begin{equation*}
 		\inf_{\tau\in\cs} \mathbb{P}(\tau<\theta)+ c\,\E\left[(\tau-\theta)^{+}\right].
 	\end{equation*}
 	Here $c$ is the cost of delay. This is precisely the objective in the Bayesian formulation of quickest detection, e.g. in \cite{Shiryaev1978, Shiryaev2010}. The corresponding stochastic process modeling the cost function in this optimal stopping problem is
 	\begin{equation}\label{eq:Y1}
 		Y_t = Y_t^1:=\1_{\{t<\theta\}}+c(t-\theta)^+.
 	\end{equation}
    An alternative formulation treats both false alarms and detection delay symmetrically as costs, rather than penalizing false alarms through a probability constraint.
 	\item[{\bf Objective II}:]  Seek an $(\cf^X_t)$-stopping time $\tau$ such that the following cost function is minimized
 	\begin{equation*}
 		\inf_{\tau\in\cs}a\,\E\left[(\tau-\theta)^{-}\right]+b\,\E\left[(\tau-\theta)^{+}\right].
 	\end{equation*}
 	Here $a$ is the false alarm cost and $b$ is the cost of delay. And the corresponding stochastic process modeling the cost function in this optimal stopping problem is
 	\begin{equation}\label{eq:Y2}
 		Y_t=Y_t^2:=a(t-\theta)^{-} +b(t-\theta)^{+}.
 	\end{equation}
 \end{enumerate}
    	    	
\subsection{The optimal stopping policy}\label{sec:opt-stop}
We now show that both optimal stopping problems in
Section~\ref{sec:obj} admit a solution of the same structural form: the first hitting time of a half-space by a linear functional of the signature. The argument adapts the framework of
\cite{Bayer2023} to the non-Markovian, non-continuous loss
process arising in quickest detection; we include the full
details in Appendix~\ref{sec:rd-stop} for completeness.

\begin{definition}\label{def:lin-sig}
    	For $l \in T((\R^{1 + d})^*)$ define the hitting time of the signature against the half-plane orthogonal to $l$ by
    	\begin{align*}
    		\tau_l \coloneqq \inf\left\{t \in [0,T]\, :\, |\langle l, {\widehat{\mathbb{X}}}_{0,t}^{< \infty} \rangle| \geq 1 \right\} .
    	\end{align*}
\end{definition}
    
\begin{proposition}\label{thm:signature-stopping}
Let $Y$ be right continuous in $t$ and $\E[\|Y\|_{\infty}]<\infty$ for any finite interval. The $(\cf^X_t)$-optimal stopping policy of $E[Y_{\tau\wedge T}]$ is the hitting time of the signature against a half-plane. More precisely,
    	\begin{align*}
    		\inf_{l \in T((\R^{1 + d})^*) } \E[Y_{\tau_l \wedge T}] = \inf_{\tau \in \mathcal{S}} \E[Y_{\tau \wedge T}].
    	\end{align*}
\end{proposition}
    
    \begin{proof}
See Appendix~\ref{sec:rd-stop}.
   \end{proof}
Proposition~\ref{thm:signature-stopping} confirms that the optimal rule requires no parametric knowledge of the pre- and post-change distributions and no Markovian structure. Instead of a low-dimensional statistic whose form is dictated by the model, the rule uses a linear functional of the (truncated) signature as a universal model-free feature of the observed path history, with the truncation level $N$ serving as a controllable approximation parameter.

 \subsection{Properties and examples of the observed process}
 In this section we provide some results and examples for the observed process. Before proceeding, we verify that the concatenated observation process inherits the rough path regularity of its two constituent processes.
 \begin{lemma}\label{lem:geo-rp}
The rough path $\X$ defined in \eqref{eq:X} is a geometric $p$-rough path almost surely.
 \end{lemma}
 \begin{proof}
For almost all $\omega \in \bar{\Omega}$ we will prove that $\X(\omega)\in \Omega_T^p$. We will abuse the notation $\X$ as any realization of itself in this proof for simplicity. We first show that $\| \mathbb{X} \|_{p-\mathrm{var};[0,T]}<\infty$. Since $\X^{(1)}, \X^{(2)}\in G^{\lfloor p \rfloor}(\R^d)$, and $(G^{\lfloor p \rfloor}(\R^d), \otimes)$ is a group, we have that for every fixed $\theta$ and $t\geq\theta$, $\X^{(1)}_{0,\theta}\otimes\X^{(2)}_{\theta,t}\in G^{\lfloor p \rfloor}(\R^d)$. Therefore, $\mathbb{X}$ is a continuous path and $\mathbb{X} \colon [0,T] \to G^{\lfloor p \rfloor}(\R^d)$ with $\X_0 = \mathbf{1}$. For every fixed $\theta\in[0,T]$, the $p$-variation of $\X$ satisfies:
\begin{equation}\label{eq:Xpvar1}
	\| \mathbb{X} \|_{p-\mathrm{var};[0,T]}^p 
    % &= \sup_{\mathcal{D} \subset [0,T]} \sum_{t_i \in \mathcal{D}} \|\mathbb{X}_{t_i,t_{i+1}}\|_{CC}^p \nonumber\\
	\leq \sup_{\mathcal{D} \subset [0,T]} \left( \sum_{t_i \leq \theta} \|\mathbb{X}^{(1)}_{t_i,t_{i+1}}\|_{CC}^p + \|\mathbb{X}_{t_m,t_{m+1}}\|_{CC}^p + \sum_{t_i \geq \theta} \|\mathbb{X}^{(2)}_{t_i,t_{i+1}}\|_{CC}^p \right),
\end{equation}
where $[t_m, t_{m+1}]$ is the interval in the partition $\mathcal{D}$ crossing $\theta$ (i.e., $t_m \leq \theta \leq t_{m+1}$). The first and the last sums on the right-hand side of \eqref{eq:Xpvar1} are naturally bounded by $\| \mathbb{X}^{(1)} \|_{p-\mathrm{var};[0,\theta]}^p$ and $\| \mathbb{X}^{(2)} \|_{p-\mathrm{var};[\theta,T]}^p$. To show $\| \mathbb{X} \|_{p-\mathrm{var}}<\infty$, it suffices to bound the cross-term over the interval $s=t_m \leq \theta \leq t_{m+1}=t$. 

Using Chen's identity and the subadditivity of $\|\cdot\|_{CC}$ we have:
\begin{align}\label{eq:Xst}
	\|\mathbb{X}_{s,t}\|_{CC}^p = \|\mathbb{X}^{(1)}_{s,\theta} \otimes \mathbb{X}^{(2)}_{\theta,t}\|_{CC}^p &\leq \left( \|\mathbb{X}^{(1)}_{s,\theta}\|_{CC} + \|\mathbb{X}^{(2)}_{\theta,t}\|_{CC} \right)^p \nonumber\\
	&\leq 2^{p-1} \left( \|\mathbb{X}^{(1)}\|_{p-\mathrm{var};[s,\theta]}^p + \|\mathbb{X}^{(2)}\|_{p-\mathrm{var};[\theta,t]}^p \right) < \infty,
\end{align}
where the last inequality follows from $\X^{(1)}, \X^{(2)}\in\Omega_T^p$. Therefore, $\| \mathbb{X} \|_{p-\mathrm{var}}<\infty$. 

Moreover, since $\X^{(1)}, \X^{(2)}\in\Omega_T^p$, there exist sequences of bounded variation paths $(X_n^{(1)}), (X_n^{(2)})$ such that their canonical lifts satisfy:
\begin{equation}\label{eq:dpvX12}
	d_{p-\mathrm{var}}(\X^{(1)}, \X_n^{(1), \leq \lfloor p \rfloor}) \to 0 \quad \text{and} \quad d_{p-\mathrm{var}}(\X^{(2)}, \X_n^{(2), \leq \lfloor p \rfloor}) \to 0,
\end{equation} 
as $n \to \infty$. Define the concatenated smooth path $(X_n)$ as:
\begin{align*}
	X_{n;t}:= \begin{cases}
		X_{n;t}^{(1)} & \text { if } 0 \leq t<\theta \\ 
		X_{n;\theta}^{(1)}+X_{n;t}^{(2)}-X_{n;\theta}^{(2)} & \text { if } t \geq \theta
	\end{cases}.
\end{align*}	 
It is easy to see that $(X_n)$ is again a sequence of bounded variation paths. From Lyons' Extension Theorem \cite[Thm 3.7]{Lyons2007}, $(X_n)$ has a unique lift $(\X_n^{<\infty})$ satisfying:
\begin{equation*}%\label{eq:Xn0t}
	\X_{n;0,t}^{<\infty}=\X^{(1),<\infty}_{n;0,\theta}\otimes\X^{(2),<\infty}_{n;\theta,t}.
\end{equation*} 
Finally, we prove $\lim_{n \to \infty} d_{p-\mathrm{var};[0,T]}(\X, \X_n) = 0$. Let us first consider the cross term. Fix $s \leq \theta \leq t$ and $k \in \{1, \dots, \lfloor p \rfloor\}$. By Chen's identity,
\begin{align}\label{eq:tensor_diff}
&\left|\pi_k(\X_{s,t}) - \pi_k(\X_{n; s,t})\right|
=\left|\sum_{j=0}^k\left(\pi_j(\X^{(1)}_{s,\theta}) \otimes \pi_{k-j}(\X^{(2)}_{\theta,t})-\pi_j(\X^{(1)}_{n;s,\theta}) \otimes \pi_{k-j}(\X^{(2)}_{n;\theta,t})\right)\right| \nonumber\\
&\quad\quad\leq\sum_{j=0}^k\left|\pi_j(\X^{(1)}_{s,\theta}) - \pi_j(\X^{(1)}_{n;s,\theta})\right|\left|\pi_{k-j}(\X^{(2)}_{\theta,t})\right| +\sum_{j=0}^k\left|\pi_j(\X^{(1)}_{n;s,\theta})\right|\left|\pi_{k-j}(\X^{(2)}_{\theta,t}) - \pi_{k-j}(\X^{(2)}_{n;\theta,t})\right|.
\end{align}
Since the increment sets of $\X^{(2)}$ and $\X_n^{(1)}$ over subintervals of $[0,T]$ lie in a common compact subset of $G^{\lfloor p\rfloor}(\R^d)$ for all sufficiently large $n$, their tensor components are uniformly bounded. Hence there exists a constant $K<\infty$ such that
\begin{equation}\label{eq:unif_bound}
	\sup_{s \leq \theta \leq t}\max_{j}\left(\left| \pi_j(\X^{(2)}_{\theta,t}) \right|\vee\left| \pi_j(\X^{(1)}_{n;s,\theta}) \right|\right)\leq K.
\end{equation}
Furthermore, \eqref{eq:dpvX12} implies that
\begin{equation}\label{eq:comp_conv}
	\lim_{n \to \infty} \sup_{s \leq \theta}\left|\pi_j(\X^{(1)}_{s,\theta}) - \pi_j(\X^{(1)}_{n;s,\theta})\right| = 0\quad \text{and} \quad\lim_{n \to \infty} \sup_{t \geq \theta}\left|	\pi_{k-j}(\X^{(2)}_{\theta,t}) - \pi_{k-j}(\X^{(2)}_{n;\theta,t})\right| = 0.
\end{equation}
Applying \eqref{eq:unif_bound} and \eqref{eq:comp_conv} to \eqref{eq:tensor_diff} yields
\begin{equation}\label{eq:cross_limit}
	\lim_{n \to \infty} \sup_{s \leq \theta \leq t}\left|\pi_k(\X_{s,t}) - \pi_k(\X_{n; s,t})\right| = 0,	\quad \forall k \in \{1, \dots, \lfloor p \rfloor\}.
\end{equation}
Equation \eqref{eq:cross_limit} establishes uniform componentwise convergence in the ambient tensor algebra $T^{\lfloor p \rfloor}(\R^d)$. Since group inversion and tensor multiplication are polynomial maps, they are continuous. Therefore, the uniform convergence $\X_{n; s,t} \to \X_{s,t}$ implies that the group increment $h_n(s,t) := \X_{n; s,t}^{-1} \otimes \X_{s,t}$ converges uniformly to the group identity $\mathbf{1}$. Consequently, its tensor components vanish:
\begin{equation}\label{eq:h_limit}
	\lim_{n \to \infty} \sup_{s \leq \theta \leq t}\left| \pi_k(h_n(s,t)) \right| = 0,	\quad \forall k \in \{1, \dots, \lfloor p \rfloor\}.
\end{equation}
By \cite[Thm 7.44]{friz-victoir}, there exists a constant $c>0$ such that for any $h \in G^{\lfloor p \rfloor}(\R^d)$,
$$
\|h\|_{CC} \leq c \max_{i=1,\dots,\lfloor p \rfloor} |\pi_i(h)|^{1/i}.
$$
Applying this to $h_n(s,t)$, and using that $d_{CC}(\X_{s,t}, \X_{n; s,t}) = \|h_n(s,t)\|_{CC},$
\eqref{eq:h_limit} implies that
\begin{equation}\label{eq:d_CC_n_conv}
	\lim_{n \to \infty} \sup_{s \leq \theta \leq t}
	d_{CC}(\X_{s,t}, \X_{n; s,t}) = 0.
\end{equation}
By a similar argument to \eqref{eq:Xpvar1}-\eqref{eq:Xst}, using \eqref{eq:d_CC_n_conv} together with the limits on the non-crossing intervals, we obtain
\begin{equation*}
	\lim_{n \to \infty} d_{p-\mathrm{var};[0,T]}(\X, \X_n) = 0.
\end{equation*}
Therefore, $\X \in \Omega_T^p$.
 \end{proof}
 
Since $\X,\X^{1},\X^{2}\in\Omega_T^p$, we can impose the distance in Definition~\ref{def:d-p-var-CC} on them. The following lemma shows that the homogeneous $p$-variation distance between the observed process $\X$ and the initial process $\X^{(1)}$ over the interval $[0,t]$ is equal to their distance over $[\theta,t]$. In other words, this distance exclusively captures the divergence between $\X$ and $\X^{(1)}$ that occurs after the change.
\begin{lemma}\label{lem:d-p-var}
	For every given $\theta$ and any $t\in(\theta,T]$, the rough path $\X$ defined in \eqref{eq:X} satisfies
	\begin{equation*}
		d_{p-\mathrm{var};[0,t]}\left(\X,\X^{(1)}\right)=d_{p-\mathrm{var};[\theta,t]}\left(\X,\X^{(1)}\right).
	\end{equation*}
\end{lemma}
\begin{proof}
By Definition~\ref{def:d-p-var-CC}, for every fixed $\theta\in[0,T]$, taking the supremum over partitions $\mathcal{D} = \{t_0, t_1, \dots, t_n\}$ of $[0,t]$, we have
{\small	
\begin{align}\label{eq:X-d-p-var-CC}
			&d_{p-\mathrm{var};[0,t]}\left(\X,\X^{(1)}\right)^p=\sup _{\mathcal{D}\subset [0,t]} \sum_{i=0}^{n-1} d_{CC}\left(\X_{t_i, t_{i+1}}, \X^{(1)}_{t_i, t_{i+1}}\right)^p=\nonumber\\
			&\sup _{\mathcal{D}\subset [0,t]} \left(\sum_{i=0}^{m-1} d_{CC}\left(\X_{t_i, t_{i+1}}, \X^{(1)}_{t_i, t_{i+1}}\right)^p+d_{CC}\left(\X_{t_m, t_{m+1}}, \X^{(1)}_{t_m, t_{m+1}}\right)^p+\sum_{i=m+1}^{n-1} d_{CC}\left(\X_{t_i, t_{i+1}}, \X^{(1)}_{t_i, t_{i+1}}\right)^p\right),
\end{align}}
where $t_m \leq \theta < t_{m+1}$. Notice that $\X$ coincides with $\X^{(1)}$ up to time $\theta$. Therefore,
\begin{equation}\label{eq:d-X-X1-0}
		d_{CC}\left(\X_{t_i, t_{i+1}}, \X^{(1)}_{t_i, t_{i+1}}\right)=0,\quad\text{for } i\leq m-1.
\end{equation}
	Furthermore, by Chen's identity and the definition of $\X$, we have $\X_{t_m,t_{m+1}}=\X^{(1)}_{t_m,\theta}\otimes\X^{(2)}_{\theta,t_{m+1}}$ and $\X^{(1)}_{t_m,t_{m+1}}=\X^{(1)}_{t_m,\theta}\otimes\X^{(1)}_{\theta,t_{m+1}}$. Therefore, by the left-invariance property of $d_{CC}$, we obtain
	\begin{equation}\label{eq:d-X-X1-m}
		d_{CC}\left(\X_{t_m, t_{m+1}}, \X^{(1)}_{t_m, t_{m+1}}\right)=d_{CC}\left(\X^{(1)}_{t_m,\theta}\otimes\X^{(2)}_{\theta,t_{m+1}}, \X^{(1)}_{t_m,\theta}\otimes\X^{(1)}_{\theta,t_{m+1}}\right)=d_{CC}\left(\X^{(2)}_{\theta,t_{m+1}}, \X^{(1)}_{\theta,t_{m+1}}\right).
	\end{equation}
	Plugging \eqref{eq:d-X-X1-0} and \eqref{eq:d-X-X1-m} into \eqref{eq:X-d-p-var-CC}, and noting that $\X_{\theta,t_{m+1}}=\X^{(2)}_{\theta,t_{m+1}}$, we have
	\begin{align*}
		d_{p-\mathrm{var};[0,t]}\left(\X,\X^{(1)}\right)&=\left(\sup _{\mathcal{D}\subset [0,t]} \left(d_{CC}\left(\X_{\theta, t_{m+1}}, \X^{(1)}_{\theta, t_{m+1}}\right)^p+\sum_{i=m+1}^{n-1} d_{CC}\left(\X_{t_i, t_{i+1}}, \X^{(1)}_{t_i, t_{i+1}}\right)^p\right)\right)^{1 / p}\\
		&=d_{p-\mathrm{var};[\theta,t]}\left(\X,\X^{(1)}\right),
	\end{align*}
	This yields the desired result.
\end{proof}
Lemma~\ref{lem:d-p-var} shows that the homogeneous $p$-variation distance serves as a natural measure of the accumulated post-change deviation, and is therefore a natural basis for constructing a detector, as we pursue in Section~\ref{sec:detector}. A natural question is whether there are practical examples of processes that satisfy Assumption~\ref{asm:distinct-paths}. Below, we provide several such examples and demonstrate that the classical quickest change-point detection problem can be adapted to our framework.
\begin{example}\label{ex:BM}
Let us consider the classical quickest detection problem (see \cite{Shiryaev1978, Shiryaev2010}). For a $d$-dimensional Brownian motion $B$, define
\begin{equation*}
	dX_t^{(1)}=\sigma dB_t, \quad dX_t^{(2)}=r dt +\sigma dB_t.
\end{equation*}
In the rough path point of view, for $i=1,2$, the natural (Stratonovich) lift (see \cite[Sec 13.2]{friz-victoir}) of $X^{(i)}$ to a geometric rough path $\X^{(i)} \in \Omega_T^p$ with $p \in(2,3)$ is given by 
$$
	\X_{s, t}^{(i)}=\left(1, X_{s, t}^{(i)}, \int_s^t X_{s, u}^{(i)} \otimes \circ d X_u^{(i)}\right), \quad 0 \leq s \leq t \leq T
$$
where $X_{s,t}=X_t-X_s$, and $\int_s^t X_{s, u} \otimes \circ d X_u \in\left(\mathbb{R}^d\right)^{\otimes 2}$ denotes the Stratonovich second iterated integral of $X$, defined as the limit of midpoint Riemann sums $\sum_k X_{s,\left(t_k+t_{k+1}\right) / 2} \otimes X_{t_k, t_{k+1}}$. Recall that the process $\X$ is defined as in \eqref{eq:X}. The level $1$ signature (the path trajectory) of $\X$ is given by:
$$X_{0,t}=X_\theta^{(1)}+\1_{\{t\geq \theta\}}\left(X_{t}^{(2)}-X_{\theta}^{(2)}\right).$$
Or equivalently, 
$$dX_t=r\1_{\{t\geq \theta\}} dt+\sigma dB_t,$$
which coincide with the \emph{classical quickest detection problem}. We will show in the following that this example satisfies the Assumption~\ref{asm:distinct-paths}. The inequality \eqref{eq:X12-q-moment} holds true for Gaussian process (see for example Fernique-estimates in \cite[Thm 15.33]{friz-victoir}). It remains to verify \eqref{eq:X12>g}.
Denote $A^{(i)}_t=A^{(i)}_{0,t}$ the (Stratonovich) L\'evy area of $\X^{(i)},\, i=1,2$. $A_t^{(i)}\in\Lambda^2(\mathbb{R}^d)$ is an antisymmetric matrix. For $j, k \in\{1, \ldots, d\}$, the $(j,k)$-component of $A^{(i)}_{s,t}$ is defined as
$$
(A_{s,t}^{(i)})^{ j, k}:=\frac{1}{2} \int_s^t\left((X_{s,u}^{(i)})^{ j} \circ d (X_u^{(i)})^{ k}-(X_{s,u}^{(i)})^{ k} \circ d (X_{u}^{(i)})^{ j}\right).
$$
A more concise form of the L\'evy area is 
$$A_{s,t}^{(i)}=\frac{1}{2}\int_s^t [X_u^{(i)}-X_s^{(i)},\circ d X_u^{(i)}].$$
Here $[\cdot,\cdot]$ is the Lie bracket. By \cite[Thm 7.30]{friz-victoir} (also see \cite[Sec 13]{friz-victoir}), we can write the group element in the following form:
 $$\X_{0, t}^{(i)}=\exp \left(X_{0, t}^{(i)}+A_t^{(i)}\right) \in G^2\left(\mathbb{R}^d\right),\quad \text{for } i=1,2 .$$
Here the exponential is the Lie-group exponential from the step-$2$ Lie algebra $\mathfrak{g}^2\left(\mathbb{R}^d\right)=\mathbb{R}^d \oplus \Lambda^2\left(\mathbb{R}^d\right)$. Recall that the Carnot-Carath\'{e}odory metric $d_{CC}$ is induced by a homogeneous norm $\|\cdot\|_{CC}$. Therefore, by Theorem~$7.44$ and Proposition~$7.45$ in \cite{friz-victoir} (also see \cite[Sec 2.3]{Friz2020}) we have 
\begin{equation}\label{eq:dCC-X1X2}
	d_{CC}\left(\X^{(1)}_{s,t},\X^{(2)}_{s,t}\right)=d_{CC}\left(e^{ \left(X_{s, t}^{(1)}+A_{s,t}^{(1)}\right)},e^{ \left(X_{s, t}^{(2)}+A_{s,t}^{(2)}\right)}\right)\sim \left|X_{s, t}^{(2)}-X_{s, t}^{(1)}\right| \vee \left|A_{s, t}^{(2)}-A_{s, t}^{(1)}\right|^{1/2}. 
\end{equation}
It is readily checked that $|X_{s, t}^{(2)}-X_{s, t}^{(1)}|=r|t-s|$. Let us find out the scaling of $|A_{s, t}^{(2)}-A_{s, t}^{(1)}|$. By definition 
\begin{equation}\label{eq:A1}
	A_{s,t}^{(1)}=\frac{1}{2}\int_s^t[\sigma(B_u-B_s),\sigma\circ d B_u],
\end{equation}
and
\begin{equation}\label{eq:A2}
	A_{s,t}^{(2)}=\frac{1}{2}\int_s^t[r (u-s)+\sigma(B_u-B_s),r du + \sigma\circ d B_u].
\end{equation}
Since $[r (u-s),r du]=0$, subtracting \eqref{eq:A1} from \eqref{eq:A2} and using the bi-linearity and the anti-symmetry property of Lie bracket, we have
\begin{align}\label{eq:A2-A1}
	A_{s,t}^{(2)}-A_{s,t}^{(1)}&=\frac{1}{2}\int_s^t\left([r(u-s),\sigma\circ d B_u]+[\sigma(B_u-B_s),r d u]\right)\nonumber\\
	&=[\frac{\sigma}{2} r,\int_s^t (u-s) \circ dB_u ]-[\frac{\sigma}{2} r,\int_s^t B_u-B_s du].
\end{align}
By integration by part for the Stratonovich integral, we have 
\begin{equation*}
	\int_s^t(u-s) \circ dB_u=(t-s)(B_t-B_s)-\int B_u-B_s du.
\end{equation*}
Plugging this back to \eqref{eq:A2-A1}, we get
\begin{equation*}
	A_{s,t}^{(2)}-A_{s,t}^{(1)}=[\frac{\sigma}{2} r,(t-s)(B_t-B_s)-2\int_s^t B_u-B_s du].
\end{equation*}
Since the term $W:=(t-s)(B_t-B_s)-2\int_s^t B_u-B_s du$ is a centered Gaussian vector, it is readily checked that
\begin{equation*}
	\E\left|A_{s,t}^{(2)}-A_{s,t}^{(1)}\right|^{1/2}\sim |\sigma r|^{1/2}(t-s)^{3/4}.
\end{equation*}
Therefore from \eqref{eq:dCC-X1X2} we know that 
\begin{equation*}
	\E\left(d_{CC}\left(\X^{(1)}_{s,t},\X^{(2)}_{s,t}\right)\right)\sim r|t-s|+|\sigma r|^{1/2}(t-s)^{3/4}.
\end{equation*}
Since $d_{p-\text{var};[s,t]}(\X,\Y)\geq d_{CC}(\X_{s,t},\Y_{s,t})$ for any $\X,\Y\in G^N(\R^d)$, the Assumption~\ref{asm:distinct-paths} is satisfied.
\end{example}
\begin{example}\label{ex:fBM}
	Let the diffusion terms in the last example driven by fractional Brownian motion with Hurst parameter $H\in (1/3,1/2)$. That is 
	\begin{equation*}
		dX_t^{(1)}=\sigma dB_t^H, \quad dX_t^{(2)}=r dt +\sigma dB_t^H.
	\end{equation*}
We can also define the natural (Stratonovich) lift of $X^{(i)}$ to a geometric rough path $\X^{(i)} \in \Omega_T^p$ with $p \in(2,3)$ as
$$
\X_{s, t}^{(i)}=\left(1, X_{s, t}^{(i)}, \int_s^t X_{s, u}^{(i)} \otimes \circ d X_u^{(i)}\right), \quad 0 \leq s \leq t \leq T, \quad i=1,2.
$$
Using a similar argument to the last example, one can get
\begin{equation*}
	\E\left(d_{CC}\left(\X^{(1)}_{s,t},\X^{(2)}_{s,t}\right)\right)\sim r|t-s|+|\sigma r|^{1/2}(t-s)^{(H+1)/2},
\end{equation*}
satisfying Assumption~\ref{asm:distinct-paths}. For $H\leq 1/3$ it is also possible to get the lift to geometric rough path $\X^{(i)} \in \Omega_T^p$ with $p \in(1/H,1/H+1)$, as a group element of $G^{\lfloor p\rfloor}(\R^d)$. We refer to \cite[Sec.\ 13--14]{friz-victoir} for the construction of geometric rough path lifts of fractional Brownian motion and related Gaussian processes.
\end{example}
The examples above confirm that Assumption~\ref{asm:distinct-paths} is satisfied in the classical Brownian and fractional Brownian settings, and that the framework of Section~\ref{sec:QD-rough} strictly extends the classical theory. In the next section, we take a complementary perspective and construct a detector directly from the intrinsic geometry of the observed path, without reference to the payoff process $Y$.
\section{Detector based on signature}\label{sec:detector}
The stopping policy derived in Section~\ref{sec:opt-stop} was motivated purely by the structure of the payoff $Y$ and the approximation theory developed in Section~\ref{sec:approx}. In this section, we provide an independent geometric explanation for why a linear functional of the signature should serve as an effective change-point detector. Specifically, we construct a detector directly from the intrinsic geometry of the observed path, without reference to the payoff $Y$, and show that it has the same structural form as the optimal stopping rule in Proposition~\ref{thm:signature-stopping}. We then establish statistical guarantees for this detector and discuss its practical implementation.
\subsection{Detector function}
As we notice in Lemma~\ref{lem:d-p-var}, the homogeneous $p$-variation distance solely records the deviation of $\X$ from $\X^{(1)}$ following the change. This inspires us to use it as a natural detector of the change-point. More precisely, for any $\Z\in\Omega_T^p$, define
\begin{equation}\label{eq:f}
	f(\widehat{\Z}):=f\left(\Z,t\right)=d_{p-\mathrm{var};[0,t]}\left(\Z,\X^{(1)}\right).
	\end{equation}
Observe that $f$ is a continuous function satisfying $f(\X^{(1)},t)=0$ for any $t\in[0,T]$. Under Assumption~\ref{asm:distinct-paths}, by using Lemma~\ref{lem:d-p-var} we have 
\begin{equation}\label{eq:fX>g}
\E[f(\X,t)]\geq g\left((t-\theta)^+\right).
\end{equation}
Equation~\eqref{eq:fX>g} shows that the expected value of $f(X,t)$  grows with the elapsed time since the change, providing a quantitative signal of the regime shift. Since $f$ depends on the entire path history through $d_{p\text{-var}}$, it is not directly computable from a finite-dimensional statistic. The next proposition shows that $f$ can nevertheless be 
well approximated by a linear functional of the signature, yielding a tractable detector with provable statistical guarantees.
\begin{proposition}\label{prop:l1-X-de-ep}
Let the observed process $\X$ be defined as \eqref{eq:X} with $\X^{(1)},\X^{(2)}$ satisfying Assumption~\ref{asm:distinct-paths}. Then for any $\ep,\delta>0$ there exist $l^{(1)}	\in T((\R^{1 + d})^*)$ such that 
\begin{equation}\label{eq:P-l1-X1}
	\PP\left(\sup_{t\in[0,T]}|\langle l^{(1)},\widehat{\X}^{(1),<\infty}_{0,t}\rangle| \leq \delta \right)\geq 1-\ep,
	\end{equation}
and 
\begin{equation}\label{eq:P-l1-X}
	\PP\left(\sup_{s\in[0,t]}|\langle l^{(1)},\widehat{\X}^{<\infty}_{0,s}\rangle| \geq g\left((t-\theta)^+\right)-\delta \right)\geq 1-\ep,
\end{equation}
for all $t\in[0,T]$. Moreover, there exist a constant $c$ depending on $T,p,q$ such that for all $t\in[0,T]$
$$\E[|\langle l^{(1)},\widehat{\X}^{<\infty}_{0,t}\rangle|]\geq g\left((t-\theta)^+\right)-c_{\ep,\delta},$$
where $c_{\ep,\delta}=\delta(1-\ep)+cM^{1/q}\ep^{1/m}$ and $1/p+1/m=1$.
	\end{proposition}
\begin{proof}
	Consider $f\in\ct$ defined as in \eqref{eq:f}. By Lemma~\ref{lem:sig_dense}, for any $\ep > 0$, there exists a compact set $\mathcal{K} \subset \hat{\Omega}_T^p$ with $\mathbb{P}(\mathcal{K}) \ge 1 - \ep$ and a sequence $\{\phi_n\} \subset \mathcal{T}_{sig}$ such that:
	\begin{equation}\label{eq:phi_n-to-f}
		\lim_{n \to \infty} \sup_{\widehat{\mathbb{Y}} \in \mathcal{K}, t \in [0, T]} |\phi_n(\widehat{\mathbb{Y}}|_{[0,t]}) - f(\widehat{\mathbb{Y}}|_{[0,t]})| = 0.
	\end{equation}
	Hence for any $\delta>0$ there exists $l^{(1)}	\in T((\R^{1 + d})^*)$ such that $\sup _{\widehat{\Y} \in \mathcal{K}, t \in[0, T]}|\langle l^{(1)},\widehat{\Y}|_{[0, t]}\rangle-f(\widehat{\Y}|_{[0, t]})| \leq \delta $. Therefore, by \eqref{eq:fX>g} in the set $\{\widehat{\mathbb{X}} \in \mathcal{K}\}$ we have 
	$$\sup_{t\in[0,T]}|\langle l^{(1)},\widehat{\X}^{(1),<\infty}_{0,t}\rangle| \leq \delta,\quad \text{and }\sup_{s\in[0,t]}|\langle l^{(1)},\widehat{\X}^{<\infty}_{0,s}\rangle| \geq g\left((t-\theta)^+\right)-\delta \, \text{for all } t\in[0,T].$$
	This prove \eqref{eq:P-l1-X1}-\eqref{eq:P-l1-X}.\\
	
	For the remaining result we decompose the expectations over the set $A = \{\widehat{\mathbb{X}} \in \mathcal{K}\}$ and its complement $A^c$:
	\begin{equation*}
		\E[|\langle l^{(1)}, \widehat{\mathbb{X}}^{<\infty}_{0,t}\rangle| ]=\E[|\langle l^{(1)}, \widehat{\mathbb{X}}^{<\infty}_{0,t}\rangle| ; A]+\E[|\langle l^{(1)}, \widehat{\mathbb{X}}^{<\infty}_{0,t}\rangle| ; A^c].
	\end{equation*}
Therefore, by \eqref{eq:fX>g} and \eqref{eq:phi_n-to-f} we have
\begin{align}\label{eq:l1-X-low}
	&\E[|\langle l^{(1)}, \widehat{\mathbb{X}}^{<\infty}_{0,t}\rangle| ]\geq \E[|\langle l^{(1)}, \widehat{\mathbb{X}}^{<\infty}_{0,t}\rangle| ; A]\geq\E[(f(\widehat{\mathbb{X}}|_{[0,t]})-\delta)\1_{A}]\nonumber\\
&=\E[(f(\widehat{\mathbb{X}}|_{[0,t]})]-\delta(1-\ep)-\E[(f(\widehat{\mathbb{X}}|_{[0,t]});A^c] \geq g\left((t-\theta)^+\right)-\delta(1-\ep)-\E[(f(\widehat{\mathbb{X}}|_{[0,t]});A^c].
\end{align}
Let us control the term $\E[(f(\mathbb{X}|_{[0,t]});A^c]$. By \cite[eq (8.5)]{friz-victoir} we have 
\begin{equation}\label{eq:f-q-moment}
	\E[f(\mathbb{X}|_{[0,t]})^q]\leq\E\left[d_{p-\mathrm{var};[0,T]}\left(\X^{(2)},\X^{(1)}\right)^q\right]\leq T^{1/p}\E\left[d_{1/p-\text{H\"{o}l};[0,T]}\left(\X^{(2)},\X^{(1)}\right)^q\right].
	\end{equation}
	Since by the sub-additive and symmetric property of $\|\cdot\|_{CC}$ we have
	\begin{align*}
		d_{CC}\left(\X^{(2)}_{s,t},\X^{(1)}_{s,t}\right)=\|(\X^{(2)}_{s,t})^{-1}\otimes\X^{(1)}_{s,t}\|_{CC}&\leq \|\X^{(2)}_{s,t}\|_{CC}+\|\X^{(1)}_{s,t}\|_{CC}\\
		&=d_{CC}\left(\X^{(2)}_{s},\X^{(2)}_{t}\right)+d_{CC}\left(\X^{(1)}_{s},\X^{(1)}_{t}\right).
	\end{align*}
	Therefore, we have
	\begin{equation}\label{eq:d-Hol-leq}
		d_{1/p-\text{H\"{o}l};[0,T]}\left(\X^{(2)},\X^{(1)}\right)\leq \|\X^{(2)}\|_{1/p-\text{H\"{o}l};[0,T]} +\|\X^{(1)}\|_{1/p-\text{H\"{o}l};[0,T]}.
	\end{equation}
	Plugging \eqref{eq:d-Hol-leq} and Assumption~\ref{asm:distinct-paths} in \eqref{eq:f-q-moment} and using an elementary inequality $(x+y)^q\leq2^{q-1}(x^q+y^q)$ for positive $x,y$ and $q>1$, we can conclude that there exist a constant $C$ depending on $T,p,q$ such that
	\begin{equation*}
			\E[f(\mathbb{X}|_{[0,t]})^q]\leq CM.
	\end{equation*}
	Therefore by H\"{o}lder inequality we have 
	\begin{equation*}
		\E[(f(\mathbb{X}|_{[0,t]});A^c]\leq \left(\E[f(\mathbb{X}|_{[0,t]})^q]\right)^{1/q}(1-\PP(\mathcal{K}))^{1/m}\leq cM^{1/q}\ep^{1/m},
	\end{equation*}
	where $1/p+1/m=1$ and $c$ is a constant depending on $T,p,q$. Plugging this back to \eqref{eq:l1-X-low} we get the desired result. 
	\end{proof}
Proposition~\ref{prop:l1-X-de-ep} shows that, with high probability, the linear functional $\langle l^{(1)}, \widehat{X}^{<\infty}_{0,t}\rangle$ 
remains near zero before the change and grows bounded away from zero after a delay of order $g^{-1}(1+\delta)$. This justifies using $\hat{\tau} = \inf\{t : |\langle l^{(1)}, 
\widehat{X}^{<\infty}_{0,t}\rangle| > 1\} \wedge T$ as a stopping rule, which coincides structurally with the optimal policy of Proposition~\ref{thm:signature-stopping}.
    
\subsection{Statistical guarantees}
We now make the statistical guarantees of the proposed stopping  rule precise. Throughout this subsection, we assume that the coefficient $l^{(1)}$ from Proposition~\ref{prop:l1-X-de-ep} is known, and derive bounds on detection delay and false alarm probability for the associated stopping time.
\begin{proposition}\label{prop:stat_1}
	Under the setting of Proposition~\ref{prop:l1-X-de-ep}, suppose $\theta\in[0,T]$. For any $0<\delta, \ep<1$ we define a stopping time
	\begin{equation*}
		\hat{\tau}=\inf\{t: |\langle l^{(1)},\widehat{\X}^{<\infty}_{0,t}\rangle|>1\}\wedge T.
	\end{equation*}
 Then 
	\begin{equation*}
		\PP\left((\hat{\tau}-\theta)^{+}\leq g^{-1}(1+\delta)\right)>1-\ep,\quad \text{and } \PP(\hat{\tau}<\theta)\leq \ep.
	\end{equation*}
	Here, the inverse function is defined as $g^{-1}(y)=\inf\{x\in\mathbb{R}:g(x)\geq y\}$. Moreover,
	\begin{equation*}
		\E[(\hat{\tau}-\theta)^+]\leq \E[|\hat{\tau}-\theta|]\leq g^{-1}(1+\delta)+\ep T.
	\end{equation*}
\end{proposition}  
\begin{proof}
	Similar to the proof of \eqref{eq:P-l1-X1}-\eqref{eq:P-l1-X}, by Lemma~\ref{lem:sig_dense}, for any $\ep > 0$, there exists a compact set $\mathcal{K} \subset \hat{\Omega}_T^p$ with $\mathbb{P}(\mathcal{K}) \ge 1 - \ep$ and in the set $A=\{\widehat{\mathbb{X}} \in \mathcal{K}\}$, for any $\delta>0$ there exists $l^{(1)}	\in T((\R^{1 + d})^*)$ such that
	\begin{equation}\label{eq:l1-delta}
		\sup_{t\in[0,T]}|\langle l^{(1)},\widehat{\X}^{(1),<\infty}_{0,t}\rangle| \leq \delta,\quad \text{and }\sup_{t\in[0,s]}|\langle l^{(1)},\widehat{\X}^{<\infty}_{0,s}\rangle| \geq g\left((t-\theta)^+\right)-\delta.
	\end{equation}
	Let us define
	\begin{equation*}
		\hat{\tau}^{\prime}=\inf\{t:g\left((t-\theta)^{+}\right)-\delta >1\}\wedge T.
	\end{equation*} 
	Within the set $A$, by \eqref{eq:l1-delta} we have
	\begin{equation}\label{eq:A-tau-theta}
		\hat{\tau}> \theta, \quad \text{and} \quad	(\hat{\tau}-\theta)^{+}\leq(\hat{\tau}^{\prime}-\theta)^{+}\leq g^{-1}(1+\delta).
	\end{equation}
	Therefore
	\begin{equation*}
		\PP\left((\hat{\tau}-\theta)^{+}\leq g^{-1}(1+\delta)\right)>1-\ep.
	\end{equation*}
	Similarly $\PP(\hat{\tau}<\theta)\leq \ep$. Moreover, by \eqref{eq:A-tau-theta} we have
	\begin{align*}
		\E[|\hat{\tau}-\theta|]&=\E[|\hat{\tau}-\theta|;A]+\E[|\hat{\tau}-\theta|;A^c]\\
		&\leq g^{-1}(1+\delta)\PP(\mathcal{K})+ T(1-\PP(\mathcal{K}))\leq  g^{-1}(1+\delta) +\ep T.
	\end{align*}
	This gives us the desired result.
\end{proof}
Proposition~\ref{prop:stat_1} provides guarantees for a single sample path. 
When $\theta$ is random and one has access to multiple independent sample paths, concentration inequalities yield the following finite-sample bound on the empirical risk.
\begin{corollary}\label{cor:stat_1-N}
Let us consider payoff $Y$ defined in \eqref{eq:Y1} and denote the $i$-th empirical payoff using stopping time $\tau$ to be $Y_{\tau^i}^i$. Suppose $\theta\in[0,T]$, and define $l^*$ to be the coefficient corresponding to the optimal stopping policy described in Section~\ref{sec:opt-stop}. More precisely,
$$
	l^*=\operatorname{\arg\inf}_{l \in T((\R^{1 + d})^*) } \E[Y_{\tau_l \wedge T}].
$$
Then for any $0<\delta, \ep<1$
$$
\E[Y_{\tau_{l^*}\wedge T}]\leq cg^{-1}(1+\delta)+\ep (cT+1),
$$
and 
$$
\PP\left(\frac{1}{R}\sum_{i=1}^{R}Y^i_{\tau_{l^*}^i\wedge T}\geq c_Y+cg^{-1}(1+\delta)+\ep (cT+1)\right)\leq\exp\left(-\frac{2Rc_Y^2}{(cT+1)^2}\right),
$$
for any $c_Y>0$.
\end{corollary}
\begin{proof}
Since $l^*$ corresponds to the optimal stopping policy, it can only improves the expected payoff than $l^{(1)}$. Therefore, by Proposition~\ref{prop:stat_1} we have
\begin{equation*}
	\E[Y_{\tau_{l^*}\wedge T}]\leq \E[Y_{\tau_{l^{(1)}}\wedge T}]=\PP(\hat{\tau}<\theta)+c\E[(\hat{\tau}-\theta)^+]\leq cg^{-1}(1+\delta)+\ep (cT+1). 
\end{equation*}
Using the above inequality and applying Hoeffding's inequality we have
\begin{align*}
	&\PP\left(\frac{1}{R}\sum_{i=1}^{R}Y^i_{\tau_{l^*}^i\wedge T}\geq c_Y+cg^{-1}(1+\delta)+\ep (cT+1)\right)\\
	 &\qquad\qquad\qquad\leq\PP\left(\frac{1}{R}\sum_{i=1}^{R}Y^i_{\tau_{l^*}^i\wedge T}-\E[Y_{\tau_{l^*}\wedge T}]\geq c_Y\right)\leq \exp\left(-\frac{2Rc_Y^2}{(cT+1)^2}\right).
\end{align*}
\end{proof}
Corollary~\ref{cor:stat_1-N} shows that the empirical risk of the optimal signature stopping rule concentrates around its expectation at an exponential rate in the number of samples $R$, thus confirming the practical learnability of the stopping coefficient $l^\ast$.
\subsection{Repeated experiments with a common change-point}\label{sec:det-theta}

We now consider the setting in which $R$ independent replications of the experiment are available, all corresponding to the same realization of the change-point $\theta$. This is consistent with $\theta$ being random: recalling the product structure $\mathbb{P} = \mu_1 \otimes \mu_2 \otimes \nu$ of the probability space in 
Section~\ref{sec:QD-rough}, one conditions on a fixed realization of $\theta$ and draws $R$ independent pairs $(X^{(1),i}, X^{(2),i})$ from $\mu_1 \otimes \mu_2$, yielding $R$ independent observed paths $X^1, \ldots, X^R$ each constructed via~\eqref{eq:X} with the same change-point $\theta$. By aggregating information across replications, one can reduce both detection delay and false alarm probability simultaneously, at rates exponential in $R$.

Let $\theta$ be a fixed realization of the change-point, and let $X^1, \ldots, X^R$ be $R$ independent observed paths constructed as above. Denote the time-augmented signature of the $i$-th replication by $\widehat{X}^{i,<\infty}_{0,t}$.
We define the process 
$$
S_t=\1_{\{\sup_{s\in[0,t]}|\langle l^{(1)},\widehat{\X}^{<\infty}_{0,s}\rangle|>1\}}
$$ 
and its empirical average 
$$
\bar{S}_t^R=\frac{1}{R}\sum_{i=1}^R \1_{\{\sup_{s\in[0,t]} |\langle l^{(1)},\X_s^{i,<\infty}\rangle|>1\}}.
$$ 
For a constant $\alpha>0$ we consider the stopping time
\begin{equation}\label{eq:tilde-tau1}
		\tilde{\tau}=\inf \left\{t: \bar{S}_t^R>\alpha\right\}.
\end{equation}
\begin{proposition}\label{prop:stat_2}
Under the setting of Proposition~\ref{prop:l1-X-de-ep}, choose $\delta < 1$ and $\varepsilon \in (0,1/2)$. Let $\theta$ be any fixed realization of the change-point in $[0,T]$, and let $X^1,\ldots,X^R$ be $R$ independent observed paths sharing this common $\theta$. Then for $\varepsilon < \alpha < 1-\varepsilon$, the delay in detection for the stopping time $\tilde{\tau}$ 
defined in \eqref{eq:tilde-tau1} satisfies
\begin{equation}\label{eq:P-tilde-tau-theta1}
		\PP\left((\tilde{\tau}-\theta)^+\leq g^{-1}(1+\delta)\right)\geq1-\exp\left(-2R(1-\ep-\al)^2\right),
\end{equation}
and the false alarm probability
\begin{equation}\label{eq:false-P1}
		\PP\left(\tilde{\tau}<\theta\right)\leq \exp(-2R(\al-\ep)^2).
\end{equation}
Furthermore, the expected delay and error can be bounded by 
$$\E[(\tilde{\tau}-\theta)^+]\leq g^{-1}(1+\delta)+T\exp\left(-2R(1-\ep-\al)^2\right).$$
and
\begin{equation*}
	\E|\tilde{\tau}-\theta|\leq g^{-1}(1+\delta)+T{\left[\exp\left(-2R(1-\ep-\al)^2\right)+\exp(-2R(\al-\ep)^2)\right]}.
\end{equation*}
\end{proposition}
\begin{proof}
We first prove for any deterministic $\theta$ in $[0,T]$. Let us define 
$$t_*=\theta+g^{-1}(1+\delta),$$
or equivalently
$$g\left((t_*-\theta)^{+}\right)=1+\delta.$$
Since $\delta<1$, by Proposition~\ref{prop:l1-X-de-ep} we have
\begin{equation}\label{eq:E-St*1}
	\E[S_{t_*}]\geq 1-\ep.
\end{equation}
From the definition of $\tilde{\tau}$ in \eqref{eq:tilde-tau1}, we know that
$$\{\bar{S}_{t_*}^R>\alpha\}\subseteq \{\tilde{\tau}\leq t_*\}.$$
Therefore 
\begin{equation}\label{eq:P-tau-t*-pre1}
	\PP\left(\tilde{\tau}> t_*\right)\leq \PP\left(\bar{S}_{t_*}^R\leq \alpha\right)=\PP\left(\E[S_{t_*}]-\bar{S}_{t_*}^R\geq \E[S_{t_*}]-\al\right).
\end{equation}
Using \eqref{eq:E-St*1} in \eqref{eq:P-tau-t*-pre1} and applying Hoeffding's inequality, we have
\begin{equation*}
	\PP\left(\tilde{\tau}> t_*\right)\leq \PP\left(\E[S_{t_*}]-\bar{S}_{t_*}^R\geq1-\ep-\al\right)\leq \exp\left(-2R(1-\ep-\al)^2\right).
\end{equation*}
Taking the complement we obtain \eqref{eq:P-tilde-tau-theta1}. Since $\bar{S}_{t}^R$ is non-decreasing in $t$ we have
$$ \{\tilde{\tau}<\theta\}\subseteq\{\bar{S}_{\theta}^R>\alpha\}.$$
Therefore
\begin{equation}\label{eq:P-tau-t*-pre11}
	\PP\left(\tilde{\tau}< \theta\right)\leq \PP\left(\bar{S}_{t_*}^R>\alpha\right)=\PP\left(\bar{S}_{\theta}^R-\E[S_{\theta}]\geq \al-\E[S_{\theta}]\right).
\end{equation}
Since $\X$ coincides with $\X^{(1)}$ for $t\leq\theta$, by \eqref{eq:P-l1-X1} we have
\begin{equation}\label{eq:ES_theta}
	\E[S_{\theta}]=\PP\left(\sup_{t\leq\theta}|\langle l^{(1)},\widehat{\X}^{(1),<\infty}_{0,t}\rangle| >1\right)\leq \ep
\end{equation}
Using \eqref{eq:ES_theta} in \eqref{eq:P-tau-t*-pre11} and applying Hoeffding's inequality, we have
\begin{equation*}
	\PP\left(\tilde{\tau}< \theta\right)\leq\PP\left(\bar{S}_{\theta}^R-\E[S_{\theta}]\geq \al-\ep\right)\leq \exp(-2R(\al-\ep)^2).
\end{equation*}
This give us the false alarm probability bound \eqref{eq:false-P1}. 

Considering the event that $(\tilde{\tau}-\theta)^+\leq g^{-1}(1+\delta)$ and its complement, by \eqref{eq:P-tilde-tau-theta1} and $(\tilde{\tau}-\theta)^+\leq T$ we have
$$\E[(\tilde{\tau}-\theta)^+]\leq g^{-1}(1+\delta)+T\exp\left(-2R(1-\ep-\al)^2\right).$$
Similarly, since $|\tilde{\tau}-\theta|\leq g^{-1}(1+\delta)$ on the event $\{\tilde{\tau}\geq\theta\}\cap\{\tilde{\tau}\leq t_*\}$ and $|\tilde{\tau}-\theta|\leq T$ otherwise, by \eqref{eq:P-tilde-tau-theta1} and \eqref{eq:false-P1} we obtain
\begin{equation*}
	\E|\tilde{\tau}-\theta|\leq g^{-1}(1+\delta)+T{\left[\exp\left(-2R(1-\ep-\al)^2\right)+\exp(-2R(\al-\ep)^2)\right]}.
\end{equation*}
Since the bounds hold for every fixed $\theta \in [0,T]$, they hold unconditionally as well.
This completes the proof.
	\end{proof}

Proposition~\ref{prop:stat_2} shows that, conditionally on any fixed realization of $\theta$, aggregating $R$ independent replications drives both the false alarm probability and the detection delay bound to their theoretical minimum at rates exponential in $R$.The empirical behavior of $\tilde{\tau}$ is examined in Section~\ref{sec:numeric}.

\section{Adversarial quickest detection}\label{sec:adv}
The framework of Sections~\ref{sec:QD-rough} and~\ref{sec:detector} assumes that the laws $\mu_1$ and $\mu_2$ of the pre-change and post-change processes are precisely known. In practice, however, the true 
data-generating distributions are rarely specified exactly, and the detection rule may be required to perform well across a range of plausible models. 

\subsection{Robust formulation}
We now formulate an adversarial or distributionally robust version of the quickest detection problem in which both the pre-change and post-change models are allowed to range over prescribed uncertainty classes $\mathcal{P}_1$ and $\mathcal{P}_2$.
The robust quickest detection problem is then formulated as a minimax problem: one seeks a stopping rule that performs well against the worst admissible pair of pre-change and post-change laws.

More precisely, for a stopping time $\tau$ adapted to the filtration generated by $\X$, we consider
$$
J^{\mathrm{rob}}:=\inf_{\tau\in\mathcal{S}}\sup_{P_1\in\mathcal{P}_1,\;P_2\in\mathcal{P}_2}\mathbb{E}^{P_1,P_2}\bigl[Y_{\tau\wedge T}\bigr],
$$
where $Y$ is a prescribed loss process. Here $\mathbb{E}^{P_1,P_2}[\cdot]$ denotes expectation under the measure $P_1 \otimes P_2 \otimes \nu$ on the product space $\bar{\Omega}$, where $\nu$ is the fixed prior on the change-point $\theta$. A natural choice for $Y$ is the loss process $Y$ defined in~\eqref{eq:Y1}, which penalizes both false alarms and delayed detection. In this case the above minimax problem may be written equivalently as
$$
J^{\mathrm{rob}}=\inf_{\tau\in\mathcal{S}}\sup_{P_1\in\mathcal{P}_1,\;P_2\in\mathcal{P}_2}\left\{\mathbb{P}^{P_1,P_2}(\tau<\theta)+c\,\mathbb{E}^{P_1,P_2}\bigl[(\tau-\theta)^+\bigr]\right\}.
$$
Thus, the robust objective balances false alarm probability and detection delay under the least favorable admissible model pair. A similar adaptation for $Y$ defined as in \eqref{eq:Y2} is:
$$
J^{\mathrm{rob}}=\inf_{\tau\in\mathcal{S}}\sup_{P_1\in\mathcal{P}_1,\;P_2\in\mathcal{P}_2}\left\{a\,\mathbb{E}^{P_1,P_2}\bigl[(\tau-\theta)^-\bigr]+b\,\mathbb{E}^{P_1,P_2}\bigl[(\tau-\theta)^+\bigr]\right\}.
$$
From Proposition~\ref{thm:signature-stopping} we know that the optimal stopping policy for a single model of pre- and post-change laws (non-robust problem) is the hitting time of linear signature of a half space:
$$
\tau_l:=\inf\left\{t\in[0,T]:\left|\langle l,\widehat{\X}_{0,t}^{<\infty}\rangle\right|\geq 1\right\}.
$$
\subsection{Reduction to the least favorable model}
\begin{assumption}\label{ass:LFD}
There exists a least favorable pair of laws 
$(P_1^*, P_2^*2) \in \mathcal{P}_1 \times \mathcal{P}_2$ such that the minimax and min problems coincide:
$$
\inf_{\tau\in\mathcal{S}}\sup_{P_1\in\mathcal{P}_1, 
P_2\in\mathcal{P}_2}\mathbb{E}^{P_1,P_2}[Y_{\tau\wedge T}] = \inf_{\tau\in\mathcal{S}}\mathbb{E}^{P^*_1,P^*_2}[Y_{\tau\wedge T}].
$$
\end{assumption}
Assumption~\ref{ass:LFD} is the rough path analogue of the classical least favorable distribution condition in robust hypothesis testing \cite{Unnikrishnan2011}. It holds, for instance, when $\mathcal{P}_1$ and $\mathcal{P}_2$ are weakly compact and the functional $(P_1, P_2) \mapsto \inf_{\tau}\mathbb{E}^{P_1,P_2}[Y_{\tau\wedge T}]$ is upper semicontinuous, as is the case for the Wasserstein ambiguity sets of Example~\ref{ex:wass-ambiguity}.

Under Assumption~\ref{ass:LFD}, the robust problem reduces to an ordinary optimal stopping problem under the least favorable model $(P_1^*, P_2^*)$. Applying Proposition~\ref{thm:signature-stopping} to this reduced problem yields the robust signature-based stopping policy:
$$
\inf_{l \in T((\R^{1 + d})^*) }\mathbb{E}^{P_1^*,P_2^*}\bigl[Y_{\tau_l\wedge T}\bigr]=\inf_{\tau\in\mathcal{S}}\sup_{P_1\in\mathcal{P}_1,\;P_2\in\mathcal{P}_2}\mathbb{E}^{P_1,P_2}\bigl[Y_{\tau\wedge T}\bigr].
$$
In this form, the stopping policy is chosen so as to minimize the worst case detection risk over all admissible pre-change and post-change laws. 

\begin{remark}
When Assumption~\ref{ass:LFD} fails, a saddle point for the minimax problem need not exist, and the optimal robust stopping rule need not take the form of a signature half-space hitting time. Nevertheless, the class of stopping rules $\{\tau_l : l \in T((\mathbb{R}^{1+d})^*)\}$ remains a natural and tractable family of candidate policies: it is rich enough to approximate any continuous stopping policy by Lemma~\ref{lem:sig_dense}, and amenable to numerical optimization by the zeroth-order methods of Section~\ref{sec:numeric}.
\end{remark}
\subsection{Examples of uncertainty classes}
We now describe two concrete instances of the uncertainty 
classes $\mathcal{P}_1$ and $\mathcal{P}_2$. The first 
uses the homogeneous $p$-variation distance on rough path 
space to define Wasserstein-type ambiguity sets, providing a quantitative distance between rough path laws that is intrinsic to the geometry of the space.
\begin{example}\label{ex:wass-ambiguity}
	In the present work, we keep the classes $\mathcal{P}_1$ and $\mathcal{P}_2$ general. This allows the robust formulation to cover a broad range of model uncertainty without committing to a particular metric structure. However, we want to mention that a more structured formulation is also possible. For instance, one may specify $\mathcal{P}_1$ and $\mathcal{P}_2$ as Wasserstein-type ambiguity sets around nominal pre-change and post-change laws, where the transportation cost is induced by the homogeneous $p$-variation. More precisely, define the Wasserstein distance of order $s$ to be:
	$$
	W_{s}^{CC}(P,Q):=\inf_{\pi\in\Pi(P,Q)}\int d_{p-\mathrm{var};[0,T]}(\mathbf{x},\mathbf{y})^s\pi(d\mathbf{x},d\mathbf{y}).
	$$
    Here the cost $d_{p\text{-var};[0,T]}(x,y)$ is the homogeneous $p$-variation distance of Definition~\ref{def:d-p-var-CC}, which metrizes the rough path topology and is therefore the natural choice of transport cost on $\Omega^p_T$.
	Then one can define the ambiguity sets for the pair of law $(P_{1}^{\circ},P_{2}^{\circ})$ to be:
	$$
	\mathcal{P}_{1}=\{P:W_{s}^{CC}(P,P_{1}^{\circ})\leq c_{1}\},\quad\mathcal{P}_{2}=\{P:W_{s}^{CC}(P,P_{2}^{\circ})\leq c_{2}\}.
	$$
	Such a construction provides a quantitative notion of closeness between rough path laws. 
\end{example}

\begin{example}
We consider a setting in which the pre-change model is precisely known but the post-change drift is uncertain, a common situation in signal processing and finance where the noise structure is well characterized but the magnitude of a potential shift is not. A simple example is obtained by taking the pre-change model class to be a singleton and the post-change model class to consist of drift perturbations of the same noise. For Brownian motion, we let
	$$
	\mathcal{P}_1=\{P^{(0)}\},\qquad\mathcal{P}_2=\{P^{(r)}: r\in\mathcal{R}\},
	$$
	where $P^{(0)}$ is the law of
	$$
	X_t^{(1)}=\sigma B_t,
	$$
	and $P^{(r)}$ is the law of
	$$
	X_t^{(2,r)}=rt+\sigma B_t.
	$$
	Here $\mathcal{R}$ is a prescribed set of admissible drift values. Thus, the pre-change model is fixed, whereas the post-change model is uncertain through the drift parameter $r$. The corresponding rough paths are given by the lifts of $X^{(1)}$ and $X^{(2,r)}$, and the observed rough path is obtained by concatenating the pre-change lift with the post-change lift at time $\theta$ as in \eqref{eq:X}.
	
	An analogous example may be formulated with fractional Brownian motion in place of Brownian motion. In that case, one considers
	$$
	X_t^{(1)}=\sigma B_t^H,
	\qquad
	X_t^{(2,r)}=rt+\sigma B_t^H,
	$$
	together with the corresponding rough path lifts. This again yields a robust quickest detection problem in which the post-change regime is uncertain through the drift parameter, while the driving noise remains unchanged.
\end{example}

\begin{example}\label{ex:adversarial-TV}
A practically important instance of model uncertainty arises from adversarial path perturbations. Let $X^{(1)}$ and $X^{(2)}$ be the pre-change and post-change processes, and let the observed path be $Z_t = X_t + w_t$, where $X$ is the concatenated process defined in~\eqref{eq:X} and $w:[0,T]\to\mathbb{R}^d$ is an adversarial perturbation with $w_0=0$. We model uncertainty in the post-change regime by taking the pre-change class to be a singleton $\mathcal{P}_1 = \{P^{(0)}\}$, the law of the unperturbed pre-change path $X^{(1)}$, and the post-change class to be
$$
\mathcal{P}_2 = \left\{P^{(w)} : w \in \mathcal{W}\right\},
$$
where $P^{(w)}$ is the law of $X^{(2)} + w$ and $\mathcal{W}$ is a prescribed class of admissible perturbations. A natural and tractable choice is the total variation budget class
$$
\mathcal{W} := \left\{w:[0,T]\to\mathbb{R}^d : 
w_0 = 0,\; \|w\|_{TV;[0,T]} \leq C_{TV}\right\},
$$
for a fixed constant $C_{TV} > 0$. Every $w\in\mathcal{W}$ has bounded variation and therefore admits a canonical geometric rough path lift, so that the perturbed post-change path $X^{(2)} + w$ remains in $\Omega^p_T$ for $p \in (2,3)$. The robust quickest detection problem then becomes
$$
J^{rob} = \inf_{\tau\in\mathcal{S}}\sup_{w\in\mathcal{W}} \mathbb{E}^w\left[Y_{\tau\wedge T}\right].
$$
\end{example}

\section{Numerical experiments}\label{sec:numeric}
We now describe the numerical implementation of the signature-based stopping rules developed in Sections~\ref{sec:QD-rough}-\ref{sec:adv} and report numerical experiments for both the nominal and adversarial quickest detection problems. Throughout, we work with the model of Example~\ref{ex:BM} and Example~\ref{ex:fBM}, %driven by fractional Brownian motion with Hurst parameter $H \in (1/3,1/2)$, 
and optimize the signature coefficient $l$ via zeroth-order methods. We compare the proposed stopping rules against classical baselines and examine the trade-off between detection delay and false alarm probability as the loss parameters are varied.
\subsection{Time discretization and loss function}
For numerical implementation, we discretize the time interval $[0,T]$ on a uniform grid
$$
0=t_0<t_1<\cdots<t_K=T, \qquad t_k=k\Delta t, \qquad \Delta t=T/K.
$$
Throughout this section, we retain for brevity $X$, $\hat{X}$, and $Y$ for the piecewise linear interpolations of the respective processes on this grid. We also use the shorthand
$$
X_k := X_{t_k},
$$
and similarly for $\widehat{X}_k$ and $Y_k$. By Proposition~\ref{thm:signature-stopping}, it is sufficient to restrict attention to the class of linear signature stopping policies introduced in Definition~\ref{def:lin-sig}. In practice, we compute the signature up to a finite truncation level $N$. For a signature coefficient $l \in T^{(N)}((\mathbb{R}^{1+d})^*)$, we define the discretized and truncated stopping rule by
\begin{equation}\label{eq:discrete_stopping_time}
	\tau_l^K := \inf\left\{0\le k\le K \;\middle|\; |\langle l, \widehat{\mathbb{X}}^{\le N}_{[0,t_k]} \rangle| \ge 1 \right\} \wedge K,
\end{equation}
where $\inf\varnothing := \infty$ by convention, capped at the terminal index $K$. The unsigned threshold $|\langle l, \hat{X}^{\leq N}_{[0,t_k]}\rangle| \geq 1$ 
in~\eqref{eq:discrete_stopping_time} is consistent with Definition~\ref{def:lin-sig}: in 
practice one may also use the one-sided rule $\tau^K_l := \inf\{k : \langle l, \hat{X}^{\leq N}_{[0,t_k]}\rangle \geq 1\}\wedge K$, which is appropriate when the sign of the post-change drift is known a priori. Let
$$
\bigl(\widehat{X}_k^{m},Y_k^{m}\bigr)_{0\le k\le K,\;1\le m\le M},\qquad \text{where } \widehat{X}_k^{m}\in\mathbb{R}^{d+1}, \quad Y_k^{m}\in\mathbb{R},
$$
denote a batch of $M$ simulated samples, with the $m$-th sample associated with change-point $\theta^m$. For each sample path, we compute the truncated signature on $[0,t_k]$ for every $k=0,\dots,K$. We denote these truncated signatures by
$$
\bigl(\widehat{\mathbb{X}}_k^{m}\bigr)_{0\le k\le K,\;1\le m\le M} := \bigl(\widehat{\mathbb{X}}^{m,\le N}_{[0,t_k]}\bigr)_{0\le k\le K,\;1\le m\le M}.
$$
We then use the averaged empirical risk as the training objective, namely
\begin{equation}\label{eq:emp-risk}
	 J(l) := \frac{1}{M}\sum_{m=1}^M Y^m_{\tau_l^K},
\end{equation}
where the optimization is over the signature coefficient $l \in T^{(N)}((\mathbb{R}^{1+d})^*)$, whose dimension grows as $\sum_{k=0}^{N}(1+d)^k$ with the truncation 
level $N$. The non-smoothness of $J(l)$ with respect to $l$, arising from the discrete threshold crossing in~\eqref{eq:discrete_stopping_time}, necessitates the zeroth-order optimization methods described in Section~\ref{sec:zo}.
\subsection{Simulation of signature}
To compute the truncated signatures of our simulated paths, we adopt the standard procedure of calculating the exact signature of their piecewise linear interpolations. For our numerical experiments, this involves generating joint discrete-time trajectories of the time-augmented process $\widehat{X}$ and the corresponding variable $Y$ on our grid. The exact truncated signature of this linearly interpolated path can then be reliably computed (up to standard floating-point precision) using the \texttt{iisignature} library in Python \cite{iisignature}. A key practical advantage for our stopping rule is that we require the signature at every discrete time step $t_k$. Instead of recalculating the signature from scratch for each subinterval $[0, t_k]$, we can significantly reduce computational overhead by incrementally updating the signature from the preceding interval $[0, t_{k-1}]$ using the path increment over $[t_{k-1}, t_k]$. The \texttt{iisignature} package handles these sequential updates highly efficiently. 

With the signature computation in place, we now describe the zeroth-order optimization procedure used to learn the stopping coefficient $\hat{l}$.
\subsection{Zeroth-order method}\label{sec:zo}
The objective $J(l)$ defined in~\eqref{eq:emp-risk} is non-differentiable with respect to $l$, since $\tau^K_l$ is defined through a discrete threshold crossing that introduces discontinuities. Consequently, exact gradients are analytically unavailable, rendering standard gradient descent methods inapplicable. This necessitates the use of zeroth-order (ZO) optimization. Let $n = \dim(T^{(N)}((\mathbb{R}^{1+d})^*))$ denote the dimension of the parameter space.

A foundational approach to derivative-free optimization is the ZO method based on Gaussian smoothing, extensively analyzed by Nesterov and Spokoiny \cite{Nesterov2015}. Instead of optimizing the non-smooth $J(l)$ directly, this method minimizes a smoothed surrogate $J_{\eta}(l) = \mathbb{E}_u[J(l + \eta u)]$, where $u$ is a realization of random variable $U \sim N(0, I_n)$ and $\eta > 0$ is a smoothing parameter. The gradient is typically estimated via random two-point evaluations:
\begin{equation*}
	\hat{g}_{\eta}(l) = \frac{J(l + \eta u) - J(l - \eta u)}{2\eta} u.
\end{equation*}
This estimator is straightforward to implement and requires only two function evaluations per iteration. However, its second moment scales quadratically with the parameter dimension $n$, yielding a worst-case iteration complexity of $\mathcal{O}(n^2\epsilon^{-2})$. For moderate or large truncation levels $N$, where $n = \sum_{k=0}^{N}(1+d)^k$ grows rapidly, this dimension dependence is computationally prohibitive.

An alternative ZO method is to employ the exponentially-shifted Gaussian smoothing (esGS) estimator proposed by \cite{Wang2024}. For random variables $V \sim \text{Exp}(1)$ and $Z \sim N(0, \eta^2 I_n)$ with realization $v,z$, the esGS gradient $\hat{g}_\eta=(\hat{g}_{\eta}^1,\cdots,\hat{g}_{\eta}^n)$ is estimated coordinate-wise:
\begin{equation}\label{eq:esgs}
	\hat{g}_{\eta}^i(l, v, z) = \frac{1}{\eta\sqrt{2\pi}} \left[ J(l_i + \eta\sqrt{2v}, l^{-i} - z^{-i}) - J(l_i - \eta\sqrt{2v}, l^{-i} - z^{-i}) \right],
\end{equation}
where $l_i$ denotes the $i$-th component of $l$ and $l^{-i}$ denotes the remaining components (and $z^{-i}$ likewise).By shifting the evaluation points via exponential random variables, the esGS estimator reduces the second moment bound to $\mathcal{O}(L_0^2n)$. This structural modification significantly improves the overall ZO iteration complexity to $\mathcal{O}(n\epsilon^{-2})$. 

Other zeroth-order techniques exist in the literature, such as Simultaneous Perturbation Stochastic Approximation (SPSA) \cite{Spall1992}, which relies on finite differences without convolution based smoothing, or methods leveraging spherical smoothing \cite{Cui2022}. Comparing their empirical performance, we restrict our focus to the two Gaussian smoothing-based methods discussed above.

In all experiments below, we use the esGS estimator on account of its improved dimension dependence. We now report numerical results, beginning with the ordinary quickest detection problem.

\subsection{Comparison with baseline models for SDE}
We present numerical results for the signature-based stopping rules trained by zeroth-order optimization. As esGS and standard Gaussian smoothing exhibited broadly similar empirical performance, we report only the esGS results for brevity. We implement our algorithm to the \textit{classical quickest detection problem}:
\begin{equation*}
	dX_t^{(1)}=\sigma dB_t, \quad dX_t^{(2)}=r dt +\sigma dB_t,
\end{equation*}
where we choose $\sigma=1,r=3$ and truncation level $N=4$, and the change-point $\theta\sim \text{Exp}(0.7)$. In the following Table~\ref{tab:bayes_comparison} we compare the performance of the signature-based stopping rules trained under the loss functions $Y^1$ and $Y^2$ defined in \eqref{eq:Y1} and \eqref{eq:Y2} with three benchmark procedures: GLR-CUSUM \cite{SiegmundVenkatraman1995} (unknown post-change drift), CUSUM \cite{Moustakides2004} (known model), and Shiryaev \cite{Shiryaev2010} (known model). We calibrate the thresholds for the baseline models such that $\PP(\tau<\theta)\approx 30\%$. For each policy, we report the empirical risk $\mathbb{E}[Y_{\tau\wedge T}^i]$ for $i=1,2$, the expected delay $\mathbb{E}[(\tau-\theta)^+]$, and the false alarm probability $\mathbb{P}(\tau<\theta)$. We are also interested in the false alarm probability $\mathbb{P}_{\infty}(\tau<T)$ and the truncated expected alarm time $\mathbb{E}_{\infty}[\tau\wedge T]$ under the no change scenario on the entire interval $[0,T]$. The truncation by $T$ appears because in our implementation we set $\tau=T$ whenever no alarm occurs before the terminal time.

\begin{table}[H]
	\centering
    \tiny
	\begin{tabular}{lccccccc}
		\toprule
		Model & $\mathbb{E}[Y_{\tau\wedge T}^{1}]$ & $\mathbb{E}[Y_{\tau\wedge T}^{2}]$ & $\mathbb{E}[(\tau-\theta)^+]$ & $\mathbb{P}(\tau<\theta)$ & $\mathbb{E}[\tau]$ & $\mathbb{P}_{\infty}(\tau<T)$ & $\mathbb{E}_{\infty}[\tau\wedge T]$ \\
		\midrule
		Signature ($Y^1$) & 0.5732 & 0.7049 & 0.2722 & 0.3010 & 1.2763 & 0.9917 & 2.4659 \\
		Signature ($Y^2$) & 0.5891 & 0.6740 & 0.3699 & 0.2192 & 1.5026 & 0.9287 & 3.6534 \\
		GLR-CUSUM & 0.7176 & 0.8303 & 0.4299 & 0.2878 & 1.4664 & 0.9607 & 3.1212 \\
		CUSUM & 0.5759 & 0.6946 & 0.2802 & 0.2958 & 1.3027 & 0.9683 & 2.9697 \\
		Shiryaev & 0.5559 & 0.6760 & 0.2654 & 0.2905 & 1.2917 & 0.9840 & 2.7294 \\
		\bottomrule
	\end{tabular}
	\caption{Performance of the five quickest-detection methods.}
	\label{tab:bayes_comparison}
\end{table}
Table~\ref{tab:bayes_comparison} shows that each signature-based rule performs surprisingly well under the loss function used for its training. The rule trained with $Y^1$ achieves performance comparable to known-model CUSUM and Shiryaev, whereas the rule trained with $Y^2$ attains the lowest empirical $Y^2$-risk. Shiryaev achieves the lowest $Y^1$-risk and shortest delay, as expected from its use of the known model and its theoretical optimality. By contrast, GLR-CUSUM, which does not know the post-change drift, exhibits the largest risks and detection delay.

\begin{figure}[h]
	\centering
	\includegraphics[width=0.5\linewidth]{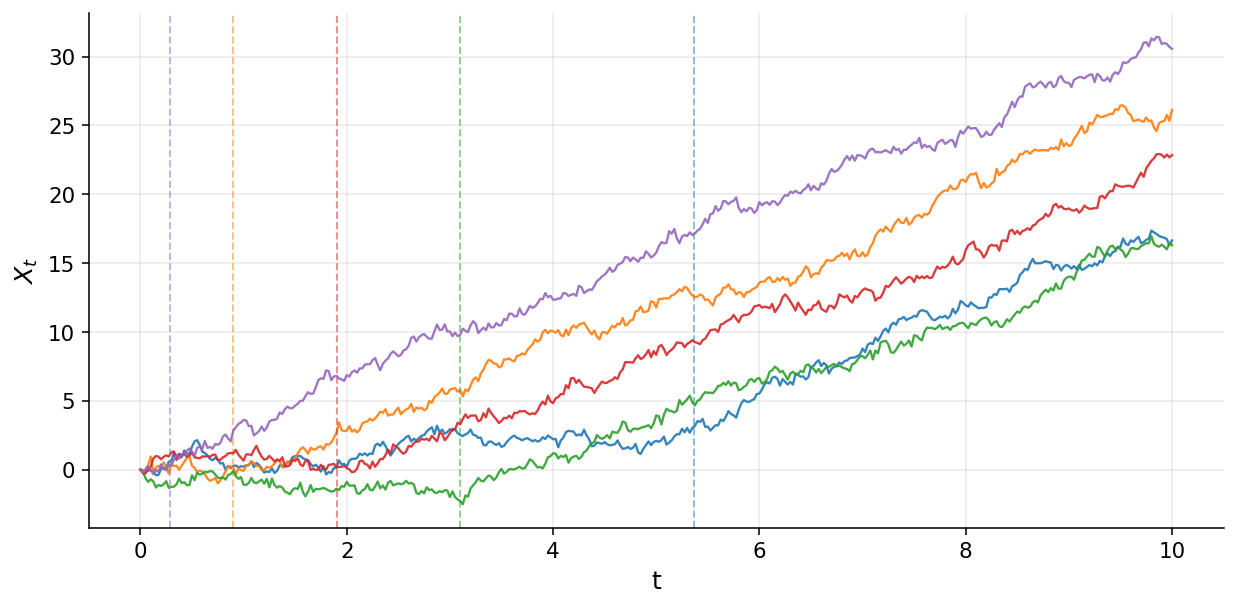}
	\caption{Path examples with $\theta\sim\text{Exp}(0.5)$}
	\label{fig:path-exp}
\end{figure}

\subsection{Quickest detection for rough paths}
Classical methods such as CUSUM and Shiryaev are most tractable when the likelihood ratio is explicitly available, as in Brownian diffusion models with known parameters. Extending quickest detection methods to rough path and general change time distributions is substantially more challenging. This work provides a way to address this gap. We implement our algorithm to the Example~\ref{ex:fBM}:
	\begin{equation*}
	dX_t^{(1)}=\sigma dB_t^H, \quad dX_t^{(2)}=r dt +\sigma dB_t^H,
\end{equation*}
where we choose $\sigma=1,r=3, H=0.35$ and truncation level $N=4$. Figure~\ref{fig:path-exp} provides some examples of the path, where the vertical dashed lines indicate the corresponding change-points $\theta$. 

 To illustrate the behavior of the learned stopping rules, we present a scatter plot of $\tau$ versus $\theta$, a histogram of $\tau-\theta$, and the empirical density of $\tau$ together with the true density of $\theta$. Figure~\ref{fig:Y1_exp} shows the results for $\theta^{\text{Exp}} \sim \mathrm{Exp}(\lambda)$ with payoff $Y^1$ given by \eqref{eq:Y1}, while Figure~\ref{fig:Y1_wei} shows the corresponding results for $\theta^{\text{Wei}} \sim \mathrm{Weibull}(1/\lambda\Gamma(1+1/k),k)$ with the same payoff, using $\lambda=0.7,k=2$ and $c=1$.
\begin{figure}[H]
	\centering
	\begin{minipage}[c]{0.23\linewidth}
		\centering
		\includegraphics[width=\linewidth]{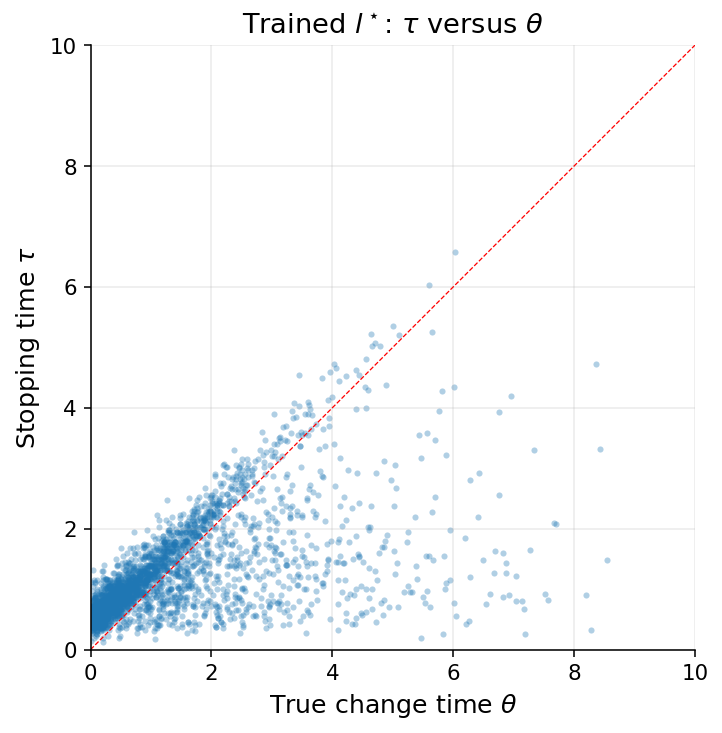}
	\end{minipage}
	\hfill
	\begin{minipage}[c]{0.36\linewidth}
		\centering
		\includegraphics[width=\linewidth]{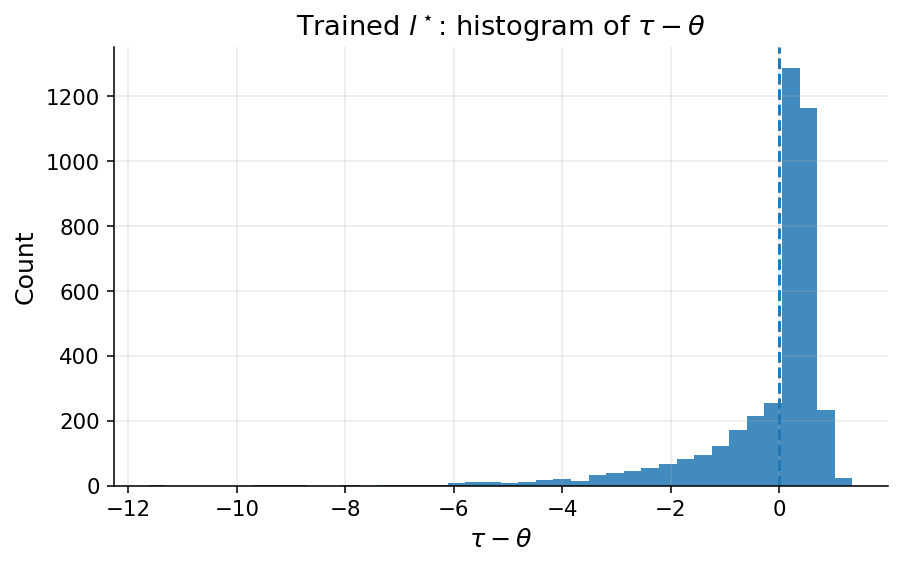}
	\end{minipage}
		\hfill
	\begin{minipage}[c]{0.36\linewidth}
		\centering
		\includegraphics[width=\linewidth]{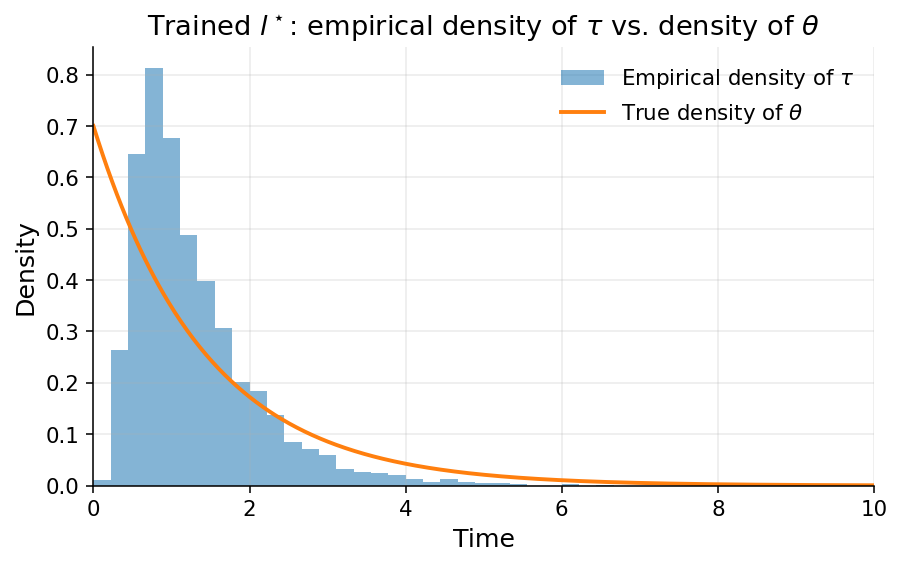}
	\end{minipage}
	\caption{Change-point $\theta^{\text{Exp}}$ and payoff $Y^1$.}
	\label{fig:Y1_exp}
\end{figure}
\begin{figure}[H]
	\centering
	\begin{minipage}[c]{0.23\linewidth}
		\centering
		\includegraphics[width=\linewidth]{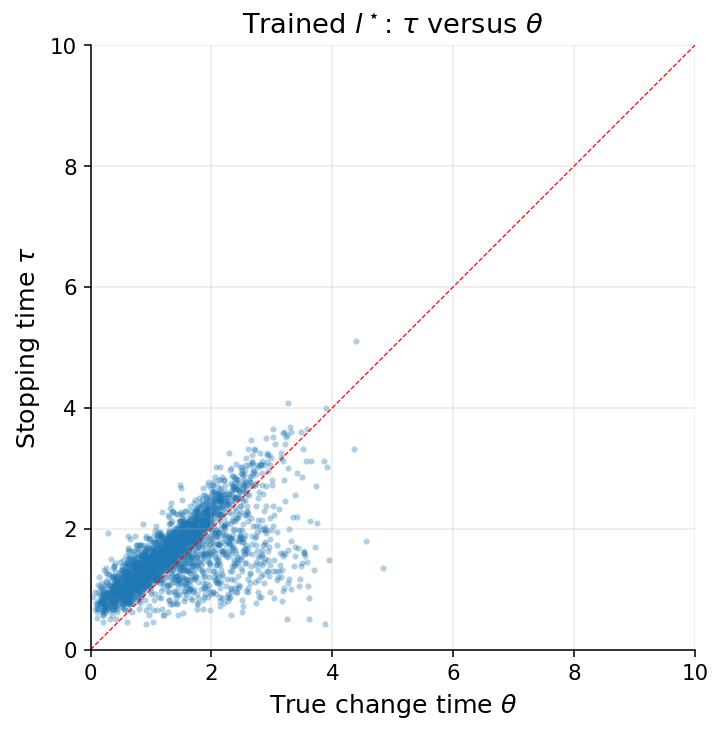}
	\end{minipage}
	\hfill
	\begin{minipage}[c]{0.36\linewidth}
		\centering
		\includegraphics[width=\linewidth]{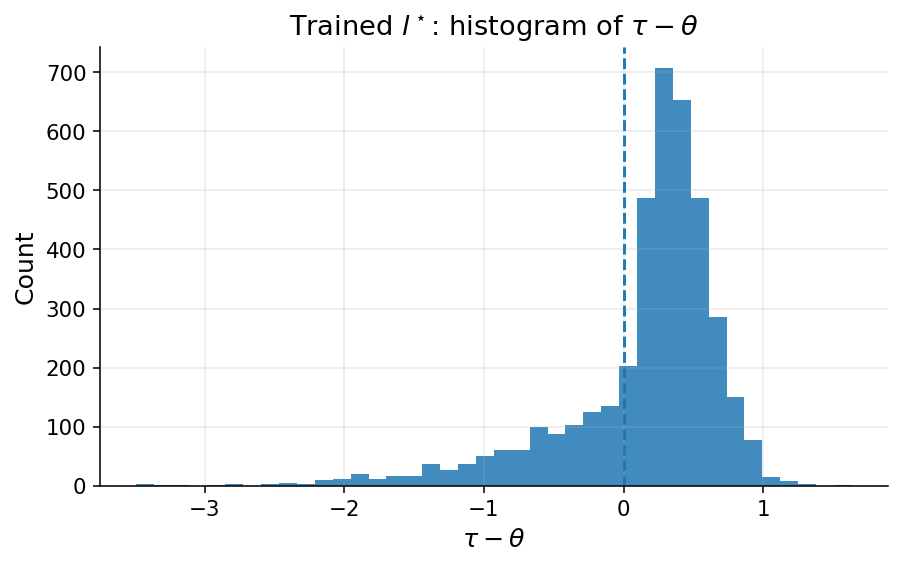}
	\end{minipage}
		\hfill
	\begin{minipage}[c]{0.36\linewidth}
		\centering
		\includegraphics[width=\linewidth]{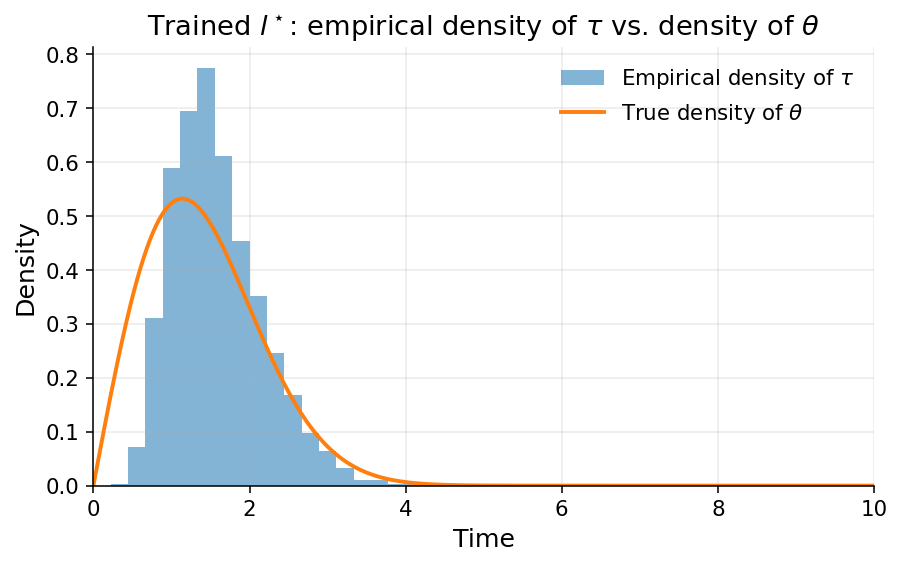}
	\end{minipage}
	\caption{Change-point $\theta^{\text{Wei}}$ and payoff $Y^1$.}
	\label{fig:Y1_wei}
\end{figure}

From Figure~\ref{fig:Y1_exp} and Figure~\ref{fig:Y1_wei}, we observe that the learned stopping rules reflect the overall structure of the change-time distribution. In both cases, the scatter plots display a clear positive association between $\tau$ and $\theta$, suggesting that the learned rule responds to later changes by stopping later on average. However, many samples fall below the diagonal line $\tau=\theta$, contributes around $\PP(\tau<\theta)\approx 30\%$ false alarm under the payoff $Y^1$ in \eqref{eq:Y1} with $c=1$. The histograms of $\tau-\theta$ further support this observation, as much of the mass is concentrated near or below zero. In addition, the empirical density of $\tau$ has a shape similar to the true density of $\theta$, though with some discrepancy due to the delayed nature of the stopping rule. Relative to the exponential case, the Weibull case appears more concentrated and exhibits less extreme variability. The following Figure~\ref{fig:Y2_exp} and Figure~\ref{fig:Y2_wei} show the similar pattern for the signature stopping rule trained by payoff $Y^2$.
\begin{figure}[H]
	\centering
	\begin{minipage}[c]{0.23\linewidth}
		\centering
		\includegraphics[width=\linewidth]{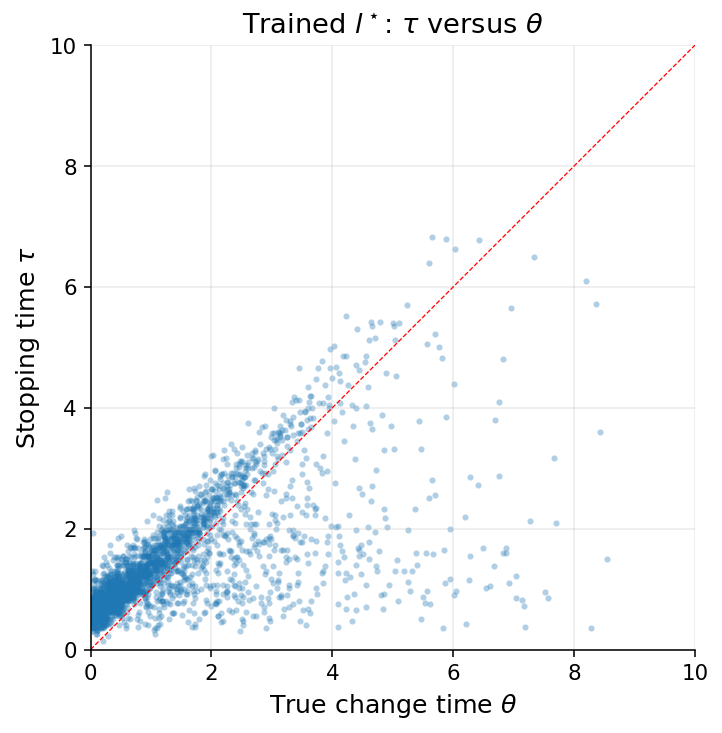}
	\end{minipage}
	\hfill
	\begin{minipage}[c]{0.36\linewidth}
		\centering
		\includegraphics[width=\linewidth]{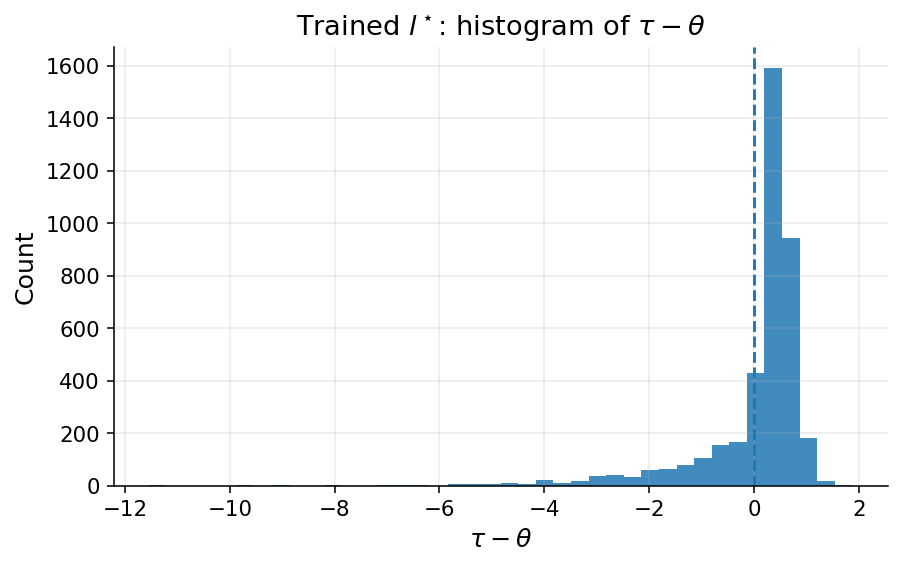}
	\end{minipage}
		\hfill
	\begin{minipage}[c]{0.36\linewidth}
		\centering
		\includegraphics[width=\linewidth]{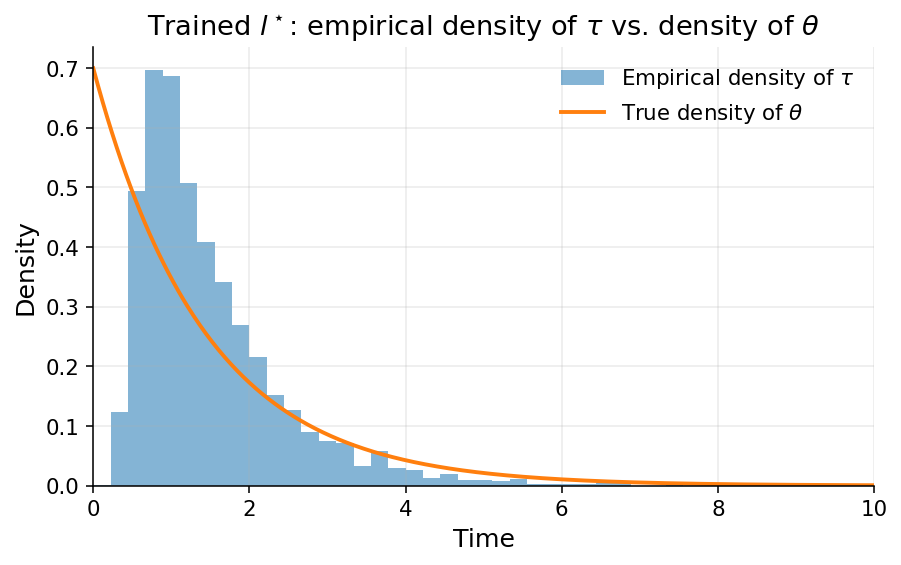}
	\end{minipage}
	\caption{Change-point $\theta^{\text{Exp}}$ and payoff $Y^2$.}
	\label{fig:Y2_exp}
\end{figure}
\begin{figure}[H]
	\centering
	\begin{minipage}[c]{0.23\linewidth}
		\centering
		\includegraphics[width=\linewidth]{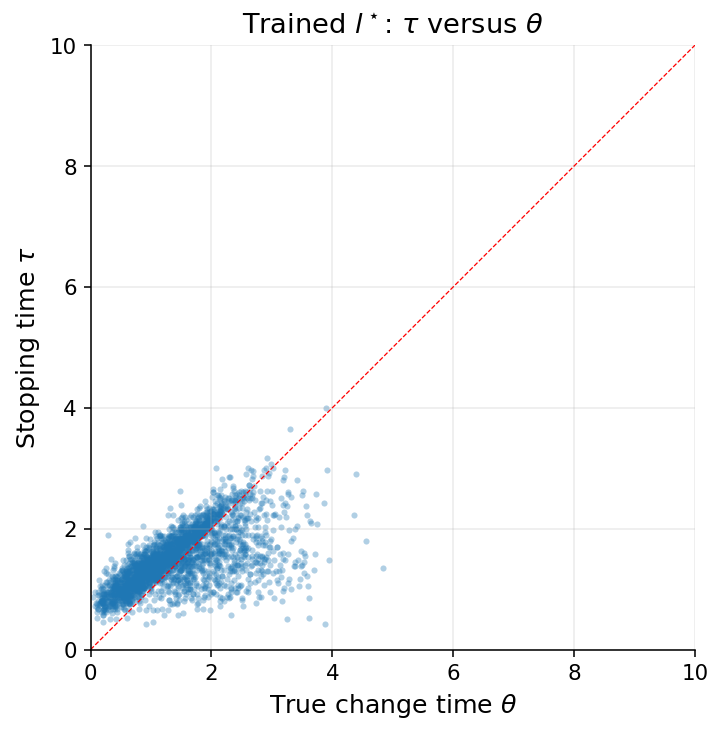}
	\end{minipage}
	\hfill
	\begin{minipage}[c]{0.36\linewidth}
		\centering
		\includegraphics[width=\linewidth]{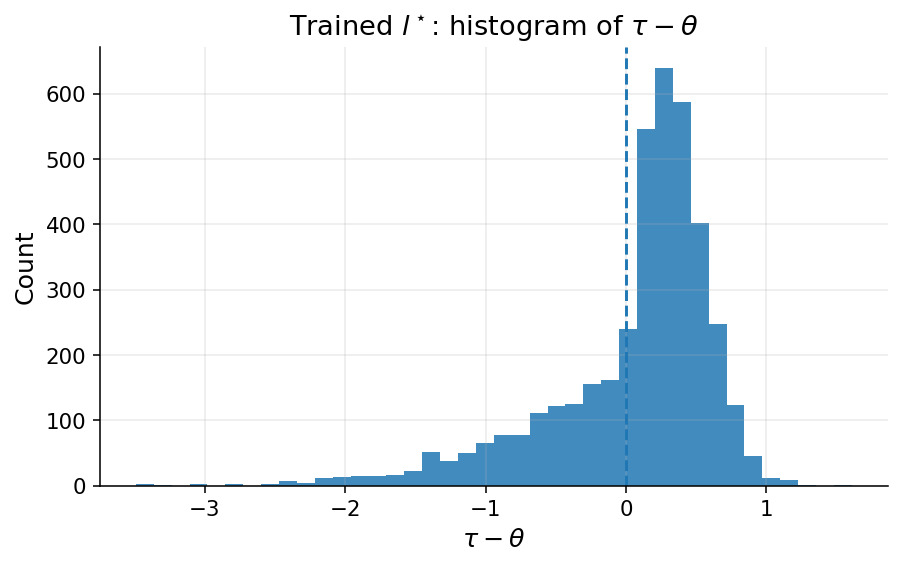}
	\end{minipage}
		\hfill
	\begin{minipage}[c]{0.36\linewidth}
		\centering
		\includegraphics[width=\linewidth]{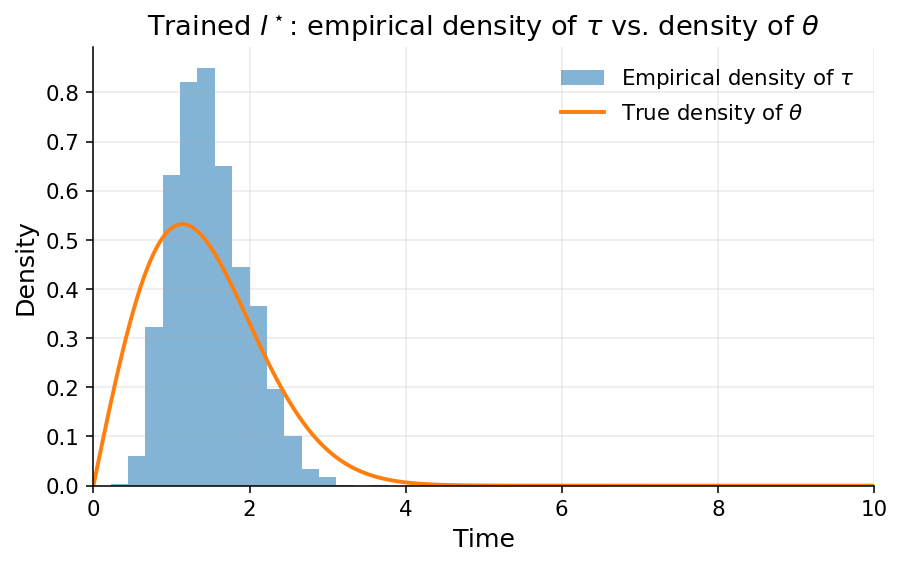}
	\end{minipage}
	\caption{Change-point $\theta^{\text{Wei}}$ and payoff $Y^2$.}
	\label{fig:Y2_wei}
\end{figure}

Since CUSUM and Shiryaev require specific structure of the model which is not available in this example, in Table~\ref{tab:cumsum_cmp} and Table~\ref{tab:cumsum_wei_cmp} we compare our stopping policies to the \emph{Page-Hinkley} detector implemented in the \texttt{River} package in Python. \texttt{River} describes this detector as implementing the CUSUM control chart for change detection. We apply it directly to the observed stream generated from $X_t$ (In our implementation, the increments of the path), so the baseline uses no latent state information and does not assume a parametric pre- and post-change model. Its detection threshold is calibrated to match the false alarm of our signature stopping rule as closely as possible. In Figure~\ref{fig:cusum} and Figure~\ref{fig:cusum_wei} we present similar plots for the Page-Hinkley detector to Figure~\ref{fig:Y1_exp}-Figure~\ref{fig:Y2_wei}. As we can see in the following comparisons, the signature-based rules retain performance comparable to that observed in the Brownian diffusion and exponential change time distribution setting, demonstrating their effectiveness for non-Markovian models. Under both exponential and Weibull change times, each signature rule performs best under its training loss and achieves a substantially shorter delay than Page-Hinkley at a comparable false alarm level. These results suggest that signature training effectively learns the relevant characteristics of the underlying process from sample paths, allowing the resulting rules to perform well without explicit knowledge of the likelihood or model parameters.
\begin{table}[htbp]
	\centering
    \tiny
	\begin{tabular}{lccccccc}
		\toprule
		Model & $\mathbb{E}[Y_{\tau\wedge T}^{1}]$ & $\mathbb{E}[Y_{\tau\wedge T}^{2}]$ & $\mathbb{E}[(\tau-\theta)^+]$ & $\mathbb{P}(\tau<\theta)$ & $\mathbb{E}[\tau]$ & $\mathbb{P}_{\infty}(\tau<T)$ & $\mathbb{E}_{\infty}[\tau\wedge T]$ \\
		\midrule
		Signature ($Y^1$) & 0.5949 & 0.7225 & 0.2809 & 0.3140 & 1.2643 & 0.9967 & 2.2597 \\
		Signature ($Y^2$) & 0.6019 & 0.6911 & 0.3794 & 0.2225 & 1.4925 & 0.9117 & 3.5827 \\
		Page-Hinkley & 1.0024 & 1.1385 & 0.6817 & 0.3208 & 1.6498 & 0.9990 & 2.0491 \\
		\bottomrule
	\end{tabular}
    \caption{Comparison of signature stopping rule and Page-Hinkley baseline with change-point $\theta^{\text{Exp}}$.}
	\label{tab:cumsum_cmp}
\end{table}

\begin{figure}[H]
	\centering
	\begin{minipage}[c]{0.23\linewidth}
		\centering
		\includegraphics[width=\linewidth]{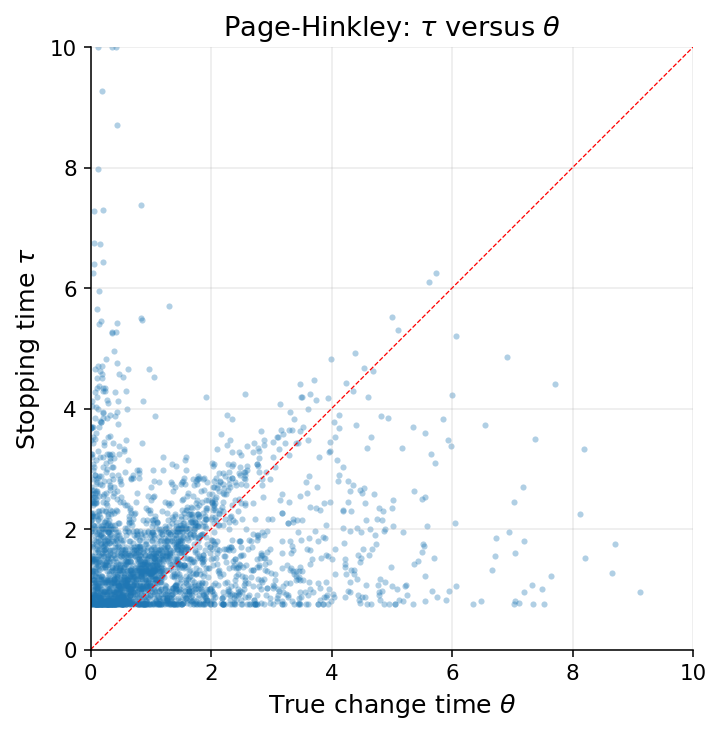}
	\end{minipage}
	\hfill
	\begin{minipage}[c]{0.36\linewidth}
		\centering
		\includegraphics[width=\linewidth]{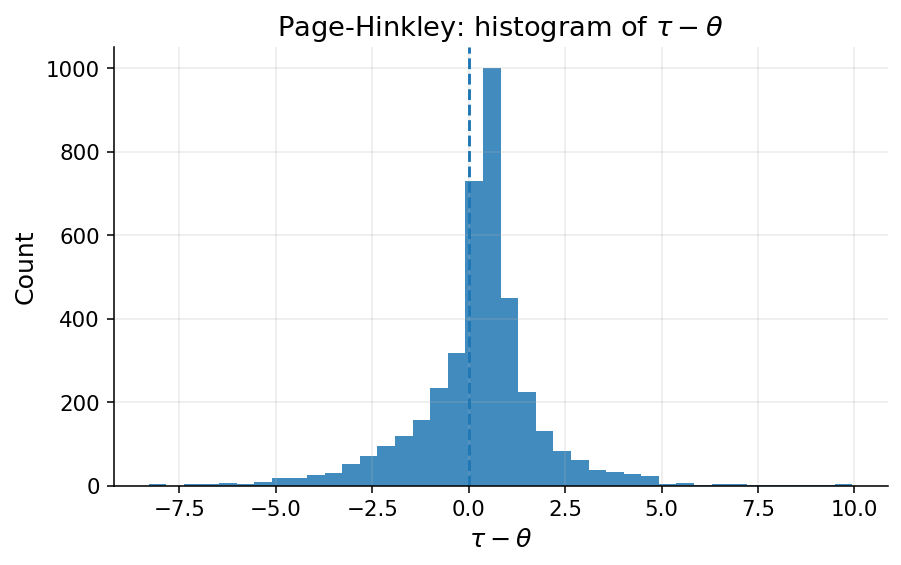}
	\end{minipage}
		\hfill
	\begin{minipage}[c]{0.36\linewidth}
		\centering
		\includegraphics[width=\linewidth]{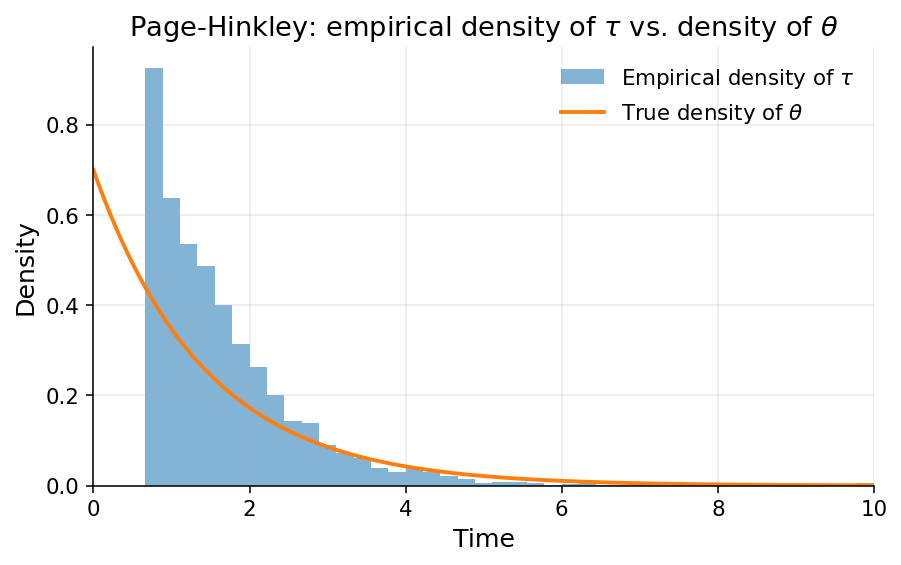}
	\end{minipage}
	\caption{Page-Hinkley quickest detection experiment with change-point $\theta^{\text{Exp}}$.}
	\label{fig:cusum}
\end{figure}

\begin{table}[htbp]
	\centering
    \tiny
	\begin{tabular}{lccccccc}
		\toprule
		Model & $\mathbb{E}[Y_{\tau\wedge T}^{1}]$ & $\mathbb{E}[Y_{\tau\wedge T}^{2}]$ & $\mathbb{E}[(\tau-\theta)^+]$ & $\mathbb{P}(\tau<\theta)$ & $\mathbb{E}[\tau]$ & $\mathbb{P}_{\infty}(\tau<T)$ & $\mathbb{E}_{\infty}[\tau\wedge T]$ \\
		\midrule
		Signature ($Y^1$) & 0.5481 & 0.4698 & 0.3169 & 0.2312 & 1.5915 & 1.0000 & 2.4296 \\
		Signature ($Y^2$) & 0.5607 & 0.4571 & 0.2612 & 0.2995 & 1.4930 & 1.0000 & 1.9353 \\
		Page-Hinkley & 0.9423 & 0.8812 & 0.7116 & 0.2308 & 1.9696 & 0.9843 & 3.1201 \\
		\bottomrule
	\end{tabular}
	\caption{Comparison of signature stopping rule and Page-Hinkley baseline with change-point $\theta^{\text{Wei}}$.}
    \label{tab:cumsum_wei_cmp}
\end{table}

\begin{figure}[H]
	\centering
	\begin{minipage}[c]{0.23\linewidth}
		\centering
		\includegraphics[width=\linewidth]{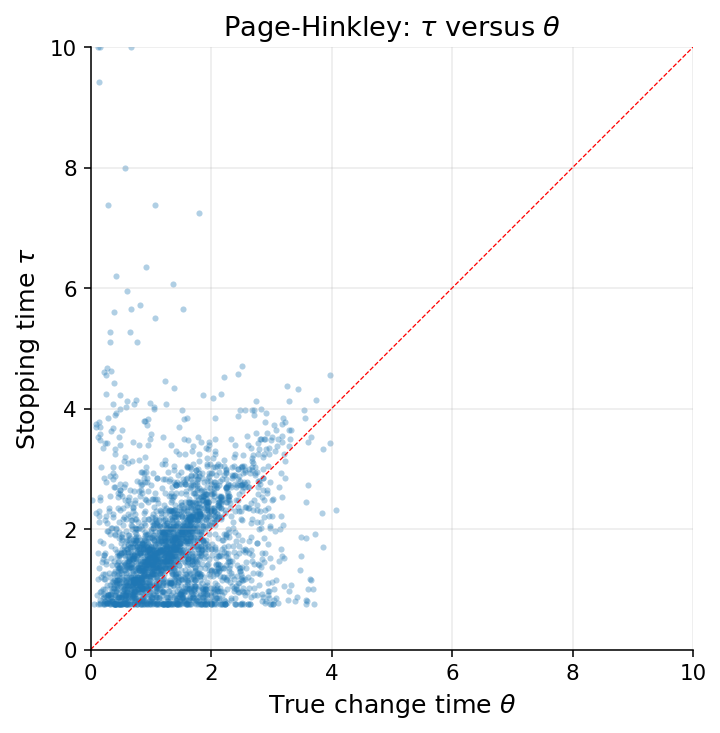}
	\end{minipage}
	\hfill
	\begin{minipage}[c]{0.36\linewidth}
		\centering
		\includegraphics[width=\linewidth]{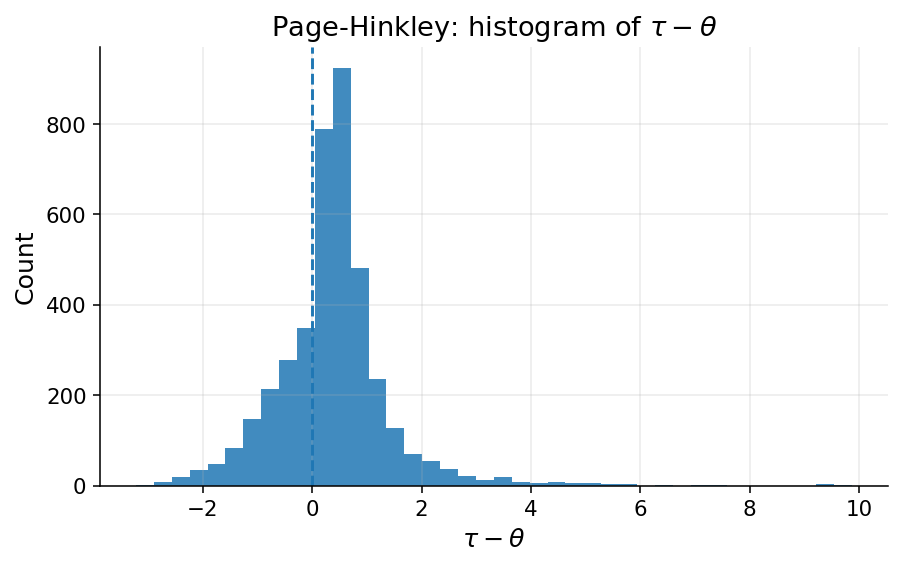}
	\end{minipage}
		\hfill
		\begin{minipage}[c]{0.36\linewidth}
		\centering
		\includegraphics[width=\linewidth]{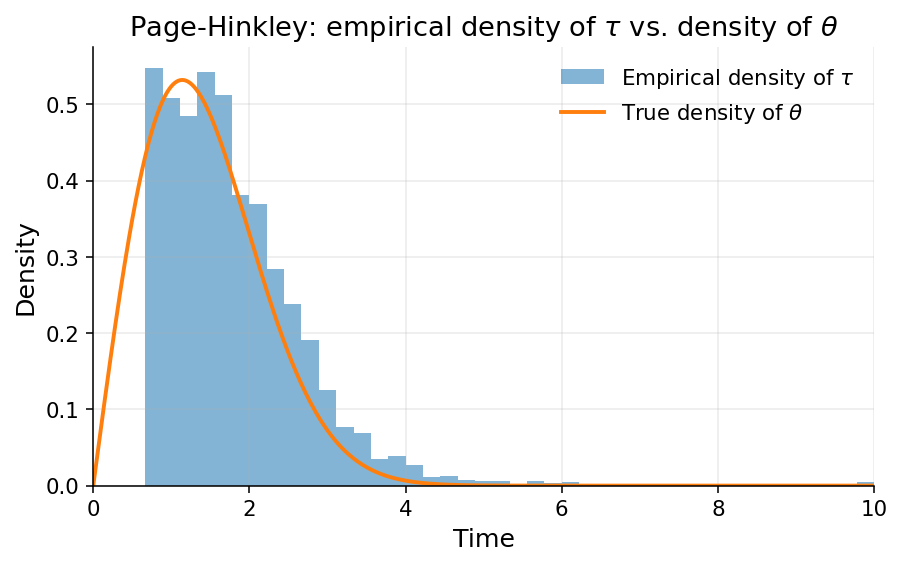}
	\end{minipage}
	\caption{Page-Hinkley quickest detection experiment with change-point $\theta^{\text{Wei}}$.}
	\label{fig:cusum_wei}
\end{figure}

For change-point $\theta^{\text{Exp}}$, we consider the two loss functions $Y^1$ and $Y^2$ defined in \eqref{eq:Y1} and \eqref{eq:Y2}, and examine how tuning the parameters in the loss affects the trade-off between expected detection delay and false alarm probability. 

For the loss $Y^1$, we vary
$$c \in \{0.25, 0.5,1,1.5,2\},$$
retraining the stopping rule for each value. The results are shown in Table~\ref{tab:Y1-exp-c}. As $c$ increases, delayed detection becomes more costly, and the learned stopping rule tends to stop earlier, reducing detection delay at the expense of a potentially larger false alarm rate.
\begin{table}[H]
	\centering
    \tiny
\begin{tabular}{lcccccc}
	\toprule
 Parameter & $\mathbb{E}[Y_{\tau \wedge T}]$ & $\mathbb{E}[(\tau-\theta)^+]$ & $\mathbb{P}(\tau<\theta)$ & $\mathbb{E}[\tau]$ & $\mathbb{P}_{\infty}(\tau<T)$ & $\mathbb{E}_{\infty}[\tau \wedge T]$ \\
	\midrule
	$c=0.25$ & 0.2249 & 0.7051 & 0.0487 & 2.0506 & 0.3376 & 7.8722 \\
	$c=0.5$ & 0.3692 & 0.5298 & 0.1043 & 1.8233  & 0.6056 & 6.0718 \\
	$c=1$ & 0.5877 & 0.2984 & 0.2893 & 1.3031  & 0.9960 & 2.4053 \\
	$c=1.5$ & 0.7015 & 0.1890 & 0.4180 & 0.9963  & 1.0000 & 1.3903 \\
    $c=2$ & 0.7876 & 0.1155 & 0.5567 & 0.7474  & 1.0000 & 0.9518 \\
	\bottomrule
\end{tabular}
\caption{Change-point $\theta^{\text{Exp}}$ and payoff $Y^1$ with different $c$.}
\label{tab:Y1-exp-c}
\end{table}

For the loss $Y^2$, we fix $a=1$ and vary
$$b \in \{0.25, 0.5,1,1.5,2\}.$$
The corresponding results are reported in Table~\ref{tab:Y2-exp-b}. Again, larger values of $b$ lead to more aggressive stopping behavior, reflecting the increased penalty on post-change delay. Similar to the last experiment, as the the calibrated $b$ increases, the false alarm probability $\PP(\tau<\theta)$ increases and the detection delay decreases substantially, illustrating the expected trade-off between false alarms and detection speed.
\begin{table}[H]
	\centering
    \tiny
\begin{tabular}{lcccccc}
	\toprule
	Parameters  & $\mathbb{E}[Y_{\tau \wedge T}]$ & $\mathbb{E}[(\tau-\theta)^+]$ & $\mathbb{P}(\tau<\theta)$ & $\mathbb{E}[\tau]$ &  $\mathbb{P}_{\infty}(\tau<T)$ & $\mathbb{E}_{\infty}[\tau \wedge T]$ \\
	\midrule
	$a=1,\ b=0.25$ & 0.2534 & 0.7874 & 0.0367 & 2.1584 & 0.2120 & 8.6246 \\
	$a=1,\ b=0.5$ & 0.4089 & 0.6241 & 0.0693 & 1.9049 & 0.3612 & 7.5396 \\
	$a=1,\ b=1$ & 0.6771 & 0.3908 & 0.2030 & 1.5305 & 0.8892 & 3.9604 \\
	$a=1,\ b=1.5$ & 0.8154 & 0.3070 & 0.2620 & 1.3472  & 0.9376 & 2.9976 \\
	$a=1,\ b=2$ & 0.9657 & 0.2127 & 0.3857 & 1.0793  & 1.0000 & 1.5920 \\
	\bottomrule
\end{tabular}
\caption{Change-point $\theta^{\text{Exp}}$ and payoff $Y^2$ with $a=1$ and different $b$.}
\label{tab:Y2-exp-b}
\end{table}
We provide the same experiments for Change-point $\theta^{\text{Wei}}$ in the following Table~\ref{tab:Y1-wei-c}-\ref{tab:Y2-wei-b} and discover similar pattern.
\begin{table}[H]
	\centering
    \tiny
\begin{tabular}{lcccccc}
	\toprule
	Parameters & $\mathbb{E}[Y_{\tau \wedge T}]$ & $\mathbb{E}[(\tau-\theta)^+]$ & $\mathbb{P}(\tau<\theta)$ & $\mathbb{E}[\tau]$ & $\mathbb{P}_{\infty}(\tau<T)$ & $\mathbb{E}_{\infty}[\tau \wedge T]$ \\
	\midrule
	$c=0.25$ & 0.1942 & 0.6353 & 0.0353 & 2.0308 & 0.8912 & 5.2084 \\
	$c=0.5$ & 0.3294 & 0.4962 & 0.0813 & 1.8861 & 0.9964 & 3.5945 \\
	$c=1$ & 0.5362 & 0.3148 & 0.2213 & 1.5844 & 1.0000 & 2.3466 \\
	$c=1.5$ & 0.6800 & 0.2282 & 0.3377 & 1.3987 & 1.0000 & 1.8178 \\
	$c=2$ & 0.7736 & 0.1705 & 0.4327 & 1.2818 & 1.0000 & 1.5838 \\
	\bottomrule
\end{tabular}
\caption{Change-point $\theta^{\text{Wei}}$ and payoff $Y^1$ with different $c$.}
\label{tab:Y1-wei-c}
\end{table}

\begin{table}[H]
	\centering
    \tiny
\begin{tabular}{lcccccc}
	\toprule
	Parameters & $\mathbb{E}[Y_{\tau \wedge T}]$ & $\mathbb{E}[(\tau-\theta)^+]$ & $\mathbb{P}(\tau<\theta)$ & $\mathbb{E}[\tau]$ & $\mathbb{P}_{\infty}(\tau<T)$ & $\mathbb{E}_{\infty}[\tau \wedge T]$ \\
	\midrule
	$a=1,\ b=0.25$ & 0.1738 & 0.5308 & 0.0730 & 1.9141 & 0.9932 & 3.6527 \\
	$a=1,\ b=0.5$ & 0.2823 & 0.4258 & 0.1247 & 1.7530 & 1.0000 & 2.8835 \\
	$a=1,\ b=1$ & 0.4477 & 0.2574 & 0.2920 & 1.4974 & 1.0000 & 1.9826 \\
	$a=1,\ b=1.5$ & 0.5506 & 0.2023 & 0.3727 & 1.3664 & 1.0000 & 1.6457 \\
	$a=1,\ b=2$ & 0.6540 & 0.1573 & 0.4597 & 1.2380 & 1.0000 & 1.4415 \\
	\bottomrule
\end{tabular}
\caption{Change-point $\theta^{\text{Wei}}$ and payoff $Y^2$ with $a=1$ and different $b$.}
\label{tab:Y2-wei-b}
\end{table}

Building on the preceding sensitivity analysis above, we next study the constrained quickest-detection problem
	$$
	\min _\tau \mathbb{E}\left[(\tau-\theta)^{+}\right] \quad \text { subject to } \quad \mathbb{P}(\tau<\theta) \leq \alpha.
	$$
We use the loss $Y_\tau^1=\1_{\{\tau<\theta\}}+c(\tau-\theta)^{+}$ as its Lagrangian formulation and calibrate $c$ for each target false alarm probability $\alpha$. Tables~\ref{tab:constraint_exp} and \ref{tab:constraint_wei} report the results of a binary search over $c$ to match target PFA $\alpha \in \{0.01, 0.05, 0.10, 0.15\}$ under exponential and Weibull change time distributions, respectively. The resulting false alarm probabilities closely match their prescribed levels. This demonstrates that the proposed training procedure can reliably enforce a desired Bayesian false alarm constraint. The corresponding no-change false alarm probabilities can nevertheless be considerably higher because, under $\PP_\infty$, the process remains in the pre-change regime throughout the entire horizon.
\begin{table}[H]
\centering
\tiny
\begin{tabular}{cccccccc}
	\toprule
	Target PFA & $\mathbb{P}(\tau<\theta)$ & c & $\mathbb{E}[Y_{\tau \wedge T}]$ & $\mathbb{E}[(\tau-\theta)^+]$ & $\mathbb{E}[\tau]$ & $\mathbb{P}_{\infty}(\tau<T)$ & $\mathbb{E}_{\infty}[\tau \wedge T]$ \\
	\midrule
	0.0100 & 0.0108 & 0.0639  & 0.0705 & 0.9343 & 2.3701  & 0.1065 & 9.3415 \\
	0.0500 & 0.0500 & 0.2688 & 0.2330 & 0.6808 & 2.0509  & 0.3892 & 7.6613 \\
	0.1000 & 0.1006  & 0.4359  & 0.3352 & 0.5382 & 1.8179 & 0.5720 & 6.4453 \\
	0.1500 & 0.1496  & 0.5221  & 0.3999 & 0.4794 & 1.6678  & 0.7700 & 4.9131 \\
	\bottomrule
\end{tabular}
\caption{Constraint problem with change-point $\theta^{\text{Exp}}$.}
\label{tab:constraint_exp}
\end{table}
\begin{table}[H]
\centering
\tiny
\begin{tabular}{cccccccc}
	\toprule
	Target PFA & $\mathbb{P}(\tau<\theta)$ & c & $\mathbb{E}[Y_{\tau \wedge T}]$ & $\mathbb{E}[(\tau-\theta)^+]$ & $\mathbb{E}[\tau]$ & $\mathbb{P}_{\infty}(\tau<T)$ & $\mathbb{E}_{\infty}[\tau \wedge T]$ \\
	\midrule
	0.0100 & 0.0096 & 0.0801 & 0.0795 & 0.8723 & 2.2919 & 0.4198 & 8.0912 \\
	0.0500 & 0.0506 & 0.2917 & 0.2152 & 0.5645 & 1.9659 & 0.9050 & 4.6911 \\
	0.1000 & 0.1000 & 0.5639 & 0.3586 & 0.4586 & 1.8228 & 0.9938 & 3.4948 \\
	0.1500 & 0.1508 & 0.7000 & 0.4305 & 0.3995 & 1.7352 & 0.9962 & 3.0030 \\
	\bottomrule
\end{tabular}
\caption{Constraint problem with change-point $\theta^{\text{Wei}}$.}
\label{tab:constraint_wei}
\end{table}
Overall, the results confirm that the learned linear signature stopping rules adapt consistently to the choice of loss function and parameters. The signature 
approach outperforms the Page-Hinkley baseline across all metrics in Table~\ref{tab:cumsum_cmp}, while the parameter 
sensitivity analysis in Tables~\ref{tab:Y1-exp-c}--\ref{tab:Y2-wei-b} demonstrates the flexibility of the framework in navigating the delay-false alarm trade-off.

\subsection{Repeated experiments with a common change-point}
We compare the hard stopping rule $\tau_{l^*}$ of Section~\ref{sec:detector} with the aggregated rule $\tilde{\tau}$ of Section~\ref{sec:det-theta}, using $R=25$ replications for $\tilde{\tau}$. Both rules use the same signature truncation level and loss function $Y^1$. Table~\ref{tab:fixed_theta_compare} reports the results for several fixed values of $\theta$.

\begin{table}[H]
	\centering
    \tiny
	\begin{tabular}{llcccc}
		\toprule
		$\theta$ & Policy & $\mathbb{E}[Y_{\tau \wedge T}]$ & $\mathbb{E}[(\tau-\theta)^+]$ & $\mathbb{P}(\tau<\theta)$ & $\mathbb{E}[\tau]$ \\
		\midrule
		0.50 & $\tau_{l^*}$ & 0.4286 & 0.4006 & 0.0280 & 0.8974 \\
		0.50 & $\tilde{\tau}$ ($R=25$) & 0.3949 & 0.3949 & 0.0000 & 0.8949 \\
		1.00 & $\tau_{l^*}$ & 0.4552 & 0.2878 & 0.1673 & 1.2367 \\
		1.00 & $\tilde{\tau}$ ($R=25$) & 0.2720 & 0.2720 & 0.0000 & 1.2720 \\
		1.50 & $\tau_{l^*}$ & 0.5705 & 0.1839 & 0.3867 & 1.4903 \\
		1.50 & $\tilde{\tau}$ ($R=25$) & 0.2530 & 0.1284 & 0.1247 & 1.6148 \\
		2.00 & $\tau_{l^*}$ & 0.6996 & 0.1136 & 0.5860 & 1.6835 \\
		2.00 & $\tilde{\tau}$ ($R=25$) & 0.8893 & 0.0079 & 0.8813 & 1.7228 \\
		3.00 & $\tau_{l^*}$ & 0.9504 & 0.0097 & 0.9407 & 1.7396 \\
		3.00 & $\tilde{\tau}$ ($R=25$) & 1.0000 & 0.0000 & 1.0000 & 1.7330 \\
		\bottomrule
	\end{tabular}
	\caption{Comparison of stopping policies for fixed change times $\theta$.}
	\label{tab:fixed_theta_compare}
\end{table}

\subsection{Adversarial quickest detection}
We now present numerical experiments for the adversarial quickest detection problem. In contrast to the nominal setting, where the observation process is generated solely by the signal and noise, we now allow an additional perturbation chosen by an adversary. More precisely, the observed path takes the form
$$
Z_t = X_t + w_t,\qquad t\in[0,T],
$$
where $X$ is the baseline process and $w$ is an adversarial perturbation. In our experiments, we use Example~\ref{ex:fBM} as the baseline model:
$$
dX_t^{(1)}=\sigma\, dB_t^H, \qquad dX_t^{(2)}=r\, dt+\sigma\, dB_t^H,
$$
where we take $\sigma=1$, $r=3$, $H=0.35$, and truncation level $N=4$. The observed path $X$ is the concatenated trajectory defined in \eqref{eq:X}. The adversary is constrained by a bounded total variation budget. More precisely, the perturbation belongs to the class
$$
\mathcal{W}:=\left\{w:[0,T]\to\mathbb{R}\,\middle|\, w_0=0,\ \|w\|_{\mathrm{TV};[0,T]}\leq C_{\mathrm{TV}}\right\},
$$
where $C_{\mathrm{TV}}>0$ is fixed. This class models structured but finite strength perturbations that may distort the observed trajectory and thereby make the detection task more challenging. The robust quickest detection problem can be formulated as in Section~\ref{sec:adv}:
$$
\inf_{\tau}\sup_{P_1\in\mathcal{P}_1,\;P_2\in\mathcal{P}_2}\E^{P_1,P_2}\bigl[Y_{\tau\wedge T}\bigr],
$$
with $\mathcal{P}_1$ and $\mathcal{P}_2$ defined appropriately. In the present setting, this is equivalently written as
$$
\inf_{\tau}\sup_{w\in\mathcal{W}} \E^{w}\bigl[Y_{\tau\wedge T}\bigr].
$$
where $\E^w[\cdot]$ denotes the expectation when perturbation is $w$. As in the ordinary quickest detection problem, we restrict attention to linear signature type stopping rules. Given a coefficient vector $l$, we consider the hard stopping policy
$$
\tau_l=\inf\left\{t_k:\left|\left\langle l,\widehat{\mathbb{Z}}^{<\infty}_{0,t_k}\right\rangle\right|\geq 1\right\}\wedge T,
$$
where $\widehat{\mathbb{Z}}$ denotes the signature lift of the discretized observation path $Z=X+w$.

One approach to approximating the supremum over $\mathcal{W}$ is to simulate a finite representative subset of admissible perturbations and maximize over this subset. As an alternative, we parameterize the adversary and employ an alternating zeroth-order optimization procedure based on esGS to solve the resulting min-max problem (see, for example, \cite{hsieh21a} for alternating optimization in min-max problem). 
In principle, the adversary may use any $w\in\mathcal{W}$ 
with $\|w\|_{TV;[0,T]} \leq C_{TV}$. In practice, 
since a rational adversary will exhaust the full budget 
to maximally disrupt the detector, we restrict attention 
to perturbations with $\|w\|_{TV;[0,T]} = C_{TV}$. 
This simplification reduces the search space without 
loss of generality for the minimax problem, and leads 
to the following piecewise constant parametrization.
\begin{enumerate}[label = (\roman*),leftmargin=*, align=left,labelsep=0em]
	\item We represent the perturbation $w$ by a piecewise constant path with $J$ jumps:
	$$
	w_t=C_{\mathrm{TV}}\sum_{j=1}^J s_j\, p_j\,\1_{\{T_j\leq t\}}.
	$$
	\item The jump magnitudes $(p_1,\dots,p_J)$ are sampled from a Dirichlet distribution with parameter $\boldsymbol{\alpha}^m\in\R_+^J$. Since $\sum_{i=1}^{J}p_i=1$, $\boldsymbol{\alpha}^m$ determines how the total variation budget is distributed across the $J$ jumps.
	\item The jump times $T_1,\dots,T_J$ are generated from another Dirichlet distribution on the $(J+1)$ interval lengths, with parameter $\boldsymbol{\alpha}^t\in\R_+^{J+1}$. The cumulative sums of these interval lengths then yield $J$ ordered jump times on $[0,T]$.
	\item The signs $s_j\in\{-1,+1\}$ are sampled from a Rademacher distribution and determine the direction of each jump. The Rademacher distribution is used rather than a fixed sign to allow the adversary to learn whether upward or downward perturbations are more harmful to the detector.
\end{enumerate}
In this way, the adversary is described by two families of positive parameters: one controlling the distribution of jump magnitudes and the other controlling the distribution of jump locations. This parameterization yields a flexible yet low dimensional class of perturbations and allows us to search numerically for challenging adversarial patterns. Denote $\boldsymbol{\alpha}=(\boldsymbol{\alpha}^m,\,\boldsymbol{\alpha}^t)$. Under this parameterization, the robust training problem becomes
$$
\inf_{l}\sup_{\boldsymbol{\alpha}} \E^{\boldsymbol{\alpha}}\bigl[Y_{\tau_l\wedge T}\bigr].
$$
Since the stopping rule is defined through a threshold crossing, the resulting objective is non-smooth and exact gradients are unavailable. We therefore adopt an alternating esGS procedure. Starting from an initial detector coefficient $l^{(0)}$ and initial adversarial parameters $\boldsymbol{\alpha}^{(0)}=(\boldsymbol{\alpha}^{m,(0)},\boldsymbol{\alpha}^{t,(0)})$, each iteration consists of two stages:
\begin{enumerate}[leftmargin=*, align=left]
	\item \textbf{Detector update.} Keeping the adversarial parameters fixed, we apply a esGS SGD step to update $l$ so as to decrease the empirical robust loss:
	$$
	l^{(k+1)}=l^{(k)}-\gamma_k\,\hat{g}_{\eta}(l^{(k)},\boldsymbol{\alpha}^{(k)},v^{(k)}_l,z^{(k)}_l).
	$$
	\item \textbf{Adversary update.} Keeping the detector coefficient fixed, we apply a esGS SGD step to update $\boldsymbol{\alpha}^m$ and $\boldsymbol{\alpha}^t$ so as to increase the same empirical loss:
	$$
	\boldsymbol{\alpha}^{(k+1)}=\boldsymbol{\alpha}^{(k)}+\gamma_k\,\hat{g}_{\eta}(l^{(k+1)},\boldsymbol{\alpha}^{(k)},v^{(k)}_\al,z^{(k)}_\al).
	$$
\end{enumerate}
Here $\hat{g}_\eta$ is the esGS gradient estimator in \eqref{eq:esgs}. These two steps are then alternated until the prescribed number of outer iterations is reached. In each outer iteration, one can also performs $k_1$ updates for $l$ and then $k_2$ updates for $\boldsymbol{\alpha}^m,\,\boldsymbol{\alpha}^t$ (here $k_1,k_2 \in \N$ are arbitrary but fixed). In practice, each objective evaluation is the empirical expectation. Thus, the detector is trained against increasingly harmful perturbations, while the adversary is simultaneously adapted to exploit the current weakness of the stopping rule.

Our main goal is to compare the nominally trained stopping rule with the robustly trained one. In particular, we examine whether adversarial training leads to improved worst case performance over the perturbation class, and how much this robustness costs in terms of detection delay or nominal performance. We do the experiments for $\theta \sim \mathrm{Exp}(0.5),\, \theta\sim \mathrm{Weibull}(2/\Gamma(\frac{3}{2}),2)$ and $C_{\text{TV}}=\{0.5,1,2,4\}$. 
\begin{table}[H]
	\centering
	\tiny
	\setlength{\tabcolsep}{5pt}
	\begin{tabular}{lcccccc}
		\toprule
		Change time & Loss & $C_{\mathrm{TV}}$
		& $\mathbb{E}[Y_{\tau\wedge T}]$ & $\mathbb{E}[(\tau-\theta)^+]$ & $\mathbb{P}(\tau<\theta)$ & $\mathbb{E}[\tau]$ \\
		\midrule
		\multirow{8}{*}{Exponential}
		& \multirow{4}{*}{$Y^1$}
		& 0.5 & 0.6330/0.6294 & 0.2875/0.2929 & 0.3455/0.3365 & 1.5460/1.5647 \\
		& & 1   & 0.6495/0.6445 & 0.2795/0.2890 & 0.3700/0.3555 & 1.5679/1.6010 \\
		& & 2   & 0.6438/0.6387 & 0.2773/0.2907 & 0.3665/0.3480 & 1.5544/1.5990 \\
		& & 4   & 0.6543/0.6500 & 0.2703/0.2755 & 0.3840/0.3745 & 1.5281/1.5558 \\
		\cmidrule(lr){2-7}
		& \multirow{4}{*}{$Y^2$}
		& 0.5 & 0.8431/0.8374 & 0.5802/0.5689 & 0.1325/0.1350 & 2.2333/2.2164 \\
		& & 1   & 0.8672/0.8529 & 0.5716/0.5486 & 0.1370/0.1440 & 2.3045/2.2728 \\
		& & 2   & 0.8367/0.8256 & 0.5530/0.5369 & 0.1445/0.1485 & 2.2315/2.2103 \\
		& & 4   & 0.9009/0.9001 & 0.5514/0.5492 & 0.1740/0.1705 & 2.2230/2.2194 \\
		\midrule
		\multirow{8}{*}{Weibull}
		& \multirow{4}{*}{$Y^1$}
		& 0.5 & 0.5824/0.5768 & 0.3249/0.3263 & 0.2575/0.2505 & 2.0691/2.0838 \\
		& & 1   & 0.5788/0.5791 & 0.3238/0.3366 & 0.2550/0.2425 & 2.0511/2.0807 \\
		& & 2   & 0.5956/0.5933 & 0.3266/0.3388 & 0.2690/0.2545 & 2.0464/2.0792 \\
		& & 4   & 0.6210/0.6197 & 0.3170/0.3217 & 0.3040/0.2980 & 2.0014/2.0276 \\
		\cmidrule(lr){2-7}
		& \multirow{4}{*}{$Y^2$}
		& 0.5 & 0.5685/0.5681 & 0.3737/0.3643 & 0.2235/0.2300 & 2.1513/2.1330 \\
		& & 1   & 0.5783/0.5756 & 0.3606/0.3547 & 0.2405/0.2440 & 2.1789/2.1697 \\
		& & 2   & 0.5721/0.5661 & 0.3721/0.3677 & 0.2240/0.2265 & 2.1646/2.1618 \\
		& & 4   & 0.5792/0.5756 & 0.3472/0.3432 & 0.2680/0.2710 & 2.1483/2.1439 \\
		\bottomrule
	\end{tabular}
	\caption{Worst case performance of the nominal and adversarial trained stopping rules. Each entry is reported as \textbf{nominal/adversarial} trained.}
	\label{tab:ad_summary}
\end{table}
Tables~\ref{tab:ad_summary} show that, in general, adversarial training improves worst case empirical performance, evaluated under the worst-case perturbations $\boldsymbol{\alpha}^m$ and $\,\boldsymbol{\alpha}^t$. For $Y^1$, the improvement is primarily achieved by reducing false alarms at the cost of a small increase in detection delay, whereas for $Y^2$, it generally results from reducing the delay while accepting a slight increase in false alarms. This difference reflects the structures of the two losses: $Y^1$ assigns the same fixed penalty to every false alarm, encouraging more conservative stopping of the adversarial training, while $Y^2$ penalizes a false alarm according to how early it occurs, allowing the rule to stop more aggressively when an alarm is expected to occur close to the change-point. The rules also tend to achieve lower risks under the Weibull change time distribution. This is consistent with the ability of the time-augmented signature to exploit the non-memoryless structure of the Weibull distribution and learn predictive information from the history.

\section{Conclusion}\label{sec:conclusion}
We develop a framework for quickest detection on rough path space, unifying classical sequential change-point detection theory with modern rough path signatures. The optimal stopping rule for natural loss objectives is a signature half-space hitting time, the first time a linear functional of the observed rough path signature crosses a threshold. This form arises from the payoff structure of the optimal stopping problem and the intrinsic geometry of the observed path, providing mutual theoretical support for the proposed stopping rules. Our framework accommodates non-Markovian, nonstationary signals driven by irregular noise, including fractional Brownian motion, for which classical sufficient statistics are unavailable. Unlike classical procedures like CUSUM and the Shiryaev rule, the truncated signature coefficient is a universal, model-free feature of the observed path, learned from data. The framework extends to a distributionally robust or adversarial formulation, where the pre-change and post-change models range over prescribed uncertainty classes. When a least favorable model pair exists, the robust problem reduces to an ordinary stopping problem, and the same signature-based rule applies. Numerical experiments confirm that adversarial training produces robust stopping rules to worst-case path perturbations.

\bigskip

\bibliographystyle{plain}
\bibliography{quick_detect.bib}

\begin{appendix}

\section{Additional preliminaries}
\subsection{Tensor algebra and Lie bracket}\label{sec:tensor}
To define iterated integrals of paths in a coordinate-free and noncommutative setting, we work in the tensor algebra over $\mathbb{R}^d$ and its associated Lie algebra.
Let $(e_1,\dots,e_d)$ be the canonical basis of $\R^d$, and let $(e_1^*,\dots,e_d^*)$ be the dual basis of $(\R^d)^*$. For each $n\ge 0$, define
$$
(\R^d)^{\otimes 0}:=\R,\qquad(\R^d)^{\otimes n}:=\underbrace{\R^d\otimes\cdots\otimes\R^d}_{n\text{ times}},\quad n\ge1.
$$
An element $a_n\in(\R^d)^{\otimes n}$ is of the form
$$a_n=\sum_{i_1,...,i_n=1}^da_{(i_1,...,i_n)}e_{i_1}\otimes\cdots\otimes e_{i_n}.$$
We set
$$
T(\R^d):=\bigoplus_{n=0}^\infty (\R^d)^{\otimes n},\qquad T((\R^d)):=\prod_{n=0}^\infty (\R^d)^{\otimes n}.
$$
Thus an element $\mathbf{a}\in T((\R^d))$ is of the form
$$
\mathbf{a}=(a_0,a_1,a_2,\dots),\qquad a_n\in (\R^d)^{\otimes n},
$$
and belongs to $T(\R^d)$ if only finitely many levels are nonzero. We denote the projections
$$
\pi_n(\mathbf{a}):=a_n ,\quad \text{and } \pi_{\leq N}(\mathbf{a}):=(a_0,a_1,...,a_N).
$$
We equip $T((\R^d))$ with sum and scalar product
\begin{align*}
	\mathbf{a}+\mathbf{b}&:=\left(a_0+b_0,\ldots,a_n+b_n,\ldots\right),\\
	\lambda\cdot\mathbf{a}&:=\left(\lambda a_0,\ldots,\lambda a_n,\ldots\right),
\end{align*}
and the tensor product $\mathbf{a}\otimes\mathbf{b}$ defined levelwise by
$$
(\mathbf{a}\otimes \mathbf{b})_n
:=
\sum_{k=0}^n a_k\otimes b_{n-k},
\qquad n\ge0,
$$
for $\mathbf{a}=(a_n)_{n=0}^\infty$, $\mathbf{b}=(b_n)_{n=0}^\infty$ in $T((\R^d))$. With this product, $(T((\mathbb{R}^{d})),+,\cdot,\otimes)$ is a real non-commutative algebra with neutral element $\mathbf{1}=(1,0,\ldots,0,\ldots)$.
For $N\in\N$, we also define the truncated tensor algebra 
$$T^N(\mathbb{R}^d):=\left\{\mathbf{a}\in T((\mathbb{R}^d)):a_n=0,\forall n>N\right\}.$$
%\begin{definition}[Free nilpotent Lie algebra and Lie bracket]
Let $T^N(\R^d)$ be equipped with the associative product $\otimes$ from above. Define the \emph{commutator (Lie bracket)} of $\mathbf{a},\mathbf{b}\in T^N(\R^d)$ by
$$
	[\mathbf{a},\mathbf{b}]\;:=\; \mathbf{a}\otimes \mathbf{b} - \mathbf{b}\otimes \mathbf{a}.
$$
Let $\mathfrak{g}^N(\R^d)\subset T^N(\R^d)$ be the smallest linear subspace containing $\R^d$ and closed under the bracket $[\cdot,\cdot]$; equivalently, $\mathfrak{g}^N(\R^d)$ is the linear span of all iterated commutators of elements of $\R^d$ of bracket-length at most $N$.
$$\mathfrak{g}^N\left(\mathbb{R}^d\right)=\mathbb{R}^d \oplus\left[\mathbb{R}^d, \mathbb{R}^d\right] \oplus \cdots \oplus \underbrace{\left[\mathbb{R}^d,\left[\ldots,\left[\mathbb{R}^d, \mathbb{R}^d\right]\right]\right]}_{(N-1) \text { brackets }} .$$
Then $(\mathfrak{g}^N(\R^d),[\cdot,\cdot])$ is a Lie algebra, and we call it the \emph{free step-$N$ nilpotent Lie algebra}, see \cite[Def 7.25]{friz-victoir}.
%\end{definition}

It is readily checked that the Lie bracket has the following properties:
\begin{enumerate}
	\item \emph{Bilinearity}: For all scalars $\alpha, \beta$ and all
	$\mathbf{a}_1, \mathbf{a}_2, \mathbf{b} \in T(\R^d)$,
	$$
	\left[\alpha \mathbf{a}_1+\beta \mathbf{a}_2, \mathbf{b}\right]
	=\alpha\left[\mathbf{a}_1, \mathbf{b}\right]+\beta\left[\mathbf{a}_2, \mathbf{b}\right],
	\quad\text{and}\quad
	\left[\mathbf{a}, \alpha \mathbf{b}_1+\beta \mathbf{b}_2\right]
	=\alpha\left[\mathbf{a}, \mathbf{b}_1\right]+\beta\left[\mathbf{a}, \mathbf{b}_2\right].
	$$
	\item \emph{Anti-symmetry}:
	$$
	[\mathbf{b}, \mathbf{a}]=-[\mathbf{a}, \mathbf{b}],
	$$
	and in particular $[\mathbf{a}, \mathbf{a}]=0$.
	\item \emph{Jacobi identity}: For all $\mathbf{a},\mathbf{b},\mathbf{c}\in T(\R^d)$,
	$$
	\bigl[\mathbf{a},[\mathbf{b},\mathbf{c}]\bigr]
	+\bigl[\mathbf{b},[\mathbf{c},\mathbf{a}]\bigr]
	+\bigl[\mathbf{c},[\mathbf{a},\mathbf{b}]\bigr]=0.
	$$
	This follows directly from the associativity of $\otimes$.
\end{enumerate}
The Lie algebra $\mathfrak{g}^N(\mathbb{R}^d)$ underlies the  exponential parametrization of the free nilpotent group introduced in Section~\ref{sec:shuffle} and reappears in Example~\ref{ex:BM} in the analysis of the Carnot--Carath\'eodory geometry of the Brownian rough path.
\subsection{Shuffle and group-like elements}\label{sec:shuffle}
To evaluate the signature against linear functionals and to exploit its multiplicative structure, we introduce the dual tensor algebra and the shuffle product.
$$
T((\R^d)^*):=\bigoplus_{n=0}^\infty ((\R^d)^*)^{\otimes n},
$$
where elements $l\in T((\R^d)^*)$ are sequences $(l^{(n)})_{n\ge 0}$ with $l^{(n)}=0$ for all but finitely many $n$, so that the pairing below is always a finite sum.
Each $l=(l^{(n)})_{n\ge0}\in T((\R^d)^*)$ acts naturally on $\mathbf{a}=(a_n)_{n\ge0}\in T((\R^d))$ by
$$
\langle l,\mathbf{a}\rangle:=\sum_{n=0}^\infty \langle l^{(n)},a_n\rangle,
$$
whenever the sum is well defined. In particular, if $l$ has only finitely many nonzero levels, then the above sum is finite. For a multi-index $I=(i_1,\dots,i_n)$ with $i_k\in\{1,\dots,d\}$, define
$$
|I|:=n,\qquad e_I:=e_{i_1}\otimes\cdots\otimes e_{i_n}\in (\R^d)^{\otimes n},\qquad e_I^*:=e_{i_1}^*\otimes\cdots\otimes e_{i_n}^*\in ((\R^d)^*)^{\otimes n}.
$$
For the empty multi-index $\varnothing$, we set $|\varnothing|=0$ and $e_\varnothing:=1$. Then
$$
\langle e_I^*,e_J\rangle={\1}_{\{I=J\}},
$$
and for $\mathbf{a}\in T((\R^d))$, the scalar $\langle e_I^*,\mathbf{a}\rangle$ is called the $I$-th tensor coordinate of $\mathbf{a}$. For two multi-indices $I=(i_1,\dots,i_{|I|})$ and $J=(j_1,\dots,j_{|J|})$, the \emph{shuffle product} is recursively defined by
$$
e_I^*\shuffle e_\varnothing^*=e_\varnothing^*\shuffle e_I^*=e_I^*,
$$
and
$$
e_I^*\shuffle e_J^*:=(e_{I'}^*\shuffle e_J^*)\otimes e_{i_{|I|}}^*+(e_I^*\shuffle e_{J'}^*)\otimes e_{j_{|J|}}^*,
$$
where $I'=(i_1,\dots,i_{|I|-1})$ and $J'=(j_1,\dots,j_{|J|-1})$. Equivalently, $e_I^*\shuffle e_J^*$ is the sum of all words obtained by interleaving the letters of $I$ and $J$ while preserving the order within each word.

We define
$$
G(\R^d):=\left\{\mathbf{a}\in T((\R^d))\setminus\{\mathbf{0}\}:\;\langle l_1\shuffle l_2,\mathbf{a}\rangle=\langle l_1,\mathbf{a}\rangle \langle l_2,\mathbf{a}\rangle\text{ for all } l_1,l_2\in T((\R^d)^*)\right\},
$$
and refer to $G(\R^d)$ as the set of \emph{group-like elements}.
The importance of group-like elements stems from the fact that the signature $S(x)_{s,t}$ of any bounded-variation path
$x$ is group-like; this will be confirmed below as a consequence of Chen's identity.
In particular, if $\mathbf{g}\in G(\R^d)$, then $\pi_0(\mathbf{g})=1.$
It is well known that $G(\R^d)$ forms a group under the tensor product $\otimes$, with identity $\mathbf{1}$ and inverse
$$
\mathbf{g}^{-1}=\sum_{n\ge0}(\mathbf{1}-\mathbf{g})^{\otimes n}.
$$
The group structure of $G(\mathbb{R}^d)$ is not merely algebraic: it encodes the concatenation of paths at the level of their iterated
integrals, as made precise by Chen's identity below.
We also set
$$
G^N(\R^d):=\pi_{\le N}\bigl(G(\R^d)\bigr).
$$
Then $G^N(\R^d)$ is the \emph{free nilpotent group of step $N$}, with group operation given by the tensor product followed by truncation. Moreover, $G^N\left(\mathbb{R}^d\right)=\exp\left(\mathfrak{g}^N\left(\mathbb{R}^d\right)\right)$, see \cite[Sec 7.5]{friz-victoir}.

\section{Randomized optimal stopping policy}\label{sec:rd-stop}
Observe that the process $Y$ defined in Section~\ref{sec:obj} is right continuous and $\E[\|Y\|_{\infty}]<\infty$ for any finite interval. In order to provide an algorithm to detect the change-point $\theta$ we need to solve the optimal stopping problem $\inf_{\tau\in\cs} \E[Y_{\tau \wedge T}]$. Let us first introduce a lemma from \cite{Bayer2023}.

\begin{lemma}\label{lem:ex_theta_stop}
	Let $\widehat{\X}$ be a stochastic process in $\hat{\Omega}_T^p$ and set $\cf^X_t \coloneqq \sigma(\widehat{\X}_{0,s}\, :\, 0 \leq s \leq t) = \sigma(\widehat{\X}|_{[0,t]})$. Let $\tau$ be a stopping time with respect to $(\cf^X_t)$. Then there is a Borel measurable map $\phi \colon \Lambda_T \to \{0,1\}$ such that
	\begin{align*}
		\phi({\widehat{\X}}(\omega)|_{[0,t]}) = \mathbbm{1}_{\{ \tau(\omega) \leq t \}}
	\end{align*}
	for every $\omega \in\bar{\Omega}$.
\end{lemma}
The following results from \cite{Bayer2023} is essential for justifying our algorithm. However, in \cite{Bayer2023} the authors work with the objective stochastic process adapted to the filtration $(\cf^X_t)$ and continuous. These assumptions are not satisfied by the loss function $Y$ in our quickest detection problem. Therefore, we include the following results and brief proofs for completeness. Recall that $\mathcal{T} \coloneqq C(\Lambda_T,\R)$ and call it the space of \emph{continuous stopping policies}. 

\begin{definition}\label{def:rand-stopping-times}
	Let $Z$ be a non-negative random variable independent of $\widehat{\X}$ and such that $\PP_Z(Z = 0) = 0$. For a continuous stopping policy $\phi \in \mathcal{T}$, we define the \emph{randomized stopping time} by
	\begin{equation*}%\label{eqn:randomized_stopping_time}
		\tau^r_{\phi} \coloneqq \inf \left\{t \geq 0\,:\, \int_0^{t \wedge T} \phi (\widehat{\X}|_{[0,s]})^2\, ds \geq Z \right\}
	\end{equation*}
	where $\inf \emptyset = + \infty$.
\end{definition}

Next we prove that stopping times can be approximated by randomized stopping times based on continuous stopping policies.

\begin{proposition}\label{prop:opt_cont_pol}
	Let $Y$ be right continuous in $t$ and $\E[\|Y\|_{\infty}]<\infty$ for any finite interval. For every stopping time $\tau \in \mathcal{S}$, there exists a sequence $\phi_n \in \mathcal{T}$ such that the randomized stopping times $\tau^r_{\phi_n}$ satisfy $\tau_{\phi_n}^r \to \tau$ almost surely as $n \to \infty$. Furthermore,
	\begin{align*}
		\inf_{\phi \in \mathcal{T}} \E[{Y}_{\tau^r_{\phi} \wedge T}] = \inf_{\tau \in \mathcal{S}} \E[Y_{\tau \wedge T}].
	\end{align*}
\end{proposition}
\begin{proof}
	From Lemma~\ref{lem:ex_theta_stop}, any stopping time $\tau\in \cs$ can be represented by a Borel measurable map $\phi(\widehat{\mathbb{X}}|_{[0, t]})=\1_{\{\tau \leq t\}}$. By \cite[Thm 1]{Winiewski1994} there exists a sequence of continuous functions $\tilde{\phi}_n\in\ct$ that satisfying $0\leq \tilde{\phi}_n\leq1$ and converge to $\1_{\{\tau \leq t\}}$ almost surely. By setting $\phi_n:=(2\tilde{\phi}_n)^n$, we construct a scaled sequence $\phi_n$ such that
	\begin{align*}
		\lim_{n \to \infty} \phi_n(\hat{\mathbb{X}}|_{[0,t]}) \to \begin{cases}
			+ \infty &\text{if } t \geq \tau \\
			0 &\text{if } t < \tau.
		\end{cases}
	\end{align*}
Moreover, let $A_t^n := \int_0^t \phi_n(\hat{\mathbb{X}}|_{[0,s]})^2 ds$.

For $t < \tau$, the integral $A_\tau^n = \int_0^\tau \phi_n^2 ds$ converges to 0. Since $Z$ is a strictly positive random variable, for sufficiently large $n$, we have $A_\tau^n < Z$ almost surely. Thus, $\tau_n \ge \tau$ almost surely.

For $t > \tau$, the sequence $\phi_n$ diverges to infinity on this interval $(\tau, t]$. Consequently, $A_t^n \to \infty$ for any $t > \tau$. This ensures that the threshold $Z$ is crossed immediately after $\tau$.

Thus, $\tau_{\phi_n}^r \rightarrow \tau$ almost surely from the right as $n \to \infty$. By dominated convergence theorem and the right continuity of $Y$ we have 
	\begin{equation*}
		\lim _{n \rightarrow \infty} \mathbb{E}\left[Y_{\tau_{\phi_n}^r \wedge T}\right]=\mathbb{E}\left[Y_{\tau \wedge T}\right].
		\end{equation*}
	Therefore 
	\begin{equation}\label{eq:phi>tau}
		\inf _{\phi \in \mathcal{T}} \mathbb{E}\left[Y_{\tau_\phi^r \wedge T}\right] \leq \inf_{\tau \in \mathcal{S}} \mathbb{E}\left[Y_{\tau \wedge T}\right] .
	\end{equation}
For any realization $z$ of the random variable $Z$, we denote
\begin{equation*}
	\tau_z:=\inf \left\{t \geq 0: \int_0^{t \wedge T} \phi(\widehat{\mathbb{X}}|_{[0, s]})^2 d s \geq z\right\}.
\end{equation*}
Since $Z$ is independent of $\X$, by Fubini's theorem we have for any $\phi\in\ct$
\begin{equation}\label{eq:phi<tau}
	\mathbb{E}\left[Y_{\tau_\phi^r \wedge T}\right]=\mathbb{E}\left[\E[Y_{\tau_\phi^r \wedge T} \mid \widehat{\X}]\right]=\int_0^{\infty} \mathbb{E}\left[Y_{\tau_z \wedge T}\right] \mathbb{P}_Z(d z) \geq \inf _{\tau \in \mathcal{S}} \mathbb{E}\left[Y_{\tau \wedge T}\right]
\end{equation}
Combining \eqref{eq:phi>tau} and \eqref{eq:phi<tau} we get the desired result.
	\end{proof}

\begin{proposition}\label{lem:regularized_value}
	Let $Y$ be right continuous in $t$ and $\E[\|Y\|_{\infty}]<\infty$ for any finite interval. Let $S$ be an $(\mathcal{F}_t)$-stopping time and let $F_Z$ denote the cumulative distribution function of $Z$. Then
	\begin{align*}
		\E[{Y}_{\tau_{\phi}^r \wedge S} ] =\E\left[ \int_0^S{Y}_{t}\, d \tilde{F}(t) + {Y}_S (1 - \tilde{F}(S))\right] = \E\left[\int_0^S (1 - \tilde{F}(t))\, d{Y}_t + {Y}_0\right],
	\end{align*}
	where the second integral is implicitly defined by integration by parts and
	\begin{align*}
		\tilde{F}(t) \coloneqq F_Z \left( \int_0^{t} \phi (\widehat{\X}|_{[0,s]})^2 \, ds \right).
	\end{align*}
	In particular, if $Z$ has a density $\varrho$,
	\begin{align*}
		\E[{Y}_{\tau_{\phi}^r \wedge S}] = \E \left[  \int_0^S {Y}_{t} \phi (\widehat{\X}|_{[0,t]})^2 \varrho \left( \int_0^{t} \phi (\widehat{\X}|_{[0,s]})^2 \, ds \right)\, dt + {Y}_S (1 - \tilde{F}(S)) \right].
	\end{align*}
\end{proposition}
\begin{proof}
	Since $Z$ is independent of $\X$, for $\phi\in\ct$ we have 
	\begin{equation}\label{eq:Ft}
		\mathbb{P}\left(\tau_\phi^r \leq t \mid \widehat{\mathbb{X}}\right)=\mathbb{P}\left(\int_0^{t \wedge T} \phi(\widehat{\mathbb{X}}|_{[0, s]})^2 d s \geq Z \mid \widehat{\mathbb{X}}\right)=\widetilde{F}(t),
	\end{equation}
	and
\begin{equation}\label{eq:1-Ft}
	\mathbb{P}\left(\tau_\phi^r=\infty \mid \widehat{\mathbb{X}}\right)=\mathbb{P}\left(\int_0^T \phi(\widehat{\mathbb{X}}|_{[0, s]})^2 d s<Z \mid \widehat{\mathbb{X}}\right)=1-\widetilde{F}(T).
	\end{equation}
Therefore, by \eqref{eq:Ft}-\eqref{eq:1-Ft} we have
\begin{align*}
	\mathbb{E}[Y_{\tau_\phi^r \wedge S} \mid \widehat{\mathbb{X}}]&=\int_0^T E[Y_{t \wedge S}\mid \widehat{\X}] d \widetilde{F}(t)+\E[Y_S\mid\widehat{\X}](1-\widetilde{F}(T))\\
	&=\int_0^S \E[Y_t\mid\widehat{\X}] d \widetilde{F}(t)+\E[Y_S\mid\widehat{\X}](1-\widetilde{F}(S))
\end{align*}
Taking expectation on both sides and applying Fubini's theorem we get the desired result.
\end{proof}
Using the regularization by randomization we will prove that it is enough to use stopping policies that are linear functionals of the signature.
Recall that $\mathcal{T}_{\mathrm{sig}} \subset \mathcal{T}$ is defined as
\begin{align*}
	\mathcal{T}_{\mathrm{sig}} = \left\{ \phi \in \mathcal{T}\, :\, \exists l \in T((\R^{1+d})^*) \text{ such that } \phi(\widehat{\X}|_{[0,t]}) = \langle l,\widehat{\X}^{< \infty}_{0,t} \rangle \ \forall \widehat{\X}|_{[0,t]} \in \Lambda_T \right\}, 
\end{align*}
and we call it the space of \emph{linear signature stopping policies}. 
\begin{definition}\label{def:linear-sig-pol}
	Let $Z$ be as in Definition~\ref{def:rand-stopping-times}, then we define the following notation for \emph{randomized stopping times associated to linear signature stopping policies} 
	\begin{equation*}%\label{eqn:randomized_sig_stopping_time}
		\tau^r_{l} \coloneqq \tau^r_{\phi_l} = \inf \left\{t \geq 0\,:\, \int_0^{t \wedge T} \langle l, \widehat{\X}^{<\infty}_{0,s}\rangle^2 \, ds \geq Z \right\}.
	\end{equation*}
\end{definition}

The main result of this section is the following
\begin{proposition}\label{prop:opt_sig_pol}
	Let $Y$ be right continuous in $t$ and $\E[\|Y\|_{\infty}]<\infty$ for any finite interval. Suppose $Z$ has a continuous density $\varrho$. Then
	\begin{align*}
		\inf_{\phi \in \mathcal{T}} \E[{Y}_{\tau_{\phi}^r \wedge T}] = \inf_{\phi \in \mathcal{T}_{\mathrm{sig}}} \E[{Y}_{\tau_{\phi}^r \wedge T}].
	\end{align*}
	It follows that
	\begin{align*}
		\inf_{\phi \in \mathcal{T}} \E[{Y}_{\tau_{\phi}^r \wedge T}] = \inf_{l \in T((\R^{1 + d})^*)} \E[{Y}_{\tau_l^r \wedge T}].
	\end{align*}
\end{proposition}
	\begin{proof}
Since $\mathcal{T}_{sig} \subset \mathcal{T}$, it suffices to show that $\inf_{\phi \in \mathcal{T}} \E[{Y}_{\tau_{\phi}^r \wedge T}] \geq \inf_{\phi \in \mathcal{T}_{\mathrm{sig}}} \E[{Y}_{\tau_{\phi}^r \wedge T}]$.
		
Let $\phi \in \mathcal{T}$. By Lemma~\ref{lem:sig_dense}, for any $\ep > 0$, there exists a compact set $\mathcal{K} \subset \hat{\Omega}_T^p$ with $\mathbb{P}(\mathcal{K}) \ge 1 - \ep$ and a sequence $\{\phi_n\} \subset \mathcal{T}_{sig}$ such that:
	\begin{equation*}
			\lim_{n \to \infty} \sup_{\hat{\mathbb{X}} \in \mathcal{K}, t \in [0, T]} |\phi_n(\hat{\mathbb{X}}|_{[0,t]}) - \phi(\hat{\mathbb{X}}|_{[0,t]})| = 0.
	\end{equation*}
Define $\tilde{F}_n(t) := F_Z(\int_0^t \phi_n(\hat{\mathbb{X}}|_{[0,s]})^2 ds)$ and $\tilde{F}(t):= F_Z(\int_0^t \phi(\hat{\mathbb{X}}|_{[0,s]})^2 ds)$. We decompose the difference in expected payoffs over the set $A = \{\hat{\mathbb{X}} \in \mathcal{K}\}$ and its complement $A^c$

\emph{Convergence on $A$:}
Since $\phi_n \to \phi$ uniformly on $\mathcal{K}$, the integrated squared policies 
$$ \int_0^t \phi_n(\hat{\mathbb{X}}|_{[0,s]})^2 ds\rightarrow\int_0^t \phi(\hat{\mathbb{X}}|_{[0,s]})^2 ds,$$ uniformly on $\mathcal{K}$. Given that $F_Z$ is uniformly continuous on compact sets, it follows that:
\begin{equation}\label{eq:A}
			\lim_{n \to \infty} \mathbb{E}[|Y_T| |\tilde{F}_n(T) - \tilde{F}(T)|; A] = 0.
\end{equation}

\emph{Bound on $A^c$:}
On the complement set, by the integrability of $Y$ the following is bounded
\begin{equation}\label{eq:A^c}
			\left| \mathbb{E}[Y_T(1-\tilde{F}_n(T)); A^c] - \mathbb{E}[Y_T(1-\tilde{F}(T)); A^c] \right| \le 2 \mathbb{E}[|Y_T|; A^c].
\end{equation}
Recall that $\E[ \lVert Y \rVert_{\infty}] < \infty$ on any bounded interval. Therefore, the right hand side of \eqref{eq:A^c} can be made arbitrarily small by the choice of $\mathcal{K}$ since $\mathbb{P}(A^c) \le \ep$.
		
Combining \eqref{eq:A}-\eqref{eq:A^c} we have $\lim_{n \to \infty} \mathbb{E}[Y_T(1-\tilde{F}_n(T))] = \mathbb{E}[Y_T(1-\tilde{F}(T))]$. By a similar argument, using the continuity of $\varrho$, we also have
\begin{equation*}
\lim_{n \to \infty}	\E \left[  \int_0^T {Y}_{t} \phi_n (\widehat{\X}|_{[0,t]})^2 \varrho \left( \int_0^{t} \phi_n (\widehat{\X}|_{[0,s]})^2 \, ds \right)\, dt\right]=\E \left[  \int_0^T {Y}_{t} \phi(\widehat{\X}|_{[0,t]})^2 \varrho \left( \int_0^{t} \phi (\widehat{\X}|_{[0,s]})^2 \, ds \right)\, dt\right]
\end{equation*}
Therefore, by Proposition~\ref{lem:regularized_value} we complete the proof. 
\end{proof}
We are ready to prove Proposition~\ref{thm:signature-stopping} now.
\begin{proof}[{\bf Proof of Proposition~\ref{thm:signature-stopping}}]
	By Proposition~\ref{prop:opt_cont_pol} and \ref{prop:opt_sig_pol}, it suffices to show that
	\begin{equation*}
		\inf_{l \in T((\R^{1+d})^*) } \E[Y_{\tau_{l}^r \wedge T}] \geq \inf_{l \in T((\R^{1 + d})^*) } \E[Y_{\tau_l \wedge T}].
	\end{equation*}
	For any $l \in T((\R^{1 + d})^*)$ and a chosen random variable $Z$ for the randomize stopping time $\tau_l^r$, the expected payoff can be expressed as:
	\begin{equation*}
		\E[ {Y}_{\tau_{l}^r \wedge T}] =\int_0^{\infty}  \E[{Y}_{\tau_z \wedge T}]\, \PP_Z(dz),
	\end{equation*}
	where $\tau_z \coloneqq \inf \{t \geq 0\,:\, \int_0^{t \wedge T} \langle l, \hat{\mathbb{X}}_{0,s}^{<\infty} \rangle ^2\, ds \geq z \} $. Using the shuffle product property for group like elements, for any fixed $z$, $\tau_z$ is equivalent to a deterministic signature hitting time, more precisely:
	\begin{equation*}
		\tau_z= \inf \left\{t \in [0,T] \,:\, \langle (l \shuffle l)1/z, \hat{\mathbb{X}}_{0,t}^{<\infty} \rangle \geq 1 \right\}.
	\end{equation*}
	Choosing $Z$ to be non-negative random variable we have
	$$
	\tau_z=\inf \left\{t \in [0,T] \,:\, |\langle l_z, \hat{\mathbb{X}}_{0,t}^{<\infty} \rangle| \geq 1 \right\},
	$$
	where $l_z=(l \shuffle l)1/z\in T((\R^{1 + d})^*)$. Since the expected payoff of the randomized time is a weighted average of payoffs from deterministic signature hitting times, it is bounded by their infimum:
	\begin{equation*}
		\E[ {Y}_{\tau_{l}^r \wedge T}]  \geq \inf_{p \in T((\R^d)^*) } \E[Y_{\tau_p \wedge T}],
	\end{equation*}
	which gives us the desired result.
\end{proof}
\end{appendix}

\end{document}